\documentclass[11pt,letterpaper]{article}
\usepackage{hdx}
\usepackage{tikz}
\usetikzlibrary{matrix,decorations.pathreplacing, calc, positioning}

\def\authornotes{1pt}
\def\displayauthors{1pt}

\ifdim\authornotes=1pt
    \newcommand{\siqi}[1]{\footnote{\color{MidnightBlue}Siqi: {#1}}}
    \newcommand{\sid}[1]{\footnote{\color{RedOrange}Sidhanth: {#1}}}
    \newcommand{\tselil}[1]{\footnote{\color{ForestGreen}Tselil: {#1}}}
    \newcommand{\eyang}[1]{\footnote{\color{Purple}Elizabeth: {#1}}}

    \newcommand{\tnote}[1]{{\color{ForestGreen}[Tselil: #1]}}
    \newcommand{\enote}[1]{{\color{Purple}[Elizabeth: #1]}}
    \newcommand{\sidnote}[1]{{\color{RedOrange}[Sidhanth: #1]}}
    \newcommand{\siqnote}[1]{{\color{MidnightBlue}[Siqi: #1]}}

\newcommand{\todo}[1]{{\color{Red}[TODO: #1]}}
\else
    \newcommand{\siqi}[1]{}
    \newcommand{\sid}[1]{}
    \newcommand{\tselil}[1]{}
    \newcommand{\eyang}[1]{}

    \newcommand{\tnote}[1]{}
    \newcommand{\enote}[1]{}
    \newcommand{\sidnote}[1]{}
    \newcommand{\siqnote}[1]{}

\newcommand{\todo}[1]{}
\fi

\newcommand{\grg}{\mathsf{Geo}}

\newcommand{\ER}{\mathsf{G}}

\newcommand{\dtv}[2]{\mathrm{d}_{\mathrm{TV}}\left(#1,#2\right)}
\newcommand{\GG}{\mathsf{gg}}
\newcommand{\Unif}{\rho}
\newcommand{\dkl}{\mathrm{D}}

\newcommand{\KL}[1]{\dkl(#1\|\Unif)}

\newcommand{\Par}{\mathrm{Par}}

\newcommand{\Binom}{\mathsf{Binom}}

\newcommand{\iprod}[1]{\left\langle #1 \right\rangle}
\newcommand{\Iprod}[1]{\langle #1 \rangle}

\newcommand{\scap}{\mathrm{cap}}
\newcommand{\mcap}{\ul{\smash{\mathrm{cap}}}}

\newcommand{\Shell}{\mathrm{shell}}

\usepackage{scalerel}

\newcommand{\cW}{\calW}
\newcommand{\sd}{\succeq_{\mathrm{st}}}
\newcommand{\sdl}{\preceq_{\mathrm{st}}}

\newcommand{\nA}{\underline{A}_{\bG}}
\newcommand{\nQ}{\underline{Q}}
\newcommand{\eD}{{D_{\bkappa}}}
\newcommand{\eDdet}{{D_{\kappa}}}

\newcommand{\mX}{m}
\newcommand{\dW}{d_W}
\newcommand{\exc}{\mathrm{exc}}
\newcommand{\gw}{G(W)}
\newcommand{\vw}{V(W)}
\newcommand{\ew}{E(W)}
\newcommand{\tw}{T(W)}
\newcommand{\gow}{G_1(W)}

\newcommand{\eow}{E_1(W)}
\newcommand{\gtw}{G_2(W)}
\newcommand{\vtw}{V_2(W)}
\newcommand{\etw}{E_2(W)}
\newcommand{\gjw}{G_J(W)}
\newcommand{\vjw}{J(W)}
\newcommand{\ejw}{E_J(W)}

\newcommand{\wo}{X}
\newcommand{\sw}{S_J(W)}
\newcommand{\dw}{D_J(W)}
\newcommand{\tjw}{T_J(W)}
\newcommand{\WS}{\mathsf{Struc}_{\ell}}
\newcommand{\uw}{U}
\newcommand{\vuw}{V(\uw)}
\newcommand{\suw}{s(\uw)}

\newcommand{\Beta}[1]{{\mathsf{D_{ip}}({#1})}}
\newcommand{\BetaPDF}{\psi}

\newcommand{\porb}[2]{q_{#1}(#2)}

\newcommand{\rhoflat}{\rho_{\mathrm{1D}}}

\newcommand{\rr}{\mathsf{RR}}
\newcommand{\defeq}{\coloneqq}
\newcommand{\xbound}{\eta}
\newcommand{\ptau}{{\nu}}

\allowdisplaybreaks

\begin{document}

\title{Local and global expansion in random geometric graphs}
\ifdim\displayauthors=1pt
\author{Siqi Liu\thanks{UC Berkeley. \texttt{sliu18@berkeley.edu}. Supported in part by the Berkeley Haas Blockchain Initiative and a donation from the Ethereum Foundation.} \and Sidhanth Mohanty\thanks{UC Berkeley. \texttt{sidhanthm@cs.berkeley.edu}. Supported by a Google PhD Fellowship.} \and Tselil Schramm\thanks{Stanford University.  \texttt{tselil@stanford.edu}. Supported by an NSF CAREER award \# 2143246.} \and Elizabeth Yang\thanks{UC Berkeley. \texttt{elizabeth\_yang@berkeley.edu}.  Supported by the NSF GRFP under Grant No. DGE 1752814.}}
\else
\fi
\date{\today}
\maketitle

\begin{abstract}
Consider a random geometric 2-dimensional simplicial complex $X$ sampled as follows: first, sample $n$ vectors $\bu_1,\ldots,\bu_n$ uniformly at random on $\bbS^{d-1}$; then, for each triple $i,j,k \in [n]$, add $\{i,j,k\}$ and all of its subsets to $X$ if and only if $\iprod{\bu_i,\bu_j} \ge \tau, \iprod{\bu_i,\bu_k} \ge \tau$, and $\iprod{\bu_j, \bu_k} \ge \tau$.
We prove that for every $\eps > 0$, there exists a choice of $d = \Theta(\log n)$ and $\tau = \tau(\eps,d)$ so that with high probability, $X$ is a high-dimensional expander of average degree $n^\eps$ in which each $1$-link has spectral gap bounded away from $\frac{1}{2}$. 

To our knowledge, this is the first demonstration of a natural distribution over $2$-dimensional expanders of arbitrarily small polynomial average degree and spectral link expansion better than $\frac{1}{2}$.
All previously known constructions are algebraic.
This distribution also furnishes an example of simplicial complexes for which the trickle-down theorem is nearly tight.

En route, we prove general bounds on the spectral expansion of random induced subgraphs of arbitrary vertex transitive graphs, which may be of independent interest.
For example, one consequence is an almost-sharp bound on the second eigenvalue of random $n$-vertex geometric graphs on $\bbS^{d-1}$, which was previously unknown for most $n,d$ pairs.

\end{abstract}

\setcounter{tocdepth}{2}
\setcounter{page}{0}
\thispagestyle{empty}
\newpage
\tableofcontents
\thispagestyle{empty}
\thispagestyle{empty}
\setcounter{page}{0}
\newpage

\section{Introduction}
A graph $G$ is called a spectral $\lambda$-expander if the second eigenvalue of its normalized adjacency matrix, $\lambda_2(G)$, is at most $\lambda$.
More generally, a sequence of graphs of increasing size $(G_n)_{n\in\N}$ is said to be a family of (1-dimensional) spectral $\lambda$-expanders if  $\lambda_2(G_n) \le \lambda < 1$ as $n \to \infty$, and importantly, this implies that no vertex cut of $G_n$ has sub-constant sparsity.
Expanders are an indispensible tool in theoretical computer science and mathematics, underlying advances in pseudorandomness, coding theory, routing algorithms, and more (e.g. \cite{INW94,SS96,Pin73}, see also the survey \cite{HLW06}); similarly, the phenomenon of expansion has enabled the analysis of approximation algorithms, probabilistically checkable proofs, embeddability of metric spaces (e.g. \cite{ABS15,Dinur07,LLR95}), as well as numerous results in number theory, group theory, and other areas of pure mathematics (see e.g. the survey \cite{Lubotzky12}).

Sparse expander graphs were first shown to exist via the probabilistic method \cite{KM93,Pin73}.
In fact, even the most extreme version of expansion is ubiquitous:
The best possible spectral expansion for a $d$-regular graph is the Ramanujan bound $\lambda = \frac{2\sqrt{d-1}}{d}$, and this bound is achieved (up to an additive $o_n(1)$) by a random $d$-regular graph with high probability \cite{Alon86,Nilli91,Friedman93,Friedman08}.\footnote{Though even in \cite{Alon86} it is conjectured that random graphs achieve the Ramanujan bound, the first graphs proven achieve this bound were explicit algebraic constructions \cite{LPS88,Margulis88}.}
Even an \erdos--\renyi graph $\ER(n,p)$ forms a ``decorated expander'' for any $p > \frac{1}{n}$, which is to say the graph is an expander when one omits small isolated connected components and dangling trees \cite{FR08,BKW14}.
The sampling of these graphs may be derandomized \cite{BL06,MOP20}, so that in addition to several known explicit constructions (see \cite{Margulis73,GG81,LPS88,Margulis88}, and others) we have a wealth of algorithmic constructions of expander graphs to use in applications.

{\em Higher-dimensional spectral expansion} is a generalization of expansion to simplicial complexes. 
For simplicity, we will for the moment limit ourselves to $2$-dimensional spectral expansion, which can be stated easily in terms of simple graphs.
A graph $G = (V,E)$ is said to be a $2$-dimensional spectral expander if $G$ itself is an expander, and further for every vertex $v \in V$, the induced graph on $v$'s neighbors $G[N(v)]$ (called the ``link'' of $v$) is a $\lambda$-expander for $\lambda < \frac{1}{2}$.
The significance of $\frac{1}{2}$ is that when the local expansion is $\lambda < \frac{1}{2}$, this is enough to trigger a ``trickling down'' phenomenon that ensures that $G$ in its entirety is an expander! 
Hence, higher-dimensional expanders have the remarkable property that global expansion is witnessed by local expansion.
This local-to-global phenomenon has led to a number of recent breakthroughs in theoretical computer science: objects inspired by high-dimensional expanders are crucial in explicit constructions of locally testable codes \cite{DELLM22,LH22} and quantum LDPC codes \cite{PK22,LZ22,DHLV22}, and the local-to-global phenomenon has been essential in analyzing Markov chains for a wide variety of sampling problems \cite{ALOV19}.

The simplest example of a $2$-dimensional expander is the complete complex, based on the complete graph $K_n$.
Sparse examples are known as well (e.g. \cite{CSZ03,Li04,LSV05a,LSV05b,KKL14}), though at first their existence may seem remarkable: such graphs must be globally sparse, and yet the $O(1)$-sized local neighborhood of every vertex must be densely connected to ensure sufficient expansion.
This is a delicate balance, and indeed given the state of our knowledge today the phenomenon of sparse high-dimensional expansion seems ``rare,'' in sharp contrast with the ubiquity of 1-dimensional expansion.
Only a few sparse constructions are currently known, and many of these constructions are algebraic, inheriting their expansion properties from the groups used to define them (as discussed further in \pref{sec:known-constructions}).

A prominent open problem in the area is to identify natural distributions over sparse higher-dimensional expanders \cite{Lubotzky18,LinialSimons,LubotzkySTOC};
 this would be highly beneficial, both for a deeper mathematical understanding and for applications in algorithms and complexity.
The simplest distributions immediately fail: 
random $d$-regular graphs are locally treelike, and so with high probability $G[N(v)]$ will be an independent set (with $\lambda = 1$) for most $v \in V$.
The same is true for an \erdos-\renyi graph $\ER(n,p)$ when $p \ll \frac{1}{\sqrt{n}}$.
Though a number of distributions have been shown to have some higher-dimensional expansion properties \cite{LM06,FGLNP12,Conlon19,CTZ20,LMY20,Gol21}, they all fall short in some sense: either they are quite dense (degree $\Omega(\sqrt{n})$) or fail to satisfy the spectral conditon $\lambda < \frac{1}{2}$.
In this work, our primary question is the following:
\begin{center}
\emph{Are there natural, high-entropy distributions over 2-dimensional expanders of average degree $\ll \sqrt{n}$?}
\end{center}

We answer this question in the affirmative.
We prove that for any $\eps > 0$ and any large enough $n \in \N$, there exists a choice of $d \in \N$ such that a random $n$-vertex geometric graph on $\bbS^{d-1}$ with average degree $n^{\eps}$ is a $2$-dimensional expander with high probability.

\subsection{Our results}

In order to state our results, we first give some formal definitions.

\begin{definition}[Simplicial complex]
A {\em $k$-dimensional simplicial complex} $X$ is a downward-closed collection of subsets of size at most $k+1$ over some ground set $X_0$, with a downward-closed weight function $w$.\footnote{
	Recall that $X$ is called downward-closed if $S \subseteq T$ and $T \in X$ imply $S \in X$, and $w$ is called downward-closed if weights are assigned to maximal faces, and for each non-maximal $S\in X$, we recursively define $w(S) = \sum_{x \in X_0} w(S \cup \{x\})$.}
Any $S \in X$ is called a {\em $(|S|-1)$-face}, and the restriction of $X$ to sets of size at most $\ell+1 \le k$ is called the {\em $\ell$-skeleton} of $X$.
The {\em degree} of $v \in X_0$ is the number of top-level faces that contain it.
\end{definition}

\noindent For example, the set of all cliques of size at most $k+1$ in a graph $G$, where the weight of each clique is proportional to the number of $(k+1)$-cliques it occurs in, defines a $k$-dimensional simplicial complex.
\begin{definition}[Link]
Let $X$ be a simplicial complex.
For any face $S$, the {\em link} of $S$ in $X$ is the simplicial complex $X_S$ with weight function $w_S$, consisting of all sets in $X$ which contain $S$, minus $S$:
\[
    X_S = \{T\setminus S\mid T\supseteq S,\, T \in X\},
    \qquad
    w_S(T\setminus S) = w(T) \quad \forall T\in X_S
\]
\end{definition}

\noindent For example, in the simplicial complex whose highest order faces are the triangles in a graph $G$, the link of a vertex $v$ is the induced graph on the neighbors of $v$ with its isolated vertices removed.

We are interested in simplicial complexes where the links expand enough to trigger a ``local-to-global phenomenon'' via the trickling-down theorem, stated below in the $2$-dimensional case.\footnote{The trickling-down theorem also generalizes to higher dimensions: sufficiently strong local spectral expansion of only the highest-order links implies global spectral expansion.}
\begin{theorem}[Trickling-down theorem {\cite{Opp18}}] \label{thm:opp-trickling}
	Let $X$ be a $2$-dimensional simplicial complex. 
If its 1-skeleton is connected, and the second eigenvalue of every link's random walk matrix is at most $\lambda$, then the second absolute eigenvalue of the random walk matrix of the 1-skeleton of $X$ is at most $\frac{\lambda}{1 - \lambda}$.
\end{theorem}
\noindent This theorem explains the significance of $\lambda = \frac{1}{2}$, since when $\lambda < \frac{1}{2}$, local expansion ``trickles down'' to imply global expansion.
We will show that random geometric graphs, in a carefully-chosen parameter regime, have sufficient link expansion.

\begin{definition}[Random geometric graph]
A {\em random geometric graph} $\bG \sim \grg_{d}(n,p)$ is sampled as follows: for each $i\in[n]$, a vector $\bu_i$ is drawn independently from the uniform distribution over $\bbS^{d-1}$ and identified with vertex $i$.
Then, each edge $\{i,j\}$ is included if and only if $\angles*{\bu_i,\bu_j}\ge\tau$ where $\tau = \tau(p,d)$ is chosen so that $\Pr_{\grg_d(n,p)}[(i,j) \in \bG ] = p$.
\end{definition}
\begin{definition}[Random geometric complex]
The {\em random geometric $k$-complex} $\grg_d^{(k)}(n,p)$ is the distribution defined by sampling $\bG \sim \grg_d(n,p)$ and taking the downward-closure of the complex whose $k$-faces are the cliques of size $(k+1)$ in $\bG$.
\end{definition}
Our main result proves that there are conditions under which random geometric 2-complexes of degree $n^\eps$ are high-dimensional expanders enjoying the trickling-down phenomenon:
\begin{theorem}\label{thm:hdx}
    For every $0 < \eps < 1$, there exist constants $C_\eps$ and $\delta = \exp(-O(1/\eps))$ such that when $\bH\sim\grg_d^{(2)}(n,n^{-1+\eps})$ for  $d=C_{\eps}\log n$, with high probability every vertex link of $\bH$ is a $(\frac{1}{2}-\delta)$-expander, and hence its $1$-skeleton is a $\parens{1-\tfrac{4\delta}{1+2\delta}}$-expander.
\end{theorem}

Along the way, we also analyze the spectrum of $\bG \sim \grg_d(n,p)$ directly and obtain sharper control of its second eigenvalue in a more general setting, giving bounds on the spectral norm of random geometric graphs in the full high-dimensional $(d \to_n \infty)$ regime. 
To our knowledge, previous results in this vein are only for $d \sim n^{1/k}$ for fixed integers $k$ \cite{Karoui10,CS13,DV13,Bordenave13,FM19,LY22}.
\begin{theorem}\torestate{\label{thm:sphere}
    Let $\bG\sim\grg_d(n,p)$ and $\tau \coloneqq \tau(p,d)$.  
Then with high probability $\bG$ is a $\mu$-expander, where
    \[
        \mu \coloneqq (1+o(1))\cdot\max\braces*{(1+o_{d\tau^2}\parens*{1})\cdot \tau, \frac{\log^4 n}{\sqrt{pn}} },
    \]
    where $o_{d\tau^2}(1)$ denotes a function that goes to $0$ as $d \cdot \tau(p,d)^2 \to \infty$.}
\end{theorem}

In \pref{sec:tight} we show that an eigenvalue close to $\tau$ is achieved (for some $p,d$), so \pref{thm:sphere} is close to sharp.
Since in \pref{thm:hdx} we show that the vertex links of $\bG$ have eigenvalue $\lambda \le \frac{\tau}{1+\tau}$, this implies that the trickling-down theorem is tight

\begin{proposition}[Trickling-down theorem is tight]
\torestate{
\label{prop:trickle-tight}
For each $\lambda \in (0,\tfrac{1}{2}]$ and $\eta > 0$ there exists a 2-dimensional expander in which all vertex link eigenvalues are at most $\lambda$ for which the 1-skeleton is connected with eigenvalue at least $\frac{\lambda}{1-\lambda} - \eta$.
}
\end{proposition}

\subsection{Spectra of random restrictions} \label{sec:random-restrictions}
\pref{thm:sphere} (and morally \pref{thm:hdx}) is a consequence of a more general theorem that we prove concerning the spectral properties of random restrictions of graphs.
We describe this result here, both because it may be of independent interest, and because it may help demystify \pref{thm:hdx}.

Random restriction is a procedure for approximating a large graph $X$ by a smaller graph $\bG$: one selects a random subset of vertices $\bS$, and then takes $\bG$ to be the induced graph $X[\bS]$.
The random restriction $\bG$ is now a smaller (and often sparser) approximation to $X$;
this idea has been useful in a number of contexts in theoretical computer science (e.g. \cite{GGR98,AFKK03,BHHS11,LRS15,HKPRSS17}).
The core question is: to what extent do random restrictions actually inherit properties of the original graph? 
We will show that if random walks on $X$ mix rapidly enough, then random restrictions inherit the {\em spectral} properties of the original graph.

To see the relevance of this result in our context, notice that a random geometric graph on the sphere is a random restriction of the (infinite) graph with vertex set $\bbS^{d-1}$ and edge set $\{(u,v) \mid \iprod{u,v} \ge \tau\}$.
\pref{thm:hdx} is then a consequence of the fact that the sphere is {\em itself} a 2-dimensional expander.

We state the theorem precisely below.

\begin{definition}[Random restriction]\label{def:rr}
Suppose $X$ is a (possibly infinite) graph, and that the simple random walk on $X$ has unique stationary distribution $\rho$.
We define an {\em $n$-vertex random restriction} of $X$ to be a graph $\bG \sim \rr_n(X)$ sampled by sampling $n$ vertices independently according to $\rho$, $\bS \sim \rho^{\otimes n}$, then taking $\bG = X[\bS]$ to be the graph induced on those vertices.
\end{definition}

We show that if the average degree in $\bG$ is not too small, $\lambda_2(\bG)$ reflects the rapid mixing of the random walk on $X$.
\begin{theorem}
\torestate{
\label{thm:random-restriction}
Let $X$ be a (possibly infinite) vertex-transitive graph on which the associated simple random walk has a unique stationary distribution $\Unif$, and let $p = \Pr_{\bG \sim \rr_n(X)}[(i,j) \in E(\bG)]$ be the marginal edge probability of a $n$-vertex random restriction of $X$.
Suppose there exist $C \ge 1$ and $\lambda \in [(np)^{-1/2},1]$ such that for any $k \in \N$, $k$-step walks on $X$ satisfy the following mixing property: for any distribution $\alpha$ over $V(X)$, 
\[
\dtv{X^k \alpha}{\rho} \le C \cdot \lambda^k,
\]
where $X^k$ denotes the $k$-step random walk operator on $X$, and furthermore suppose $pn \gg C^6 \log^4 n$.
Then for any constant $\gamma > 0$,
\[
\Pr_{\bG \sim \rr_n(X)}\left[ \left|\lambda_2(\widehat{A}_{\bG}) \right|, \left|\lambda_n(\widehat{A}_{\bG}) \right| \le (1+o(1)) \cdot \max\left(\lambda, \frac{\log^4 n}{\sqrt{pn}}\right)\right] \ge 1 - n^{-\gamma},
\]
where $\widehat{A}_{\bG}$ is the (normalized) adjacency matrix of $\bG$.}
\end{theorem}

\begin{remark}\label{rem:rr}
It is likely that some of the conditions of \pref{thm:random-restriction} could be weakened.
The decay of total variation could plausibly be replaced with a (much weaker) assumption about the spectral gap of $X$; this would not impact our results for $\bbS^{d-1}$, but may be useful in other applications. 
Transitivity is assumed mostly to make the proof of \pref{thm:random-restriction} go through at this level of generality; to prove \pref{thm:hdx} we re-prove a version of \pref{thm:random-restriction} for the specific non-transitive case where $X$ is a link of a vector in the sphere (a spherical cap).
\end{remark}

\subsection{Related work}
We give a brief overview of related work.
While so far we have focused on a spectral notion of high-dimensional expanders (HDX), there are two additional notions: coboundary and cosystolic expansion.
These are meant to generalize the Cheeger constant, a cut-based measure of graph expansion. 

\smallskip
\noindent{\bf Distributions over high-dimensional expanders.}
The existence of natural distributions over sparse HDXs has been a question of interest since sparse HDX were first shown to exist (and this was highlighted as an important open problem in e.g. \cite{Lubotzky18,LubotzkySTOC}).

The early work of Linial and Meshulam \cite{LM06} considered the distribution over $2$-dimensional complexes in which all edges $\binom{[n]}{2}$ are included, and each triangle is included independently with probability $p$; they identified the phase transition at $p$ for coboundary connectivity for this distribution (see also the follow-ups \cite{BHK11,MW09,LP16}).
This distribution has the drawback that the 1-skeleton of these complexes is $K_n$, and so the resulting complex is far from sparse.

In \cite{FGLNP12}, the authors show that a union of $d$ random partitions of $[n]$ into sets of size $k+1$ with high probability produces a {\em geometric} expander \cite{G10}, which is a notion of expansion which measures how much the faces must intersect when the complex is embedded into $\R^k$.
The resulting complexes have disconnected links when $d \ll \sqrt{n}$, and so they fail to be spectral HDXs.

The work of \cite{LMY20} introduces a distribution over spectral expanders with expansion exactly $\frac{1}{2}$ by taking a tensor product of a random graph and a HDX; the authors show that down-up walks on these expanders mix rapidly, and
\cite{Gol21} introduces a reweighing of these complexes which yields improved mixing time bounds.
However, the links in these complexes fail to satisfy $\lambda < \frac{1}{2}$, and so fall outside of the range of the trickling-down theorem.
The same drawback applies to \cite{Conlon19,CTZ20}: they show that up-down walks mix on random polylogarithmic-degree graphs given by subsampling a random set of generators of a Cayley graph. 
However, these graphs do not satisfy the conditions of the trickling-down theorem.

\medskip
\noindent {\bf Explicit constructions.}
\label{sec:known-constructions}
One of the first constructions of sparse high-dimensional spectral expanders was the Ramanujan complex of \cite{CSZ03,Li04,LSV05a,LSV05b}, which generalize the Ramanujan expander graphs of \cite{LPS88}. 
Not only are these spectral expanders, but \cite{KKL14,EK16} also show that they are co-systolic expanders. 
These Ramanujan complexes are algebraic by nature, constructed from the Cayley graphs of $\text{PSL}_d(\mathbb{F}_q)$. 
Other algebraic constructions include that of \cite{KO18}; the authors analyze the expansion properties of coset complexes for various matrix groups.
They achieve sparse spectral expanders, with local expansion arbitrarily close to 0. 
More recently, \cite{OP22} extend the coset complex construction to the more general family of Chevalley groups. 

A few combinatorial constructions for HDX are also known.
\cite{CLP20} prove that objects called $(a, b)$-expanders are two-dimensional spectral expanders; they give a graph-product-inspired construction of a family of such expanders, and show that other known complexes \cite{CSZ03,Li04,LSV05a,LSV05b,KO18} are also $(a,b)$-expanders.
Their work is extended by \cite{FI20} to higher dimensions. 

\medskip
\noindent{\bf Applications of HDX.}
The local-to-global phenomenon in HDX has already been useful in many settings.
\cite{DK17} use high-dimensional spectral expanders to construct ``agreement expanders,'' whose links give rise to local agreement tests: given ``shards'' of a function that pass a large fraction of the local agreement tests, the authors can conclude the presence of a ``global'' function $g$ that stitches the shards together. 
In coding theory, the locally testable codes of \cite{DELLM22} and quantum LDPCs of \cite{PK22, LZ22, DHLV22} utilize a common simplicial-complex-like structure called the square Cayley complex, whose local-to-global properties are essential in the analysis of these codes.

The local-to-global phenomenon also implies that ``down-up'' walks on the associated simplicial complex mix (as made formal in \cite{AL20}). 
A $k$-down-up walk is supported $k$-faces of the simplicial complex, and transitions occur by dropping down into a random $(k - 1)$-face, then transitioning up to a random $k$-face (one can also define the ``up-down'' walk analagously).
This local-to-global analysis has recently been influential in the study of mixing times of Markov chains.
Several well-studied Markov chains can be recast as the $k$-down-up random walk of a carefully designed simplicial complex. 
One notable example is the matroid basis exchange walk, which is an algorithm for sampling independent sets of a matroid (e.g. spanning trees in the graphical matroid). 
\cite{ALOV19} were able to obtain an improved mixing time bound for the basis exchange walk--a significant breakthrough that, due to the local-to-global property, was achieved through the analysis of simple, ``local'' view of the matroids. 

\medskip
\noindent{\bf Random geometric graphs and random kernel matrices.}
Random restrictions of metric spaces such as $\bbS^{d-1}$ and $[-1,1]^d$ are well-studied in the fixed-dimensional regime, where $d = O(1)$ and $n \to \infty$ (see the survey of Penrose \cite{Pen03}).
In our work we are interested in the high-dimensional setting, where $d\to \infty$ with $n$.
The high-dimensional setting was first studied only recently, initiated by \cite{DGLU11,BDER16}, and many mysteries remain in this young area of study.

Our \pref{thm:sphere} is related to the study of {\em kernel random matrices}: random $n \times n$ matrices whose $(i,j)$-th entry is given by $f_d\left(\iprod{\bu_i,\bu_j}\right)$, for $f_d:\R \to \R$ and $\bu_1,\ldots,\bu_n$ sampled independently from some distribution over $\R^d$.
The special case of $\bu_i \sim \mathrm{Unif}(\bbS^{d-1})$ and $f_d(x) = \Ind[x \ge \tau(p,d)]$ yields the adjacency matrix of $\grg_d(n,p)$.
A line of work initiated by \cite{KG00} studies the spectrum of kernel random matrices \cite{Karoui10,CS13,DV13,Bordenave13,FM19}, and the most recent work \cite{LY22} characterizes the limiting empirical spectral distribution when $d =\Theta(n^{1/k})$ for $k$ a fixed constant and $f$ can be ``reasonably'' approximated by polynomials (in a sense that is flexible enough to capture the indicator $f_d(x) = \Ind[x \ge \tau(p,d)]$).
In comparison with our results, they characterize the entire empirical spectral distribution, but we do not need to restrict $d \sim n^{1/k}$ for integer $k$, which is crucial for our applications.

\subsection{Discussion and open questions}\label{sec:discussion}

\paragraph{Sparser high-dimensional expanders from random restrictions?}
As hinted in \pref{sec:random-restrictions} the random geometric complex fits in the broader framework of \emph{random restrictions} of simplicial complexes: starting with a dense high-dimensional expander $X$, we sample a subset of vertices $\bS$ of $X$ to produce the sparser induced complex $X[\bS]$.

We have shown in \pref{thm:random-restriction} that $X[\bS]$ (to some extent) inherits the spectral properties of $X$ itself, and we've leveraged this to show that for any polynomial average degree, one can produce a 2-dimensional expander by taking a random restriction of $X$ the sphere in a particular dimension and with a particular connectivity distance.
We hope that \pref{thm:random-restriction} (or a strengthening thereof, see \pref{rem:rr}) might help us identify additional natural distributions over sparser and/or higher-dimensional complexes.
More specifically,
\begin{displayquote}
    {\em Is there a simplicial complex $X$ whose random restrictions yield high-dimensional expanders whose links have eigenvalue $< \frac{1}{2}$, of sub-polynomial or polylogarithmic degree}\footnote{We note that constant average degree would likely require additional work; this is not just because of the polylogarithmic factors appearing in the statement of \pref{thm:random-restriction}, but because in a random restriction the degree distribution of each vertex is $\Binom(n,p)$ and so when $p = \Theta(1/n)$ one will have isolated vertices; this is the same as the phenomenon wherein \erdos-\renyi graphs of degree $O(1)$ are not expanders until one restricts to the giant component.}?
\end{displayquote}
As a starting point, we remark that geometric graphs on the unit sphere work because the corresponding $X$ itself has link expansion better than $\frac{1}{2}$, witnessing that $\bbS^{d-1}$ itself is an expander.
Some simpler-to-analyze metric spaces do not have this property; for example:
\begin{description}
    \item[Shortest path metric in a graph.]
Starting with a connected, locally-treelike $d$-regular graph $G$, consider the geometric graph given by connecting pairs of vertices at distance $\le 2$ in $G$.
The triangle complex on the resulting graph has links which are connected, and further the 1-skeleton is an expander if $G$ is an expander.\footnote{Technically, we require this of a reweighting of $G$ where each edge is weighted according to the number of triangles it participates in; concentration phenomena ensure that the expansion of $G$ and this weighted graph are similar.}
However, the links cannot identify whether $G$ is an expander, and so link expansion cannot be better than $\frac{1}{2}$. 
To see why, consider a first case where $G$ is a random $d$-regular graph (expanding), and a second case where $G$ consists of two random $d$-regular graphs connected by a bridge (non-expanding); because $G$ is locally treelike in both cases, the links in these two cases will be identical. 

    \item[The $d$-dimensional torus with $\ell_\infty$ metric.]
Consider the geometric graph $X$ on the $d$-dimensional torus $[-R,R]^d$ with ``wraparound'' (so that $-R$ is identified with $R$), in which we connect $u,v$ if $\norm*{u-v}_{\infty} \le \tau$. 
This space is simple to analyze because of its product structure; the expansion is dictated by the ratio of $\tau$ to the side length $R$, worsening as $R$ grows relative to $\tau$.\footnote{This can be seen by analogy to the spectrum of a $d$-tensor-power of the discrete $R\tau^{-1}$-cycle.}
Each link is simply the box $[-\tau,\tau]^d$ with the same $\ell_\infty$-edge condition, regardless of the value of $R$.
Since $R$ dictates the global expansion, the link expansion cannot be better than $\frac{1}{2}$.
\end{description}
By way of contrast, it is not possible to plant the links of the geometric graph on $\bbS^{d-1}$ in a nonexpanding graph: for instance, it is possible to determine the radius $R$ of a sphere of unknown scale given only a link in its geometric graph.

\medskip

We also remark that \pref{thm:random-restriction} could be used to obtain expanders of dimension $k > 2$; indeed, it seems that this is within reach even using $\grg_d^{(k)}(n,p)$.
A direct approach, in the case of the sphere, is to perform the conditioning from \pref{sec:link-eigs} not only for spherical caps, but for intersections of $k-1$ spherical caps as well; perhaps there is a more elegant alternative approach?

\paragraph{How faithfully do random geometric graphs discretize continuous manifolds?}
One interpretation of \pref{thm:sphere} is that the random geometric graph $\grg_d(n,p)$ offers a good approximation (in spectral norm) for the corresponding metric on $\bbS^{d-1}$ when $pn$ is large enough relative to $d$.  
A natural question is to extend this to other properties of $\bbS^{d-1}$; for example, do random geometric graphs offer a good approximation on the rest of the spectrum?
Numerical experiments suggest the following (informal) conjecture.
\begin{conjecture}
    For $\bG\sim\grg_d(n,p)$, the spectrum of the normalized adjacency matrix $A_{\bG}$ breaks into a ``bulk'' portion and an ``outlier'' portion where every bulk eigenvalue is at most $O(1/\sqrt{pn})$ in magnitude, and every outlier eigenvalue is ``close'' to an eigenvalue of the graph on the sphere with an edge between every $u,v$ with $\angles*{u,v}\ge\tau(p,d)$.
\end{conjecture}
A proof of the above conjecture, and an investigation of whether an analogous phenomenon holds on general manifolds, would be very interesting.

\paragraph{Spectral algorithms for random geometric graphs.}

Here, we have given some of the first analyses of the spectral radius of random geometric graphs on the sphere.
One appeal of random geometric graphs on the sphere, or in Gaussian space, is that they offer a more natural model for networks arising from data than, e.g., \erdos-\renyi graphs. 
The idea is that in modern networks, we often think of each node as being representable by a latent feature vector, with nearby nodes having similar features.
Hence, geometric graphs are promising as an alternative testbed for rigorous analysis of algorithms. 
Yet currently, they have not been studied much in such a context, in part because of the absence of tools for their analysis.

A natural question is whether one could build on our work to analyze spectral clustering algorithms in ``random geometric block model'' graphs.
\begin{question}\label{question:block-model}
Suppose $\bG \sim \rr_n(\frac{1}{2} \cN(0,\Sigma_1) + \frac{1}{2}\cN(\mu,\Sigma_2))$; that is, $n$ points $\bu_1,\ldots,\bu_n$ are sampled from the uniform mixture over the $d$-dimensional Gaussian distributions $\cN(0,\Sigma_1)$ and $\cN(\mu,\Sigma_2)$, then $(i,j) \in E(\bG)$ if and only if $\|\bu_i - \bu_j\| \le \eps$.
Does spectral clustering recover the component membership of the datapoints?
\end{question}
This question is a more accurate representation of clustering problems arising from real data than, say, the question of applying spectral clustering to recover cluster memberships in the stochastic block model; it would be interesting to understand the conditions (on $n$, $d$, $\delta$, $\mu$, and $\Sigma_1,\Sigma_2$) which guarantee that spectral clustering succeeds.

\subsection{Overview of the proof}
We now explain how we prove our main theorem, \pref{thm:hdx}, which states that for a complex sampled from $\bH\sim\grg_d^{(2)}(n,p)$ for $p = n^{-1+\eps}$ with $0 < \eps < 1$ and $d = C_\eps \log n$, with high probability every link of $\bH$ is a $\left(\frac{1}{2}-\delta\right)$-expander for some $\delta = \exp(-O(1/\eps))$, and its $1$-skeleton is a $\parens*{1-\frac{4\delta}{1+2\delta}}$-expander.
By the trickling-down theorem, it suffices for us to prove:
\begin{enumerate}
    \item \label{item:links-expanding} All $n$ vertices' corresponding links in $\bH$ are $\parens*{\frac{1}{2}-\delta}$-expanders with high probability.
    \item \label{item:connectedness} The $1$-skeleton of $\bH$ is connected with high probability.
\end{enumerate}

To show \pref{item:connectedness}, it is enough to show that some reweighting of the $1$-skeleton expands; \pref{item:links-expanding} implies that every edge $(i,j)$ must participate in at least one triangle (otherwise the link would contain isolated vertices), so the unweighted $1$-skeleton is just the adjacency matrix of an unweighted graph from $\grg_d(n,p)$. 
En route to proving \pref{item:links-expanding} we'll prove that unweighted random geometric graphs expand, by this logic yielding \pref{item:connectedness} a consequence.

\paragraph{Analyzing link expansion.}
We establish \pref{item:links-expanding} by showing that that each of the $n$ links is a $\parens*{\frac{1}{2}-\delta}$-expander with probability $1-o(1/n)$, then applying a union bound.
We can think of sampling the link of vertex $i_w$ in $\bH$ by first choosing the number of neighbors $\br\sim\Binom(n-1,p)$, then sampling $\br$ points $\bv_1,\dots,\bv_{\br}$ independently and uniformly from a measure-$p$ cap in $\bbS^{d-1}$ centered at some point $w$ (corresponding to the vector of the link vertex $i_w$), placing an edge between every $i,j$ such that $\angles*{\bv_i,\bv_j}\ge \tau(p,d)$.
Finally, we remove any isolated vertices; here, we'll show that the graph expands with high probability before removing these isolated vertices, which implies that no isolated vertices have to be removed.
For the remainder of the overview, let $\tau = \tau(p,d)$. 
We'll show that:
\begin{theorem}[Informal version of \pref{thm:main-thm-links}]\label{thm:inf-main-thm-links}
    Let $\bG$ be the link of some point $w\sim \bbS^{d-1}$ induced by $\bv_1,\dots,\bv_m\sim\scap_{p}(w)$ .  
Then with high probability $\bG$ is a $\mu$-expander where
    \[
        \mu \coloneqq (1+o(1))\cdot\max\braces*{\,\frac{\tau}{\tau+1},\, \frac{\log^4 m}{\sqrt{qm}} \,}+ o_d(1).
    \]
    Here $q = \Pr_{u,v\sim\bbS^{d-2}}\left[\iprod{u,v}\ge \frac{\tau}{\tau+1}\right]$.
\end{theorem}

\paragraph{Links are essentially random geometric graphs in one lower dimension.}
Since most of the measure of the cap lies close to its boundary, intuitively the link is distributed \emph{almost} like a random geometric graph with points drawn independently from the cap boundary, i.e. the shell $\Shell_p(w)\coloneqq\{x:\angles*{x,w}=\tau\}$.
Our proof of \pref{thm:inf-main-thm-links} must pay attention to the fluctuations in $\iprod{\bv_i,w}-\tau$, but to simplify our current discussion we assume each link is in fact a random geometric graph on $\Shell_p(w)$, and address the fluctuations later in the overview.

Observe that a uniformly random $\bv$ from $\Shell_p(w)$ is distributed as $\tau\cdot w + \sqrt{1-\tau^2} \cdot \bu$ where $\bu$ is a uniformly random unit vector orthogonal to $w$.
Using this decomposition, we see that $\angles*{\bv_i,\bv_j}\ge\tau$ if and only if
\(\angles*{\bu_i,\bu_j} \ge \frac{\tau}{1+\tau}\).
Thus, under our simplifying assumption, the link is distributed exactly like a random geometric graph on $\bbS^{d-2}$ with inner product threshold $\frac{\tau}{1+\tau}$.
Hence (up to the difference between $\scap_p(w)$ and $\Shell_p(w)$), to understand link expansion we can study the second eigenvalue of a random geometric graph on the sphere.

\begin{remark}[Requiring $d = \Theta(\log n)$]
In light of \pref{thm:inf-main-thm-links} (and even the heuristic discussion above), it turns out that $d = \Theta(\log n)$ is the only regime for which the links can be connected while the 1-skeleton has average degree $\ll \sqrt{n}$.
To see this, we consider the relationship between $p,\tau,$ and $d$; we have that 
\begin{equation}\label{eq:p-val}
p = \Pr_{\bv,\bv' \sim \bbS^{d-1}}\left[\Iprod{\bv,\bv'} \ge \tau\right]
= \Theta\parens*{\tfrac{1}{\tau d}} \cdot \parens*{1-\tau^2}^{\frac{d-1}{2}} 
\approx \exp(-d\tau^2/2).
\end{equation}
See \pref{lem:approx-tails-dpd} for a formal argument.\footnote{Heuristically, it makes sense that $p = \Pr[\Iprod{\bv,\bv'} \ge \tau] \approx \exp(-\Theta(\tau^2 d))$, because $\Iprod{\bv,\bv'}$ is approximately $\cN(0,\frac{1}{d})$.}
Note that the arguments above in conjunction with \pref{eq:p-val} imply that the probability that two vertices within the link are connected is also roughly 
\[
q = \Pr_{\bu,\bu' \sim \bbS^{d-2}}\left[\Iprod{\bu,\bu'} \ge \tfrac{\tau}{1+\tau}\right] = \Theta \left(\tfrac{1}{\tau d} \right) \cdot  \left(1-\tfrac{\tau^2}{(1+\tau)^2}\right)^{\frac{d-2}{2}},
\] 
since the link is like a random geometric graph on $\Shell_p(w)$.

Connectivity within the links in conjunction with sparsity now requires us to have $d \in \Theta(\log n)$:
The number of vertices inside each link concentrates around $m = np$, so the average degree inside the link is $q m \approx qpn$; we must have the average link degree $qpn \ge 1$, otherwise the link is likely disconnected. 
Now, if $\tau = o(1)$, then $\tau \approx \frac{\tau}{1+\tau}$ and $p \approx q$, so $qp n \ge 1 \implies p^2 n \gtrapprox 1 \implies p \gtrapprox n^{-1/2}$, ruling out a $1$-skeleton with average degree $\ll\sqrt{n}$.
Hence we need $\tau = \Omega(1)$.
Given that $\tau = \Omega(1)$, \pref{eq:p-val} implies that to have the average 1-skeleton degree $\sqrt{n} \ge pn \ge 1$ we need $d \in \Theta(\log n)$.
\end{remark}

\paragraph{Spectral expansion in random geometric graphs.}
We now explain how to prove near-sharp second eigenvalue bounds for random geometric graphs.
\restatetheorem{thm:sphere}
As mentioned above, \pref{thm:sphere} is a consequence of the more general \pref{thm:random-restriction} about the second eigenvalue of random restrictions of vertex-transitive graphs, and the inner product threshold $\tau = \tau(p,d)$ appears as the mixing rate of the random walk on $\bbS^{d-1}$ where a step originating at $v$ walks to a random vector in $\scap_p(v)$.
Via standard concentration arguments applied to the vertex degrees, to prove the above it suffices to bound $\norm*{A_{\bG}-\E A_{\bG}} \le \mu \cdot pn$, where $A_{\bG}$ is the (unnormalized) adjacency matrix of $\bG$. 
We'll focus on the regime where $pn \gg \polylog n$, so that $\mu \approx \tau$.

\medskip
\noindent {\bf Trace method for random geometric graphs.}
To bound $\norm*{A_{\bG}-\E A_{\bG}}$, we employ the trace method, bounding the expected trace of a power of $A_{\bG} - \E A_{\bG}$.
This is sufficient for the following reason:
for convenience, let $\ol{A}_{\bG} = A_{\bG} - \E A_{\bG}$, and let $\ell$ be any non-negative, even integer.
Since $\ell$ is even,
\[
\norm*{\ol{A}_{\bG}}^{\ell} = \norm*{\ol{A}_{\bG}^{\ell}} \le \tr\parens*{\ol{A}_{\bG}^{\ell}},
\]
And so applying Markov's inequality,
\[
\Pr\left(\norm*{\ol{A}_{\bG}} \ge e^{\eps} \left(\E \tr\left( \ol{A}_{\bG}^{\ell}\right)\right)^{1/\ell}\right) 
= \Pr\left(\norm*{\ol{A}_{\bG}}^{\ell} \ge e^{\eps\ell} \E \tr \left(\ol{A}_{\bG}^{\ell}\right)\right) 
\le \exp(-\eps \ell).
\]
Thus, our goal reduces to bounding the expectation of $\tr\parens{\ol{A}_{\bG}^{\ell}}$ for a sufficiently large even $\ell$; in particular, if we choose $\ell \gg \log n$, then since $\ol{A}_{\bG}$ has $n$ eigenvalues, $\tr\parens{\ol{A}_{\bG}^{\ell}}^{1/\ell}$ is a good ``soft-max'' proxy for $\norm{\ol{A}_{\bG}}$, and we will obtain high-probability bounds. 

We now explain why properties of random walks on $\bbS^{d-1}$ naturally arise when applying the trace method.
Concretely, $\tr\parens{\ol{A}_{\bG}^{\ell}}$ is a sum over products of entries of $\ol{A}_{\bG}$ corresponding to closed walks of length $\ell$ in the complete graph $K_n$ on $n$ vertices:
\[
    \tr\parens*{\ol{A}_{\bG}^{\ell}}
 = \sum_{i_0,\ldots,i_{\ell-1} \in [n]}\, \prod_{t=0}^{\ell-1} (\ol{A}_{\bG})_{i_ti_{t+1\mathrm{\,\, mod } \ell}},
\]
The walk $i_0,i_2,\ldots,i_{\ell-1},i_0$ can be represented as a directed graph. When we take the expectation, the symmetry of the distribution means that all sequences $i_0,\ldots,i_{\ell-1}$ which result in the same graph (up to relabeling) give the same value.
That is, letting $\calW_{\ell}$ be the set of all such graphs, and for each $W \in \calW_{\ell}$ letting $N_W$ be the number of ways it can arise in the sum above,
\begin{equation}
    \E \tr\parens*{\ol{A}_{\bG}^{\ell}}
 = \sum_{W \in \calW_{\ell}}\, N_W \cdot \E \prod_{(i,j) \in W} (\ol{A}_{\bG})_{ij}.\label{eq:trace-sum}
\end{equation}
To bound this sum, we must bound the expectation contributed by each $W \in \calW_{\ell}$. 
For the sake of this overview we will consider only the case when $W = C_{\ell}$, the cycle on $\ell$ vertices, as it requires less accounting than the other cases; however it is reasonable to restrict our attention to this case for now, as bounding it already demonstrates our main ideas, and because this term roughly dominates the sum with $N_{C_{\ell}} \gg N_{W'}$ for all other $W' \in \calW_{\ell}$ at $\ell = \polylog n$ and $pn \gg \polylog n$.\footnote{
Briefly, this is because whenever $i_0,\ldots,i_{\ell-1}$ are all distinct elements of $[n]$, the resulting walk's graph is a cycle, and when $\ell = \polylog n$, $\ell$ indices sampled at random from $[n]$ are all distinct with high probability.} 

We now bound the expectation for the case $W=C_{\ell}$; readers uninterested in the finer details may skip to the conclusion in \pref{eq:subgraph-prob}.
We expand the product using that $(\ol{A}_{\bG})_{ij} = \bA_{ij} - p$ (since  $\E[\bA_{ij}] = p$):
\[
    \E \prod_{i=1}^\ell (\bA_{i,i+1} - p)
= \sum_{ T \subseteq [\ell]} (-p)^{\ell-|T|} \E \prod_{i \in T} \bA_{i,i+1} 
= \sum_{ T \subseteq [\ell]} (-p)^{\ell-|T|} \Pr[\{(i,i+1) : i \in T\}\text{ is subgraph of }\bG]. 
  \numberthis \label{eq:e-trace-subgraph}
\]
and thus our focus is to understand subgraph probabilities in a random geometric graph.
It is not too hard to see that when the edges specified by $T$ form a forest, its subgraph probability is $p^{|T|}$, identical to its counterpart in an \erdos--\renyi graph; the nontrivial correlations introduced by the geometry only play a role when $T$ has cycles.
Hence, the sum \pref{eq:e-trace-subgraph} simplifies,
\[
\E\prod_{i=1}^\ell (\ol{A}_{\bG})_{i,i+1}
= \sum_{ T \subsetneq [\ell]} (-p)^{\ell-|T|} p^{|T|} + \Pr[C_\ell \text{ is subgraph of } \bG] 
= \Pr[C_\ell \text{ is subgraph of } \bG] - p^\ell,
\numberthis\label{eq:subgraph-prob}
\]
where we used that the binomial sum is equal to $(p-p)^\ell = 0$.

Hence it remains to estimate the subgraph probability of a length-$\ell$ cycle.
We will now see how subgraph probabilities are related to the mixing rate of a random walk on $\bbS^{d-1}$.

\paragraph{Subgraph probability of a cycle in a random geometric graph.}
For the cycle $C_\ell = 0,1,\ldots,\ell-1,0$, by Bayes' rule:
\begin{align*}
    \Pr[C_\ell\in \bG] 
&= \prod_{i=0}^{\ell-1} \Pr[(i,i+1)\in\bG \mid \forall j < i, (j,j+1) \in\bG ]
= p^{\ell-1} \cdot \Pr[(\ell-1,0)\in\bG \mid 0,1,\ldots \ell-1\in\bG],
\end{align*}
since in all but the step $i+1=\ell$, the graph in question is a forest.
Identifying each $i$ with a point $\bx_{i}$ on $\bbS^{d-1}$, for any choice of $\bx_{0}$ the above probability can equivalently be written as
\[
    p^{\ell-1} \cdot \Pr\bracks*{ \angles*{\bx_{\ell-1}, \bx_0} \ge \tau \mid \angles*{\bx_i, \bx_{i+1}} \ge \tau:  0 \le i \le \ell-2 }.
\]
Denoting with $P$ the transition kernel of the random walk we alluded to earlier, where in one step we walk from a point $x$ to a uniformly random point in $\scap_p(x)$, we can write the distribution of $\bx_{\ell}\mid\braces*{\bx_0, \angles*{\bx_i, \bx_{i+1}} \ge \tau:  0 \le i \le \ell-2}$ as $P^{\ell-1}\delta_{\bx_0}$ where $\delta_{\bx_0}$ refers to the point mass probability distribution supported at $\bx_0$.
In turn, we can write the subgraph probability as:
\[
    p^{\ell-1} \cdot \Pr_{\bx_{\ell-1}\sim P^{\ell-1}\delta_{\bx_0} } \bracks*{ \bx_{\ell-1} \in \scap_p(\bx_0) }.
\]
If $\bx_{\ell-1}$ were sampled from the uniform distribution $\Unif$ on $\bbS^{d-1}$ then the probability of landing in $\scap_p(\bx_0)$ would be $p$, which lets us upper bound the subgraph probability by:
\[
    p^{\ell-1} \cdot \parens*{ p + \dtv{P^{\ell-1}\delta_{\bx_0}}{\Unif} }.
\]
The terms for more complicated subgraphs $W' \in \calW_{\ell}$ also similarly depend on the mixing properties of $P$ via subgraph probabilities.
Our next goal then is to understand the mixing properties of $P$.
\begin{remark}
    To prove \pref{thm:random-restriction} about random restrictions, the same strategy is used to relate subgraph probabilities with mixing rate of the random walk on the original graph we start with. 
\end{remark}

\paragraph{Mixing properties of $P$.}
We show that the walk over $\bbS^{d-1}$ with transition kernel $P$ contracts the TV distance by coupling this discrete walk with the continuous Brownian motion $U_t$ over $\bbS^{d-1}$. Then via a known log-Sobolev inequality for Brownian motion on spheres, we can prove the following contraction property for $P$. 

\begin{theorem}[Informal version of \pref{thm:decay}]
    \label{thm:inf-decay-multi}
    For any probability measure $\alpha$ over $\bbS^{d-1}$ and integer $k \ge 0$,
    \[
        \dtv{P_{p}^k \alpha}{\rho} \le \left((1+o_{d\tau^2}(1))\cdot\tau \right)^{k-1} \cdot \sqrt{\frac{1}{2}\log \frac{1}{p}},
    \]
    where $P_{p}$ denotes the transition kernel in which every $x\in\bbS^{d-1}$ walks to a uniformly random point in the measure-$p$ cap around it and $o_{d\tau^2}(1)$ denotes a function that goes to $0$ as $d\tau^2 \to \infty$.
\end{theorem}

We leave the details to \pref{sec:bm-to-discrete}, but in brief, the reason we are able to execute this coupling is that the probability mass in $P\delta_{\bx_0}$ concentrates around $\Shell_{=\tau}(\bx_0)$, and most of the $(\frac{1}{d-1}\log \frac{1}{\tau})$-step Brownian motion starting from $\bx_0$ concentrates at $\Shell_{=\tau}(\bx_0)$, so when $t = \frac{1}{d-1}\log\frac{1}{\tau}$ the operators $P$ and $U_t$ have similar action.

We can now apply \pref{thm:inf-decay-multi} to bound $\dtv{P^{\ell-1}\delta_{\bx_0}}{\rho}$ with $\alpha = \delta_{\bx_0}$ and $k = \ell-1$:
\[
\dtv{P^{\ell-1}\delta_{\bx_0}}{\rho} \le ((1+o(1))\tau)^{\ell-2} \sqrt{\tfrac{1}{2}\log \tfrac{1}{p}}.
\]

\paragraph{Spectral norm of random geometric graph.}

We now return to bounding the expected trace of $\ol{A}_{\bG}^\ell$; putting together the above, we have the bound
\[
\E \prod_{(i,j) \in C_\ell} (\ol{A}_{\bG})_{ij} 
\le \Pr[C_\ell \in \bG] - p^\ell
\le p^{\ell-1}\left(p + \dtv{P^{\ell-1}\delta_{\bx_0}}{\rho} \right) - p^\ell
\le p^{\ell-1}((1+o(1))\tau)^{\ell-2} \sqrt{\tfrac{1}{2}\log\tfrac{1}{p}}.
\]

The coefficient $N_{C_{\ell}}$ in front of the $W = C_{\ell}$ term  in \pref{eq:trace-sum} is the number of sequences $i_1,\ldots,i_\ell \in [n]$ which yield an $\ell$-cycle graph; this happens if and only if all of the indices are distinct, so $N_{C_{\ell}} = \ell! \cdot \binom{n}{\ell} \le n^\ell$.
Hence the contribution of the $\ell$-cycle to the sum is at most $((1+o(1)) np \tau)^{\ell-2} \cdot \poly(n)$ when $p > 1/n$.
By a careful accounting similar to the above for all graphs $W \in \calW_\ell$, one can show that in the parameter regime $pn \gg \polylog(n)$ and $\ell = \polylog n$, the term $W = C_{\ell}$ contains $(1-o(1))$ of the total value of this sum, so we obtain the bound
\[
\left[\E \tr(\ol{A}_{\bG}^\ell)\right]^{1/\ell} \le \left((1+o(1))\cdot ((1+o(1)) np \tau)^{\ell-2} \cdot \poly(n)\right)^{1/\ell} = (1+o(1)) np \tau,
\]
when we choose $\ell = \omega(\log n)$.
Applying Markov's inequality we conclude that $\|\ol{A}_{\bG}\| \le (1+o(1)) np \tau$ with high probability, and normalizing by the degrees (which concentrate well around $np$) we conclude our upper bound of $\tau$ in \pref{thm:sphere}.

\paragraph{Adapting the spectral norm bound to links.}
Up until now, we have pretended that the link of $i_w$ is a random geometric graph, where the vertices are identified with vectors in $\Shell_{=\tau}(w)$, rather than $\scap_{\ge \tau}(w)$.
While it is true that most of the probability mass in $\scap_{\ge \tau}(w)$ is close to the boundary, some $\frac{1}{\poly(m)}$-fraction of the vertices $j$ in the link will have $\iprod{\bv_j,w} = \bkappa_j > (1+\delta)\tau$ for some $\delta > 0$.
And within the link, these vertices will have higher expected degree: for $\bv_i,\bv_j$ having $\Iprod{\bv_i,w} = \bkappa_i$ and $\Iprod{\bv_j,w} = \bkappa_j$, following a similar calculation to the one above, 
\begin{equation}
q_{ij} \coloneqq \Pr[i \sim j] = \Pr_{\bu_i,\bu_j \sim \bbS^{d-2}}\left[\Iprod{\bu_i,\bu_j} \ge \tfrac{\tau - \bkappa_i\bkappa_j}{\ssqrt{(1-\bkappa_i^2)(1-\bkappa_j^2)}}\right]\label{eq:qij}
\end{equation}
And this quantity is $\gg q = \Pr_{\bu_i,\bu_j \sim \bbS^{d-2}}[\Iprod{\bu_i,\bu_j} \ge \tfrac{\tau}{1+\tau}]$ when $\bkappa_i > (1+\delta)\tau$ and $\bkappa_j \ge \tau$.
Hence, vertex degrees are not as well concentrated within each link as they are (around $pn$) in the entire graph $\bH$.

As a result, if we let $\bG_w$ now stand for the link and $A_{\bG_w}$ now stand for the adjacency matrix of the link, it is no longer the case that $\|A_{\bG_w} - \E A_{\bG_w}\|$ is small:
$\E A_{\bG_w}$ still has every entry equal to $q$, but the top eigenvector of $A_{\bG}$ will not be close to the all-$1$ vector.

To contend with this, we analyze the spectral norm of $A_{\bG}$ {\em conditioned} on the shells that the points in $\scap_p(w)$ are in. 
Letting $\bkappa \in [\tau,1]^m$ be such that $\bkappa_i = \iprod{\bv_i,w}$,
vertex degrees concentrate in $\bG_w$ conditioned on $\bkappa$, and we can readily bound the spectral norm of $\ol{A}_{\bG_w} \mid \bkappa=  A_{\bG_w} \mid \bkappa - \E[A_{\bG_w} \mid \bkappa]$.

The analysis of the spectral norm of $\ol{A}_{\bG_w}$ is then not so different from that of $\ol{A}_{\bG}$ for $\bG$ a random geometric graph; the main difference is that now, instead of working with the walk $P$ in which we walk from $\bu_i$ to a random point in $\scap_{\ge \tau}(\bu_i)$, at each step of the walk we must adjust the volume of the cap: when considering the probability that the edge $i,j$ is present, we apply the operator $P_{q_{ij}}$ for $q_{ij}(\bkappa_i,\bkappa_j)$ as defined in \pref{eq:qij}, which walks from $\bu_i$ to a random point in $\scap_{q_{ij}}(\bu_i)$.
This requires some additional accounting, but one can show that the slowest mixing occurs when $\bkappa_i = \bkappa_j = \tau$ and  $q_{ij} = \frac{\tau}{1+\tau}$, from which we obtain the desired bound on $\|\ol{A}_{\bG_w} \mid \kappa\|$.
For details, see \pref{sec:link-eigs}.

One additional complication is that $\E A_{\bG_w}\mid \bkappa$ is not a rank-1 matrix, so bounding $\|\ol{A}_{\bG_w} \mid \kappa\|$ does not directly imply a bound on the second eigenvalue of $A_{\bG_w}$.
However, it turns out that $\E A_{\bG_w}\mid \bkappa$ is sufficiently close to a rank-1 matrix $R_{\bG_w}$ (the matrix whose $(i,j)$th entry is the product of the expected degrees conditioned on $\bkappa$) that we can apply the triangle inequality:
\[
\norm*{(A_{\bG_w} - R_{\bG_w}) \mid \bkappa } \le \norm*{ A_{\bG_w} \mid \bkappa - \E[A_{\bG_w}\mid \bkappa]} + \norm*{\E[A_{\bG_w} \mid \bkappa] - R_{\bG_w} \mid \bkappa},
\]
the first term we bound using the trace method as described above.
The second term we bound via more-or-less direct calculation: because all but an $o(1)$ fraction of $\bkappa_i \approx \tau$, when ignoring an $o(1)$ fraction of rows and columns, the rows of $\E A_{\bG_w}\mid \bkappa$ are almost constant multiples of each other, and further these $o(1)$ fraction of rows and columns represent an $o(1)$ fraction of the total absolute value of $\E[A_{\bG_w}\mid \bkappa]$. (This is because the high-degree vertices in $\bG_w$ represent an $o(1)$ fraction of the total edges in $\bG_w$.)
Now, thinking of $\E[A_{\bG_w}\mid \bkappa]$ as a transition operator of a Markov chain, we are able to use this to argue that the Markov chain mixes so rapidly that $\E[A_{\bG_w}\mid \bkappa]$ must be close to $R_{\bG_w}\mid \bkappa$, yielding the desired bound.
For details, see \pref{sec:shells}.

\subsection*{Organization}

In \pref{sec:prelim} we give some technical preliminaries.
In \pref{sec:trace-method} we use the trace method prove the spectral norm bound for random restrictions of arbitrary graphs, \pref{thm:random-restriction}.
To apply \pref{thm:random-restriction} to bound the spectrum of $\grg_d(n,p)$ (and also to ultimately prove that $\grg_d^{(2)}(n,p)$ is a 2-dimensional expander), we must prove the total variation decay condition for a random walk on the sphere with steps consisting of jumps to a random point in a spherical cap. 
We do this in \pref{sec:bm-to-discrete} by relating this walk with discrete jumps to Brownian motion on $\bbS^{d-1}$.
The links of vertices in $\grg_d^{(2)}(n,p)$ do not conform to the requirements of \pref{thm:random-restriction} because they are random restrictions of graphs which are not vertex-transitive, and so in \pref{sec:link-eigs} and \pref{sec:shells} we prove a version of \pref{thm:random-restriction} specialized to these links; \pref{sec:link-eigs} contains the trace method and \pref{sec:shells} addresses the fact that the top eigenvector is not proportional to $\vec{1}$.
Finally we put the pieces together in \pref{sec:wrapup} to prove \pref{thm:hdx}.
In \pref{sec:tight}, we show that the trickling-down theorem is tight.

\section{Preliminaries} \label{sec:prelim}

\paragraph{Notation.}
For a self-adjoint matrix $M$, we denote its eigenvalues in decreasing order as $\lambda_1(M)\ge \dots \ge \lambda_n(M)$, the absolute values of its eigenvalues as $|\lambda|_1(M)\ge\dots\ge|\lambda|_n(M)$, and $\lambda_{\max}(M)$ and $|\lambda|_{\max}(M)$ to denote $\lambda_1(M)$ and $|\lambda|_1(M)$ respectively.
Given a sequence of matrices $M_1,\dots,M_T$ we use $\prod_{i=1}^T M_i$ to denote the matrix $M_T \cdot M_{T-1}\cdots M_1$.

\medskip
For a graph $G$, we use $V(G)$ to refer to its vertex set and $E(G)$ to refer to its edge set.
For a vertex $v\in V(G)$, we use $N(v)$ to denote the set of neighbors of $v$.

\medskip
For a probability distribution $\calD$, we use $\Phi_{\calD}(x)$ to denote the CDF of $\calD$ at $x$, and $\ol{\Phi}_{\calD}(x)\coloneqq 1 - \Phi_{\calD}(x)$ to denote the tail of $\calD$ at $x$.  For any point $x$, we use $\delta_x$ to denote the delta distribution at $x$ .

\subsection{Linear algebra}
The following articulates how one gets a handle on the second eigenvalue of a matrix after subtracting a rank-$1$ term, which will be used in \pref{sec:trace-method} and \pref{sec:link-eigs}.
\begin{fact}   \label{fact:rank-1-sub}
    For any $n\times n$ symmetric matrix $M$ and rank-$1$ PSD matrix $R$,
    \(
        |\lambda|_2(M) \le \norm*{M-R}.
    \)
\end{fact}
\begin{proof}
    By Cauchy's interlacing theorem,
    $\lambda_2(M) \le \lambda_1(M-R) \le \norm*{M-R}$ and
    $-\lambda_n(M) \le -\lambda_n(M-R) \le \norm*{M-R}$.
    The desired inequality is then true since $|\lambda|_2(M) \le \max\braces*{\lambda_2(M),-\lambda_n(M)}$.
\end{proof}

Establishing second eigenvalue bounds in \pref{sec:trace-method} and \pref{sec:link-eigs} also involves bounding the spectral norm of some matrices via the ``trace method'' articulated below.
\begin{claim}[Trace Method]\label{claim:trace-method} 
    Let $\bM$ be a symmetric (random) matrix. 
    Then for any even integer $\ell \ge 0$,
    \[
        \Pr\left[\norm*{\bM} \ge e^\eps \cdot \E\left[\tr\parens*{\parens*{\bM}^{\ell}}\right]^{1/\ell}\right] \le \exp(-\eps\ell).
    \]
\end{claim}
\begin{proof}
    By Markov's inequality, $\Pr\left[\norm*{\bM} \ge t\right] \le t^{-\ell} \E \parens*{\norm*{\bM}^\ell}$.
    The claim then follows because for any self-adjoint matrix $M$, $\lambda_{\max}\parens*{M^{\ell}} \le \tr(M^\ell)$ when $\ell$ is even.
\end{proof}

We will also require the following bound on the spectrum of a matrix, which is a special case of the Gershgorin circle theorem.
\begin{claim}[Row sum bound]    \label{claim:row-sum-bound}
    For any matrix $M$, $|\lambda|_{\max}(M) \le \max_i \norm*{M[i,*]}_1$.
\end{claim}
\begin{proof}
Let $v$ be the eigenvector achieving $\lambda = |\lambda|_{\max}(M)$.
Then letting $k$ be the index maximizing $|v_k|$, we have 
\[
|\lambda v_k| = |(Mv)_k| = \left|\sum_{j} M_{kj} v_j\right|\le |v_k| \sum_{j} |M_{kj}| \le |v_k| \max_{i} \|M[i,*]\|_1,
\]
and dividing through by $|v_k|$ gives the conclusion.
\end{proof}

\subsection{Probability}

\begin{definition}
    The \emph{total variation distance} between probability distributions $\mu$ and $\nu$ is defined as:
    \[
        \dtv{\mu}{\nu} \coloneqq \max_{\calE} \abs*{\mu(\calE)-\nu(\calE)}.    
    \]
\end{definition}

\begin{fact}
    When $\rho$ is a nonnegative measure such that $\mu$ and $\nu$ are absolutely continuous with respect to $\rho$, then:
    \[
        \dtv{\mu}{\nu} = \frac{1}{2}\int \abs*{\frac{d\mu}{d\rho}(x) - \frac{d\nu}{d\rho}(x)}\,d\rho(x) = \int \parens*{\frac{d\mu}{d\rho}(x) - \frac{d\nu}{d\rho}(x)}\cdot\Ind\bracks*{\frac{d\mu}{d\rho}(x) > \frac{d\nu}{d\rho}(x)}\,d\rho(x).   
    \]
    When $\mu$ and $\nu$ are supported on $[n]$, then:
    \[
        \dtv{\mu}{\nu} = \frac{1}{2}\norm*{\mu - \nu}_1 = \sum_{i=1}^n (\mu(i) - \nu(i)) \cdot \Ind[\mu(i) > \nu(i)]
    \]
    where $\mu$ and $\nu$ are the vectors of probabilities.
\end{fact}

We describe a Markov chain via its transition operator $P$ where
$P(i, j)$ denotes the probability of transitioning from state $i$ to state $j$.

We call the joint distribution $\omega(\mu, \nu)$ a coupling between two distributions $\mu$ and $\nu$ if $\mu = \omega(\cdot, \nu)$ and $\mu = \omega(\mu, \cdot)$. 
In other words, the marginals of $\omega$ correspond to $\mu$ and $\nu$. 

\begin{fact} \label{fact:coupling-tv}
	Let $x$ and $y$ be two arbitrary states in a Markov chain over state space $\Omega$ with transition operator $P$, and sample $X \sim P(x, \cdot)$ and $Y \sim P(y, \cdot)$, where $P(z, \cdot)$ denotes the distribution over $\Omega$ given by a single step of the walk starting from state $z$. 
	Then, there exists a coupling of $X$ and $Y$ such that $X = Y$ with probability $1 - \eps$ if and only if $\dtv{P(x, \cdot)}{P(y, \cdot)} \leq \eps$.  
\end{fact}

\subsection{The uniform distribution over the unit sphere}\label{sec:prelim-sphere}
We use $\Unif$ to denote the uniform distribution on $\bbS^{d-1}$.

Let $v\in\bbS^{d-1}$ and $\bw\sim\Unif$.  Then the distribution $\Beta{d}$ of $\angles*{\bw,v}$ is invariant under the choice of $v$, is supported on $[-1,1]$ and has probability density function:
\[
    \BetaPDF_{d}(x) = \frac{\Gamma\parens*{\frac{d}{2}}}{\Gamma\parens*{\frac{d-1}{2}} \sqrt{\pi} }\cdot \parens*{1-x^2}^{(d-3)/2}.
\]
Henceforth, we use $Z_d$ to denote the normalizing constant $\frac{\Gamma\parens*{\frac{d}{2}}}{\Gamma\parens*{\frac{d-1}{2}} \sqrt{\pi}}$.
\begin{fact}
    $Z_d \leq O(\sqrt{d})$.
\end{fact}
\noindent In addition, we will rely heavily on the following sharp estimate of the tail of $\Beta{d}$.
\begin{lemma} \label{lem:approx-tails-dpd}
    Let $\Phi_{\Beta{d}}(t) \coloneqq \Pr_{X\sim\Beta{d}}[X\ge t]$.
    Then, when $t \geq 0$:
    \[
        \frac{Z_d}{t(d-1)}\cdot\parens*{1-t^2}^{(d-1)/2}\cdot\parens*{1 - \frac{4\log\parens*{1+d \cdot t^2}}{d \cdot t^2}} \le \ol{\Phi}_{\Beta{d}}(t) \le \frac{Z_d}{t(d-1)}\cdot\parens*{1-t^2}^{(d-1)/2}.
    \]
\end{lemma}
\begin{proof}
    It suffices to upper and lower bound $\int_t^1 (1-x^2)^{(d-3)/2}$.
    We first obtain an upper bound.
    \begin{align*}
        \int_{t}^1 \parens*{1 - x^2}^{(d - 3)/2} dx &= \frac{1}{t} \int_t^1 t\parens*{1 - x^2}^{(d - 3)/2} dx \\
        &\leq \frac{1}{t} \int_t^1 x\parens*{1 - x^2}^{(d - 3)/2} dx \\
        &= -\frac{1}{t(d - 1)} \cdot \parens*{1 - x^2}^{(d - 1)/2} \Bigg|_t^1 \\
        &= \frac{1}{t(d - 1)} \cdot \parens*{1 - t^2}^{(d - 1)/2}
    \end{align*}
    Now we prove the lower bound.
    For any $\eps > 0$ such that $t\cdot\sqrt{1-\eps+\frac{\eps}{t^2}} \le 1$, and defining $\delta \coloneqq \frac{\eps}{t^2} - \eps$, we have the following.
    \begin{align*}
        \int_{t}^1 (1 - x^2)^{(d - 3)/2} dx &\geq \frac{1}{t\sqrt{1 + \delta}} \int_t^{t\sqrt{1 + \delta}} \parens*{t\sqrt{1+\delta}}\parens*{1 - x^2}^{(d - 3)/2} dx \\
        &\geq \frac{1-\delta}{t} \int_t^{t\sqrt{1+\delta}} x\parens*{1 - x^2}^{(d - 3)/2} dx \\
        &= -\frac{1-\delta}{t(d - 1)} \cdot \parens*{1 - x^2}^{(d - 1)/2} \Bigg|_{t}^{t\sqrt{1+\delta}} \\
        &= \frac{1-\delta}{t(d - 1)} \cdot \parens*{1-t^2}^{(d-1)/2} \cdot \parens*{1 - (1-\eps)^{(d-1)/2}}
    \end{align*}
    where the second inequality uses $\frac{1}{\sqrt{1+\delta}} \ge 1-\delta$ and the last equality uses $1-t^2(1+\delta) = (1-t^2)(1-\eps)$.
    Choosing $\eps = \frac{2\log\parens*{1+dt^2}}{d-1}$ yields: 
    \begin{align*}
        \int_{t}^1 \parens*{1 - x^2}^{(d - 3)/2} \ge \frac{1}{t(d-1)}\cdot\parens*{1-t^2}^{(d-1)/2} \cdot \parens*{1 - \frac{4\log\parens*{1+dt^2}}{dt^2}}. & \qedhere
    \end{align*}
\end{proof}

We use $\Beta{d} \vert_{\geq \tau}$ to represent $\Beta{d}$ conditioned on lying in $[\tau, 1]$.

\begin{definition}
    For a vector $y$, we use $\scap_p(y)$ and $\scap_{\ge \tau(p)}(y)$ interchangeably to denote the measure-$p$ spherical cap around $y$:
    \[
        \scap_p(y)= \scap_{\ge \tau(p)}(y) \coloneqq \left\{u: \langle u, y\rangle \ge \tau(p), u\in\bbS^{d-1} \right\}.
    \]
    We use $\mcap_p(y)$ and $\mcap_{\ge \tau(p)}(y)$ to denote the uniform measure over the set $\scap_p(y)$.
    We denote the boundary of $\scap_p(y)$ by $\Shell_p(y)$ or $\Shell_{=\tau(p)}(y)$.
    That is,
    \[
        \Shell_p(y) \coloneqq \left\{u: \langle u, y\rangle = \tau(p), u\in\bbS^{d-1} \right\}.
    \]
\end{definition}

\section{The second eigenvalue of random restrictions}  \label{sec:trace-method}
\newcommand{\sing}{\mathrm{sing}}
In this section we prove \pref{thm:random-restriction}.
Let $X$ be a (possibly infinite) vertex-transitive graph with a unique stationary measure $\Unif$.
Let $\bG \sim \rr_n(X)$ be a random restriction of $X$ as defined in \pref{def:rr}, and let $p = \Pr_{\bG \sim \rr_n(X)}[(i,j) \in E(\bG)]$ be the marginal edge probability in $\bG$.
Suppose furthermore that
\begin{equation}
     \exists \, C,\lambda \text{ with } C\ge 1 \text{ and } \frac{1}{\sqrt{pn}}\le\lambda\le 1 \quad \text{s.t. for any distribution } \alpha \text{ on } V(X),\quad \dtv{X^k \alpha}{\Unif} \le C\lambda^k.\label{assumption:tv-bound}
\end{equation}
We overload notation and use $X$ to denote the transition operator for the simple random walk on $X$, and for $H\subseteq V(X)$ we also use $H$ to denote the indicator vector of the set $H$.

We denote its adjacency matrix by $A_{\bG}$, the diagonal degree matrix by $D_{\bG}$, the centered adjacency matrix by $\ol{A}_{\bG} = A_{\bG} - \E A_{\bG}$, and the normalized adjacency matrix by $\widehat{A}_{\bG} = D_{\bG}^{-1/2}A_{\bG}D_{\bG}^{-1/2}$.
Then we'll show the following.
\begin{theorem}
    As long as $pn \gg C^6 \log^8 n$, for any constant $\gamma > 0$, with probability at least $1-n^{-\gamma}$,
    \[
        |\lambda|_2\parens*{\widehat{A}_{\bG}} \le (1+o(1))\cdot \max\left(\lambda, \frac{\log^4 n}{\sqrt{pn}}\right).
    \]
\end{theorem}
\begin{proof}
By \pref{fact:rank-1-sub}, for any rank-$1$ PSD matrix $R$,
$|\lambda|_2\parens*{\wh{A}_{\bG}} \le \norm*{\wh{A}_{\bG}-R}$.
Thus we turn our attention to bounding $\norm*{\wh{A}_{\bG}-R}$ for appropriately chosen $R$.
Setting $R_{\bG} = pD_{\bG}^{-1/2} J D_{\bG}^{-1/2}$ where $J$ is the all-ones matrix and using submultiplicativity of the operator norm, we see:
\[
    \norm*{\wh{A}_{\bG}-R_{\bG}} \le \norm*{D_{\bG}^{-1/2}}^2 \cdot \norm*{A_{\bG} - pJ}.
\]
Now, observe that $\norm*{D_{\bG}^{-1/2}}^2 = \norm*{D_{\bG}^{-1}}$.
To bound this quantity, we'll use the concentration of the vertex degrees (the entries of the diagonal of $D_{\bG}$).
For every vertex, the marginal distribution of the degree is $\Binom(n,p)$.
So by Hoeffding's inequality and the union bound, when $pn \gg \log^8 n$, for any fixed $\gamma > 0$, $\left|(D_{\bG})_{ii} - pn\right| \le \sqrt{pn \log^2 n}$ for all $i \in [n]$ with probability at least $1-n^{\gamma}$.
So with probability at least $1-n^{-\gamma}$, $D_{\bG}^{-1} = \frac{1}{pn}I + \Delta$ for $\Delta$ a diagonal matrix with entries with absolute value of order $\sqrt{\log^2 n/ (pn)^3}$.
Thus, $\norm*{D_{\bG}^{-1}} \le \frac{1}{pn}\cdot\parens*{1+\frac{\log n}{\sqrt{pn}}}$.

Next, $\norm*{A_{\bG}-pJ} \le \norm*{\ol{A}_{\bG}} + p$,
where recall $\ol{A}_{\bG}=A_{\bG}-\E A_{\bG}$.
We will show:
\[
    \norm*{\ol{A}_{\bG}} \le \parens*{1+o(1)} \cdot \max\braces*{\lambda p n, \sqrt{pn}\log^4 n}.
\]
Putting these bounds together gives:
\[
    |\lambda|_2\parens*{\wh{A}_{\bG}} \le \parens*{1+o(1)}\cdot \max\braces*{\lambda, \frac{\log^4 n}{\sqrt{pn}} }.
\]

Finally, we devote the rest of the proof to bounding $\norm*{\ol{A}_{\bG}}$.
By \pref{claim:trace-method}, it suffices to bound $\E \tr\parens{\parens{\ol{A}_{\bG}}^{\ell}}$ for a large enough even $\ell$.

For an $n \times n$ matrix $M$, $\tr(M^\ell)$ can be written as a sum over length-$\ell$ closed walks on the complete graph $\calK_n$, with each walk $W$ weighted according to $\prod_{(i,j)\in W} M_{ij}$.
The exchangeability of entries in $\ol{A}_{\bG}$ means that the walks can be partitioned into equivalence classes based on their topology as graphs, where the members of each class contribute identically to the summation.
\begin{definition}
    We use $\cW_{\ell}$ to denote the collection of length-$\ell$ walks in $\calK_n$, the complete graph on $n$ vertices. 
 For $W\in \cW_{\ell}$, we use $G(W)=(V(W),E(W))$ to denote the simple graph induced by edges walked on in $W$.  
We let the \emph{multiplicity} of $e$ in $W$, $m(e)$, be the number of times $e$ occurs in $W$. 
\end{definition}
\noindent We can then write:
\begin{align*}
    \E \tr\left((A_{\bG} -\E A_{\bG})^\ell\right) &= \sum_{W\in\cW_{\ell}} \E \prod_{e\in E(W)} \parens*{\Ind[e\in\bG]-p}^{m(e)}
\numberthis \label{eq:walk-weight}
\end{align*}
We now focus on understanding each term of the above summand in terms of the properties of $G(W)$.  
Our first step is to handle leaves.
\begin{definition}
    We use $G_2(W) = (V_2(W),E_2(W))$ to denote the \emph{$2$-core} of $G(W)$, the graph obtained by recursively deleting degree-$1$ vertices from $G(W)$.  
We denote the graph induced on the edges deleted in this process as $G_1(W)$.
\end{definition}
\begin{observation} \label{obs:2-core-decomp}
    We have $G(W) = G_1(W)\cup G_2(W)$. 
Further, every vertex in $G_2(W)$ has degree at least $2$, and $G_1(W)$ is a forest where each connected component has at most one vertex in $G_2(W)$.
\end{observation}
Notice that if $F$ is a forest, then $\Pr[F \in \bG] = p^{|E(F)|}$, and further if $F$ is a forest sharing at most one vertex with a graph $H$, then the events $\{H \in \bG\}$ and $\{F \in \bG\}$ are independent.
Hence, with the above decomposition in hand, we can ``peel off'' the one-core and for any $W\in\cW_{\ell}$ we can write:
\begin{align*}
\pref{eq:walk-weight}
    =& \E_{\substack{\bu_i\\i\in V_2(W)}}\,\, \E_{\substack{\bu_j\\j\in V_1(W)\setminus V_2(W)}}\,\, \prod_{e\in E(W)} \parens*{\Ind[e\in\bG] -p}^{m(e)} \\
    =& \prod_{e\in E_1(W)} \E\parens*{ \parens*{\Ind[e \in \bG] - p}^{m(e)}} \cdot \E_{\substack{\bu_i\\i\in V_2(W)}} \prod_{e\in E_2(W)} \parens*{\Ind[e\in\bG] -p}^{m(e)}\\
    =& \prod_{e\in E_1(W)} \E\parens*{ \Ind[e \in \bG]\parens*{(1-p)^{m(e)} -(-p)^{m(e)}} + (-p)^{m(e)}} \cdot \E_{\substack{\bu_i\\i\in V_2(W)}} \prod_{e\in E_2(W)} \parens*{\Ind[e\in\bG] -p}^{m(e)}\\
    \le& \abs*{\prod_{e\in E_1(W)} \parens*{p(1-p)^{m(e)} + (1-p)(-p)^{m(e)}}} \cdot \abs*{\E_{\bu_i:i\in V_2(W)} \prod_{e\in E_2(W)} \parens*{\Ind[e\in\bG]-p}^{m(e)}}, \numberthis \label{eq:sep-2-core}
\end{align*}
where in the third line we've used that $(\Ind[e \in\bG] - p)^k = \Ind[e\in \bG]((1-p)^k - (-p)^k) + (-p)^k$.

It now remains to handle the $2$-core $G_2(W)$.
To simplify the expression, we'll exploit the following fact: if $J$ is a subset of vertices in $G_2(W)$, conditional on an assignment of $\bu_i$ for all $i\in J$, the existence of edges in regions of $G_2(W)$ separated by $J$ are independent.   
We'll take advantage of this fact by splitting $G_2(W)$ into regions separated by the set of vertices in $G_2(W)$ of degree at least $3$, leaving us to bound a collection of paths and cycles.
\begin{definition}[Junction vertices] \label{def:junc-ver}
    We use $J(W)$ to denote the set of \emph{junction vertices} of $G_2(W)$, which are vertices with degree-$\ge 3$ in $G_2(W)$, or in the case that $G_2(W)$ only has vertices of degree-$2$, we choose an arbitrary vertex ot $G_2(W)$ and add it to $J(W)$.  
We use $G_J(W) = (J(W), E_J(W))$ to denote the \emph{junction graph} of $G_2(W)$, which is a multigraph obtained by starting with $G_2(W)$ and contracting to an edge all walks $\gamma = u_0 \dots u_t$ satisfying the following conditions:
    \begin{enumerate}
        \item $u_0$ and $u_t$ are (possibly identical) junction vertices,
        \item $u_1,\dots,u_{t-1}$ are distinct vertices with degree-$2$ in $G_2(W)$.
    \end{enumerate}
    For an edge $f\in E_J(W)$, we use $\gamma(f) = u_0,\ldots,u_t$ to identify the walk from which $f$ arose in $G_2(W)$, $s(f)$ to denote the ``start'' vertex $u_0$ of $\gamma(f)$, and $t(f)$ to denote the ``terminal'' vertex $u_t$ of $\gamma(f)$. 
\end{definition}
Then we can bound the contribution of the $2$-core in terms of the contribution of the walk $\gamma(f)$ corresponding to each edge $f$ in the junction graph:
\begin{align*}
\abs*{\E_{\substack{\bu_i\\i\in V_2(W)}} \prod_{e\in E_2(W)} \parens*{\Ind[e\in\bG]-p}^{m(e)}}
    =& \abs*{\E_{\bu_i:i\in J(W)} \E_{\bu_i:i\notin J(W)} \prod_{e\in E_2(W)} \parens*{\Ind[e\in \bG] -p}^{m(e)}} \\
    =& \abs*{\E_{\bu_i:i\in J(W)} \prod_{f\in E_{J(W)}} \E_{\bu_i:i\in \gamma(f)\setminus J(W)} \prod_{e\in\gamma(f)} \parens*{\Ind[e\in \bG]-p}^{m(e)}} \\
    \le& \E_{\bu_i:i\in J(W)} \prod_{f\in E_{J(W)}} \abs*{\E_{\bu_i:i\in \gamma(f)\setminus J(W)} \prod_{e\in\gamma(f)} \parens*{\Ind[e\in \bG] -p}^{m(e)}}.
    \numberthis \label{eq:indep-paths}
\end{align*}
We now focus on understanding the innermost expected value, the expectation over the internal vertices along a path, conditioned on the endpoints.
Again using $(\Ind[e \in\bG] - p)^k = \Ind[e\in \bG]((1-p)^k - (-p)^k) + (-p)^k$,
\begin{align*}
&\abs*{\E_{\bu_i:i\in \gamma(f)\setminus J(W)} \prod_{e\in\gamma(f)} \parens*{\Ind[e\in \bG] -p}^{m(e)}}\\
    &\qquad\qquad= \abs*{\E_{\bu_i:i\in \gamma(f)\setminus J(W)} \prod_{e\in\gamma(f)} \parens*{\Ind[e\in \bG] \cdot \parens*{(1-p)^{m(e)} - (-p)^{m(e)}} + (-p)^{m(e)}} },\\
    &\qquad\qquad= \abs*{\sum_{T\subseteq \gamma(f)} \E_{\bu_i:i\in\gamma(f)\setminus J(W)} \prod_{e\in T} \Ind[e\in \bG] \cdot \parens*{(1-p)^{m(e)} - (-p)^{m(e)}} \prod_{e\in \gamma(f) \setminus T} (-p)^{m(e)}} \\
\intertext{Now, using the independence of edges in a forest, we can bound terms where $T \neq \gamma(f)$ simply, and the term $T = \gamma(f)$ in terms of the probability that a $|\gamma(f)|$-length walk in $X$ starting at $\bu_{s(f)}$ ends at $\bu_{t(f)}$ (which is where properties of the random walk in $X$ will enter into the bound):}
    &\qquad\qquad= \Bigg|\sum_{ \substack{T \subseteq \gamma(f) \\ T \ne \gamma(f) } } \prod_{e\in T} p\cdot\parens*{(1-p)^{m(e)} - (-p)^{m(e)}} \cdot \prod_{e\in\gamma(f)\setminus T} (-p)^{m(e)} \\ 
&\qquad\qquad\qquad + \prod_{e\in\gamma(f)} \parens*{(1-p)^{m(e)} - (-p)^{m(e)}} \cdot p^{|\gamma(f)|-1} \cdot \angles*{N(\bu_{s(f)}), X^{|\gamma(f)|-1} \delta_{\bu_{t(f)}} } \Bigg|,  \\
\intertext{where $N(\bu_{s(f)})$ is the neighborhood of $\bu_{s(f)}$ in $X$, and $\delta_{\bu_{t(f)}}$ is the point mass at $\bu_{t(f)}$.
 Now adding and subtracting $\prod_{e\in\gamma(f)} \parens*{\parens*{(1-p)^{m(e)} - (-p)^{m(e)}} + (-p)^{m(e)}} \cdot p^{|\gamma(f)|}$, we complete the first summation and from the triangle inequality we obtain the bound}
    &\qquad\qquad\le \abs*{ \prod_{e\in\gamma(f)} \parens*{p(1-p)^{m(e)} + (1-p)(-p)^{m(e)} } } \\ &\qquad\qquad\qquad+ \abs*{\prod_{e\in\gamma(f)} \parens*{(1-p)^{m(e)} - (-p)^{m(e)}} \cdot p^{|\gamma(f)|-1} \cdot \parens*{ \angles*{N(\bu_{s(f)}), X^{|\gamma(f)|-1} \delta_{\bu_{t(f)}} } - p} }. \numberthis \label{eq:inner-exp}
\end{align*}
We bound \pref{eq:inner-exp} based on the graphical properties of $\gamma(f)$.
\begin{definition}
    We say an edge $e$ is a \emph{singleton edge} if $m(e) = 1$ and a \emph{duplicative edge} otherwise.
\end{definition}
If $\gamma(f)$ contains any singleton edges, then the first term of \pref{eq:inner-exp} is $0$; otherwise it is bounded by 
\[
    \prod_{e\in\gamma(f)}(p(1-p)^2 + (1-p)p^2) \le \prod_{e\in\gamma(f)} p(1-p) \le p^{|\gamma(f)|}.
\]
The second term can always be bounded by
\[
    \prod_{e\in\gamma(f)} \parens*{(1-p)^{m(e)} + p^{m(e)}} \cdot p^{|\gamma(f)|-1} \cdot \abs*{ \angles*{N(\bu_{s(f)}), X^{|\gamma(f)|-1} \delta_{\bu_{t(f)}} } - p} \le p^{|\gamma(f)|-1} \cdot \abs*{ \angles*{N(\bu_{s(f)}), X^{|\gamma(f)|-1} \delta_{\bu_{t(f)}} } - p}.
\]
Using $D_J(W)$ to denote the collection of edges $f$ in $G_J$ such that $\gamma(f)$ contains no singleton edges, and $S_J(W)$ to use the collection of edges $f$ in $G_J$ such that $\gamma(f)$ contains a singleton edge, and plugging the above bounds into \pref{eq:indep-paths} tells us:
\begin{align*}
    \pref{eq:indep-paths} \le& \E_{\substack{\bu_i\\i\in J(W)}} \prod_{f\in D_J(W)} p^{|\gamma(f)|-1} \cdot \parens*{ \abs*{ \angles*{N(\bu_{s(f)}), X^{|\gamma(f)|-1} \delta_{\bu_{t(f)}} } - p } + p }
    \cdot \\ 
    &\prod_{f\in S_J(W)} p^{|\gamma(f)|-1} \cdot \abs*{ \angles*{N(\bu_{s(f)}), X^{|\gamma(f)|-1} \delta_{\bu_{t(f)}} } - p}.
\end{align*}
If $G_J(W)$ were a tree, we could recursively take the expectation over leaf vertices to bound the quantity above, as we did to get rid of $G_1$. 
However, it is not a tree, so we'll pick an arbitrary spanning tree $T_J(W)$ of $G_J(W)$, and bound edges outside of the spanning tree directly.
For $f\in E_J(W)\setminus T_J(W)$, we use \pref{assumption:tv-bound} to conclude that $\dtv{X^{|\gamma(f)|-1}\delta_{\bu_{t(f)}}}{\rho} \le C \lambda^{|\gamma(f)|-1}$, which thus implies that
\[
    \abs*{ \angles*{N(\bu_{s(f)}), X^{|\gamma(f)|-1} \delta_{\bu_{t(f)}} } - p}
\le C \lambda^{|\gamma(f)|-1}, \numberthis\label{eq:non-tree-bound}
\]
because $\angles*{N(\bu_{s(f)}), X^{|\gamma(f)|-1}\delta_{\bu_{t(f)}}}$ represents the probability that a point sampled at random from the measure $X^{|\gamma(f)|-1}\delta_{\bu_{t(f)}}$ lands in $N(\bu_{s(f)})$, which is a set of measure $p$ under $\rho$.
We now prove the following by induction.
\begin{claim}   \label{claim:tree-bound}
    We have the following bound on the contribution of $f \in T_{J}(W)$:
\begin{align*}
    \E_{\bu_i:i\in J(W)} \prod_{f\in T_J(W)} p^{|\gamma(f)|-1} \cdot \parens*{ \abs*{ \angles*{N(\bu_{s(f)}), X^{|\gamma(f)|-1} \delta_{\bu_{t(f)}} } - p} + p \cdot\Ind[f\in D_J(W)] } \\
    \le \prod_{f\in T_J(W)} p^{|\gamma(f)|} \cdot \parens*{ 2C\lambda^{|\gamma(f)|} + \Ind[f\in D_J(W)]}.
\end{align*}
\end{claim}
\begin{proof}
We fix an order for $i\in J(W)$, $i_0,\dots,i_t$ such that $i_j$ is a leaf in $T_J^{(j)}(W)$, the graph obtained by taking $T_J(W)$ and deleting $i_{j+1},\dots,i_{t}$.
We use $f_j$ to denote the unique edge incident to $i_j$ in $T_J^{(j)}(W)$.
Then if we define
\[
    a_j \coloneqq \E_{\bu_i:i\in V(T_J^{(j)}(W))} \prod_{f\in T_J^{(j)}(W) } p^{|\gamma(f)|-1} \cdot \parens*{ \abs*{ \angles*{N(\bu_{s(f)}), X^{|\gamma(f)|-1} \delta_{\bu_{t(f)}} } - p} + p \cdot\Ind[f\in D_J(W)] }
\]
Because $f_j$ is independent of $f_{j'}$ for $j' < j$ we can write:
\begin{align*}
    a_j \coloneqq \E_{\bu_{i_0}} \cdots &\E_{\bu_{i_{j-1}}} \prod_{f\in T^{(j-1)}_J(W)} p^{|\gamma(f)|-1} \cdot \parens*{ \abs*{ \angles*{N(\bu_{s(f)}), X^{|\gamma(f)|-1} \delta_{\bu_{t(f)}} } - p} + p \cdot\Ind[f\in D_J(W)] } \cdot \\
    &\E_{\bu_{i_j}} p^{|\gamma(f_j)|-1} \cdot \parens*{ \abs*{ \angles*{N(\bu_{s(f_j)}), X^{|\gamma(f_j)|-1} \delta_{\bu_{t(f_j)}} } - p} + p \cdot\Ind[f_j\in D_J(W)] }
\end{align*}
Without loss of generality we can assume $i_j = t(f_j)$, and because $N(\bu_{s(f_j)}) = p X \delta_{\bu_{s(f_j)}}$,
\begin{align*}
    \E_{\bu_{i_j}} \abs*{ \angles*{N(\bu_{s(f_j)}), X^{|\gamma(f_j)|-1} \delta_{\bu_{t(f_j)}}} - p} 
    &=\E_{\bu_{i_j}} \abs*{ \angles*{N(\bu_{s(f_j)}), X^{|\gamma(f_j)|-1} \delta_{\bu_{i_j}}} - p} \\
&= p \E_{\bu_{i_j}} \abs*{ \angles*{X^{|\gamma(f_j)|} \delta_{\bu_{s(f_j)}}, \delta_{\bu_{i_j}}} - 1 } \\
    &= 2p \cdot \dtv{X^{|\gamma(f_j)|} \delta_{\bu_{s(f_j)}} }{\Unif}\\
    &\le p\cdot 2C \lambda^{|\gamma(f_j)|}.
\end{align*}
This gives us the inequality:
\[
    \alpha_j \le \alpha_{j-1} \cdot p \parens*{2C \lambda^{|\gamma(f_j)|} + \Ind[f_j\in D_J(W)] }.
\]
The above inequality combined with the fact that $\alpha_0 = 1$ yields the claim.
\end{proof}

We use $e(W)$ to denote $|E(W)|$ and $\sing(W)$ to denote the number of singleton edges in $G_2(W)$,\footnote{Note $\sing(W)$ is the same as the number of singleton edges in $G(W)$ since $G_1(W)$ cannot have singleton edges, as it is the multigraph induced by a closed walk of length $\ell$.}.  
For any graph $H$ we use $\exc(H)$ to denote the \emph{excess} of $H$, which is $|E(H)|-|V(H)|+1$, the number of edges $H$ has over a tree.
\begin{observation} \label{obs:exc-conservation}
    \(\exc(G(W)) = \exc(G_2(W)) = \exc(G_J(W))\).  Thus, we denote this quantity as $\exc(W)$.
\end{observation}

\begin{observation} \label{obs:minor-exc-ineq}
    \(|E_J(W)| \le 3\exc(W)\).
\end{observation}
\begin{proof}
    We use \pref{obs:exc-conservation} to write:
    \begin{align*}
        2\exc(W) - 2 = 2|E_J(W)| - 2|V_J(W)| &= \sum_{v\in V_J} (\deg_{G(J)}(v) - 2) \ge |V_J(W)|-1,
    \end{align*}
where the degree a self-loop incurs on a vertex is $2$, and the $-1$ on the right-hand side is to capture the possibility that $|J(W)| = 1$ when $G_2(W)$ has no degree-$3$ vertices.
    Adding $\exc(W)$ to both sides gives:
    \[
        3\exc(W) \ge |E_J(W)|.     \qedhere
    \]
\end{proof}
Using the bound on the non-tree edges from \pref{eq:non-tree-bound} and \pref{claim:tree-bound}, we get:
\begin{align*}
    \pref{eq:indep-paths} 
	&\le \prod_{f\in T_J(W)} p^{|\gamma(f)|} \cdot \parens*{ 2C \lambda^{|\gamma(f)|} + \Ind[f\in D_J(W)] } \cdot \prod_{f\in E_J(W)\setminus T_J(W)} p^{|\gamma(f)|-1} \cdot \parens*{ C \lambda^{|\gamma(f)|-1} + p\cdot\Ind[f\in D_J(W)] }
\intertext{
Now, we bound separately the contribution of singleton and duplicative edges. 
For each $f \in S_J(W)$, we pull out a factor of $(p\lambda)^{|\gamma(f)|}2C$ if the edge was in the tree, and a factor $(p\lambda)^{|\gamma(f)|-1}C$ if the edge was not in the tree; this fully accounts for the contributions of singleton edges.
For each $f \in D_J(W)$, we upper bound its contribution by $p^{|\gamma(f)|}3C $ if the edge was in the tree, and a factor $p^{|\gamma(f)|-1}3C $ otherwise; this is potentially loose because we don't keep the factors of $\lambda$, but it is a valid upper bound because $C \ge 1$ and $p,\lambda \le 1$.
We thus have a factor of $p$ from $|E_2(W)| - \exc(W)$ edges, a factor of $\lambda$ from $\sing(W) - \exc(W)$ edges, and a factor of at most $3C$ from each edge in $E_J(W)$.
Summarizing, 
}
    &\le p^{|E_2(W)|-\exc(W)} \lambda^{\sing(W)-\exc(W)} \cdot \parens*{3C}^{|E_J(W)|},
    \intertext{and by \pref{obs:minor-exc-ineq}, the above is bounded by:}
    &\le p^{|E_2(W)|-\exc(W)}\lambda^{\sing(W)-\exc(W)}\cdot\parens*{3C}^{3\exc(W)}.
\end{align*}
Since $m(e)\ge 2$ for every edge in $e\in E_1$ (otherwise the walk cannot be closed), by an analysis identical to that of the first term of \pref{eq:inner-exp}, we have:
\begin{align*}
    \pref{eq:sep-2-core} &\le p^{|E_1(W)|} \cdot p^{|E_2(W)|-\exc(W)}\lambda^{\sing(W)-\exc(W)}\cdot\parens*{3C}^{3\exc(W)} = p^{e(W)-\exc(W)} \lambda^{\sing(W)} \parens*{\frac{27C^3}{\lambda}}^{\exc(W)}.
\end{align*}
Finally, we can bound the trace power \pref{eq:walk-weight} as follows.
\begin{align*}
    \pref{eq:walk-weight} &\le \sum_{W\in\cW_{\ell}} p^{e(W)-\exc(W)} \lambda^{\sing(W)} \parens*{\frac{27C^3}{\lambda}}^{\exc(W)} \\
    &= \sum_{a=1}^{\ell} \sum_{b=1}^{\ell} \sum_{c=1}^{\ell} \sum_{\substack{W\in\cW_{\ell} \\ e(W)=a,\, \sing(W)=b,\, \exc(W)=c}} p^{a-c}\lambda^{b} \parens*{\frac{27C^3}{\lambda}}^c \\
    &= \sum_{a=1}^{\ell} \sum_{b=1}^{\ell} \sum_{c=1}^{\ell} p^{a-c}\lambda^{b} \parens*{\frac{27C^3}{\lambda}}^c \cdot \abs*{\braces*{W\in \cW_{\ell}: e(W)=a,\, \sing(W)=b,\, \exc(W)=c }} \numberthis \label{eq:reduce-to-counting}
\end{align*}
To finish bounding the trace power, it remains to count length-$\ell$ closed walks with a specified number of edges, excess edges, and singleton edges.
\begin{claim}   \label{claim:walk-count}
    The number of walks $W$ such that $e(W)=a$, $\sing(W)=b$, and $\exc(W)=c$ is at most:
    \[
        n^{a-c+1}\cdot\ell^{2(\ell-b)} \cdot \ell^{2c}.
    \]
\end{claim}
\begin{proof}
    Observe that $W$ has $a-c+1$ vertices.
    Then the following information about $W$ is sufficient to reconstruct it:  
    \begin{itemize}
        \item The labels of the visited vertices in $[n]$ in the order in which they are visited.
        \emph{There are at most $n^{a-c+1}$ labelings.}
        \item The timestamps when the edge walked on is not a singleton edge.
        \emph{There are at most $\ell^{\ell-b}$ possibilities.}
        \item The timestamps when $W$ takes a step $uv$ such that the edge $\{u,v\}$ has not been previously covered by $W$, but $v$ has been previously visited, along with the timestamp of when $v$ was visited for the first time.
        \emph{There are $c$ such steps, and hence there are at most $\ell^{2c}$ possibilities.}
        \item The timestamps when $W$ takes a step $uv$ such that the edge $\{u,v\}$ has been previously covered by $W$ along with the timestamp of when $\{u,v\}$ was covered the first time.
        \emph{There are at most $\frac{\ell-b}{2}$ such steps, and hence there are at most $\ell^{\ell-b}$ possibilities.}
    \end{itemize}
    Putting the above bounds together completes the proof.
\end{proof}

\begin{observation}     \label{obs:edge-bound}
    Any walk with $b$ singleton edges and $c$ excess edges has at most $\frac{\ell+b}{2}$ edges.
\end{observation}
\begin{proof}
    Each nonsingleton edge must be visited at least twice. 
	There are at most $\ell-b$ non-singleton steps. 
	So, there are at most $\frac{\ell-b}{2}$ nonsingleton edges, and the total number of edges is at most $\frac{\ell+b}{2}$.
\end{proof}
Now we can continue bounding the trace power.
\begin{align*}
    \pref{eq:reduce-to-counting} &\le \sum_{a=1}^{\ell} \sum_{b=1}^{\ell} \sum_{c=1}^{\ell} p^{a-c}\lambda^b \parens*{\frac{27C^3}{\lambda}}^c \cdot n^{a-c+1} \cdot \ell^{3(\ell-b)}\cdot \ell^{2c} \\
    &= n \sum_{a=1}^{\ell} \sum_{b=1}^{\ell} \sum_{c=1}^{\ell} (pn)^{a} \lambda^b \parens*{\frac{27C^3\ell^2}{\lambda pn }}^c \cdot \ell^{2(\ell-b)} \\
    &\le n \ell \cdot \max\braces*{1, \parens*{\frac{27C^3\ell^2}{\lambda pn}}^{\ell} } \cdot \sum_{a=1}^{\ell} \sum_{b=1}^{\ell} (\lambda pn)^b \cdot (pn)^{a-b} \cdot \ell^{2(\ell-b)}
    \intertext{By \pref{obs:edge-bound} and the assumption on $\lambda$ from \pref{assumption:tv-bound}, we can bound the total edges $a$ and hence the below.}
    &\le n\ell\cdot \max\braces*{1, \parens*{\frac{27C^3\ell^2}{\sqrt{pn}}}^{\ell} } \cdot \sum_{a=1}^{\ell} \sum_{b=1}^{\ell} \parens*{\lambda pn}^b \cdot (pn)^{\frac{\ell-b}{2}} \cdot \ell^{2(\ell-b)} \\
    &\le n\ell^3 \cdot \max\braces*{1, \parens*{\frac{27C^3\ell^2}{\sqrt{pn}}}^{\ell} } \max\braces*{(\lambda pn)^{\ell}, (pn\ell^4)^{\ell/2}}
\end{align*}
By \pref{claim:trace-method},
\[
\Pr\left[\|\ol{A}_{\bG}\| \ge e^{\eps} \cdot \left(n^{1/\ell}\ell^{3/\ell}\right)\max\left\{1,\frac{27C^3\ell^2}{\sqrt{pn}}\right\} \max\left\{\lambda pn, \sqrt{pn}\ell^2 \right\}\right] \le \exp(-\eps \ell),
\]
and choosing $\ell=\log^2 n$, $\eps = \log\log n / \log n$, for any constant $\gamma$, we get:
\begin{equation}
    \norm*{A_{\bG} - \E A_{\bG}} \le (1+o(1)) \cdot \parens*{1+\frac{27C^3\log^4n}{\sqrt{pn}}} \cdot \max\braces*{ \lambda pn, \sqrt{pn}\log^4 n }\label{eq:final-trace}
\end{equation}
with probability at least $1-n^{-\gamma}$.
Now, by the assumption of the theorem, $pn \gg C^6\log^8 n$, so $1 + \frac{27C^3 \log^4 n}{\sqrt{pn}} = 1+o(1)$.
\end{proof}

\section{Analyzing the discrete walk with Brownian motion}  \label{sec:bm-to-discrete}
In this section we quantify the extent to which convolving a measure $\alpha$ over $\bbS^{d-1}$ with a spherical cap of measure $p$ brings $\alpha$ closer to uniform, provided that $\alpha$ satisfies a certain monotonicity property.
We now define this monotonicity property, establishing a couple of additional definitions along the way.

\begin{definition}
We say a distribution on $\bbS^{d-1}$ with relative density $\alpha$ is {\em symmetric about $y \in \bbS^{d-1}$} if there exists a function $\ell_{\alpha}:[-1,1] \to \R$ such that $\alpha(z) = \ell_{\alpha}(\iprod{z,y})$.
We note that $\ell_{\alpha}$ is also the density $\alpha$ projected onto the \ul{line} defined by $y$ relative to the projection of the uniform distribution, so that
\[
\ell_{\alpha}(t) 
= \frac{\int_{\bbS^{d-1}} \Ind[\iprod{z,y} = t] \cdot \ell_\alpha(t)\, d\rho(z)}{\int_{\bbS^{d-1}} \Ind[\iprod{z,y} = t] \, d\rho(z)}
= \frac{\int_{\bbS^{d-1}} \Ind[\iprod{z,y} = t] \cdot \alpha(z)\, d\rho(z)}{\int_{\bbS^{d-1}} \Ind[\iprod{z,y} = t] \, d\rho(z)}.
\]
\end{definition}
Notice that $\ell_\rho = 1$.
\begin{definition}
A measure $\alpha$ over $\bbS^{d-1}$ which is symmetric about some $y \in \bbS^{d-1}$ is said to be {\em spherically monotone} if $\ell_\alpha$ is monotone non-decreasing.
\end{definition}
An alternate characterization of spherically monotone distributions is that their relative densities can be written as a non-negative combination of spherical caps.
Recall that we use $\mcap_{p}(y)$ and $\mcap_{\ge\tau(p)}(y)$ interchangeably to denote the uniform measure over $\scap_p(y)$.
\begin{claim}  \torestate{ \label{claim:pancake-basis}
    A density $\alpha:\bbS^{d-1}\to \R$ which is symmetric about $y \in \bbS^{d-1}$ is spherically monotone if and only if there is a distribution $r$ on $[-1,1]$ such that:
    \[
        \alpha = \int \mcap_{\ge \theta} \, dr(\theta).
    \]
    We call the above way of writing $\alpha$ as the \emph{cap decomposition} of $\alpha$.
Further, $\ell_\alpha = \int \ell_{\mcap_{\ge\theta}}\, dr(\theta)$.}
\end{claim}
We give the straightforward \hyperlink{proof:pancake-basis}{proof} later.
Notice that in writing the expression for $\ell_\alpha$ we have replaced $\ell_{\mcap_{\ge\theta}(y)}$ with $\ell_{\mcap_{\ge\theta}}$; this is because $\ell_{\mcap_{\ge\theta}(y)}$ does not depend on $y$.

\begin{definition}
Given a measure $\mu$ over $\bbS^{d-1}$ which is symmetric about some $y \in \bbS^{d-1}$, its \emph{spherical kernel} $P_{\mu}$ is the transition operator of the random walk on $\bbS^{d-1}$ where a single step, starting from $x\in \bbS^{d-1}$, samples $\ba \sim \ell_\mu$ and then walks from $x$ to a uniformly random $\bw \in \bbS^{d-1}$ satisfying $\iprod{\bw,x} = \ba$.
Equivalently, the density of $P_\mu \alpha$ is $\mu * \alpha$ for $*$ denoting convolution.
\end{definition}

\begin{remark}
    For brevity, we will use $P_p$ as a shorthand for $P_{\mcap_p}$.
\end{remark}

\noindent The main result of this section, proved \hyperlink{proof:decay}{after} developing some tools, is the following:
\begin{theorem}\label{thm:decay}
If a probability distribution $\alpha$ over $\bbS^{d-1}$ is symmetric and spherically monotone, then for any integer $k \ge 0$,
\[
    \dtv{P_p^k \alpha}{\rho} \le \left((1+o_{d\tau^2}(1))\cdot\tau\right)^k \cdot \sqrt{\tfrac{1}{2}\KL{\alpha}},
\]
where $o_{d\tau^2}(1)$ denotes a function that goes to $0$ as $d\tau^2 \to \infty$.
\end{theorem}
\noindent As an immediate corollary, we obtain the following version which can be used in conjunction with \pref{thm:random-restriction} to conclude a bound on the second eigenvalue of random geometric graphs.
\begin{corollary}   \label{cor:decay-useful}
    For any probability distribution $\alpha$ over $\bbS^{d-1}$,
    \[
        \dtv{P_p^k \alpha}{\rho} \le \parens*{\parens*{1+o_{d\tau^2}(1)}\cdot \tau}^{k-1} \cdot \sqrt{\frac{1}{2}\cdot\log\frac{1}{p}}.
    \]
\end{corollary}
\begin{proof} 
We write $\alpha$ as a convex combination of (symmetric, spherically monotone) point masses $\delta_x$.
Then we apply \pref{thm:decay} in conjunction with the triangle inequality and the fact that $P_p^k \delta_x = P_p^{k-1}\mcap_p(x)$ and $\KL{\mcap_p(x)} = \log \frac{1}{p}$.
\end{proof}

Our proof of \pref{thm:decay} will relate the action of $P_p$ to the action of the {\em Brownian motion kernel}.
\begin{definition}[Brownian motion on $\bbS^{d-1}$]
    Let $(\bB_t)_{t\ge 0}$ be standard Brownian motion in $\R^d$.
    We define Brownian motion on $\bbS^{d-1}$ starting at some point $V_0\in\bbS^{d-1}$ as the process $(\bV_t)_{t\ge 0}$ via the following stochastic differential equation:
    \[
        d\bV_t = \sqrt{2} \parens*{\Id - \bV_t \bV_t^{\top}}\, d\bB_t - (d-1)\bV_t \, d t.
    \]
\end{definition}

\begin{definition}
For any $t \ge 0$, let the {\em time-$t$ Brownian motion kernel} $U_t$ be the transition operator of a random walk on $\bbS^{d-1}$ where a single step samples runs a time-$t$ Brownian motion on the sphere.
Equivalently, $U_t = P_{\beta_t}$ for $\beta_t$ the (spherically symmetric) density of a $t$-step Brownian motion.
\end{definition}

For any $y \in \bbS^{d-1}$, $P_p y$ is highly concentrated near the boundary of the cap of measure $p$ around $y$.
As we will show in \pref{sec:bm}, the same is true for $U_t y$; it is highly concentrated near the boundary of a cap of measure $q = q(t)$ around $y$.
So, choosing $T>0$ so that $q(T) \approx p$, we will argue that $U_{T}$ and $P_p$ have similar action on spherically monotone measures.

We can then take advantage of the contractive properties of $U_T$ in order to prove that $P_p$ is contractive.
The Brownian motion kernel satisfies the following mixing condition  (which can be obtained, e.g., as a corollary of \cite[Theorem 5.2.1]{BGL14} and \cite[Corollary 2]{DEKL14}):
\begin{theorem}[{Mixing of Brownian motion on $\bbS^{d-1}$}] \label{lem:log-Sobo-to-KL}
    For any probability distribution $\phi$ on $\bbS^{d-1}$,
    \[
        \dkl(U_t\phi\,\|\,\Unif) \le \exp(-2(d-1)t)\cdot \dkl(\phi\,\|\,\Unif)
    \]
\end{theorem}
As a corollary of the above and Pinsker's inequality, for any $t > 0$ and measure $\alpha$ over $\bbS^{d-1}$,
\begin{align}
    2\left(\dtv{U_t\alpha}{\rho}\right)^2 \le \dkl(U_{t}\alpha\,\|\,\rho) \leq \exp(-2(d-1)t)\cdot \dkl(\alpha\,\|\,\rho).\label{eq:log-sob-cor}
\end{align}

Armed with \pref{eq:log-sob-cor}, we can pass to working exclusively with the 1-dimensional projection of the measures in question onto the direction $y$.
\begin{claim}\label{claim:marginal}
    For any spherically symmetric distribution with relative density $\gamma$, $\dtv{\gamma}{\rho} = \dtv{\ell_\gamma}{\ell_\rho}$.
\end{claim}
\begin{proof}
    We express the total variation distance in terms of the $\ell_1$ norm:
    \begin{align*}
        2\dtv{\gamma}{\rho}
        = \int_{z\in\bbS^{d-1}} \left|\gamma(z) -1\right| \, d\rho(z)
        &= \int_{z\in\bbS^{d-1}}  \left|\ell_\gamma(\iprod{z,y}) -1\right| \, d\rho(z)\\
        &= \int_{t\in[-1,1]} \left|\ell_{\gamma}(t) -1\right| \, d\ell_\rho(t)
        = 2\dtv{\ell_\gamma}{\ell_\rho}.\qedhere
        \end{align*}
\end{proof}

Note that if $\alpha$ is spherically symmetric about $y$ then so is $U_t \alpha$, by the rotational invariance of Brownian Motion on the sphere. 
Hence combining \pref{claim:marginal} with \pref{eq:log-sob-cor}, we have that
\begin{align*}
\dtv{\ell_{U_t\alpha}}{\ell_\rho} \le \sqrt{\frac{1}{2} \cdot \exp(-(d-1)t)\cdot \dkl(\alpha\|\rho)}.
\end{align*}

Now, we'll show that for a well-chosen $T>0$, $\ell_{U_T\alpha}$ nearly stochastically dominates $\ell_{P_p \alpha}$, and that $P_p \alpha$ and $U_T \alpha$ are both spherically monotone, and that this furthermore implies that $\dtv{\ell_{U_T\alpha}}{\ell_\rho}$ and $\dtv{\ell_{P_p \alpha}}{\ell_\rho}$ are related.
Specifically, we show the following lemmas:

\begin{lemma}  \torestate{ \label{lem:stoc-dom-tv}
    If $\nu$ and $\mu$ are spherically monotone densities and $\ell_\nu\sdl \ell_\mu$, then\footnote{As will be apparent from the proof, one may replace $\ell_\nu,\ell_\mu$ with any monotone non-decreasing densities on $[-1,1]$.}
    \[
        \dtv{\ell_\nu}{\ell_\rho} \le \dtv{\ell_\mu}{\ell_\rho}.
    \]}
\end{lemma}
We prove the lemma \hyperlink{proof:stoc-dom-tv}{below}, but intuitively, a spherically monotone distribution can be realized as a non-negative combination of spherical caps; the uniform distribution has all of its mass on the largest cap (of measure $1$).
If $\ell_\nu \sdl \ell_\mu$, then the total probability mass within any radius $\theta$ of the mode of $\mu$ exceeds that of $\nu$, witnessing a larger total variation distance.

\begin{lemma}\torestate{\label{lem:op-act}
Let $\mu,\nu,\alpha$ be spherically monotone densities over $\bbS^{d-1}$, with $\ell_\nu \sdl \ell_\mu$.
Then
\begin{enumerate}
\item \label{part:spherical-mon} $P_\mu \alpha$ is spherically monotone (as is $P_\nu \alpha$),
\item \label{part:stoc-dom-dists} $\ell_{P_{\alpha} \nu} \sdl \ell_{P_{\alpha} \mu}$, and
\item \label{part:stoc-dom-ops} $\ell_{P_\nu \alpha} \sdl \ell_{P_\mu \alpha}$.
\end{enumerate}}
\end{lemma}
We will prove this lemma \hyperlink{proof:op-act}{below} as well; the crux of the proof of \pref{part:spherical-mon} is to realize that because $\alpha,\mu$ are spherically monotone, they can be decomposed as a non-negative combination of spherical caps. 
Then, by linearity of $P_\mu$ and by the commutativity of convolution, \pref{part:spherical-mon} reduces to showing that the convolution of two spherical caps is spherically monotone (this is a statement that we find intuitive, and it is easy to verify by directly examining the expression for $\ell_{P_{\ge \theta}\mcap_{\ge \psi}}$). 
To show \pref{part:stoc-dom-dists}, we observe that by decomposing $\alpha$ in its cap decomposition, it is then enough to compare $\ell_{P_{\ge \theta}\nu}$ with $\ell_{P_{\ge \theta}\mu}$ for each $\theta$.
Here, when $\ell_\mu \sd \ell_\nu$, a straightforward coupling demonstrates that $\ell_{P_{\ge \theta}\mu} \sd \ell_{P_{\ge \theta}\nu}$.
\pref{part:stoc-dom-ops} is a consequence of \pref{part:stoc-dom-dists} and commutativity of convolution.

Our aim is to now apply these lemmas with $\nu \approx \mcap_p(y)$  and $\mu = \beta_T$ (note that $P_\nu \alpha = P_p \alpha$ and $P_\mu \alpha = U_T \alpha$).
We now verify that these densities meet the conditions above.
The density $\mcap_p(y)$ is spherically monotone because it is the same as $\rho$ conditoned on being closer to $y$; we now show that $\beta_t$ is indeed spherically monotone.
\begin{claim}\label{claim:bm-monotone}
The density of a time-$t$ Brownian motion, $\beta_t$, is spherically monotone.
\end{claim}
\begin{proof}
Since Brownian motion on $\bbS^{d-1}$ can be realized as a sequence of random steps within spherical caps of infintesimally small measure $ds$, the measure of a $t$-step Brownian motion starting from $y \in \bbS^{d-1}$ is achieved by iteratively applying $P_{\mcap_{ds}}$ to the point mass at $y$.
The proof is then complete by noting that $\ell_{\mcap_p}$ is spherically monotone for every $p$, then applying \pref{part:spherical-mon} of \pref{lem:op-act}.
\end{proof}

Next, we argue that for $T = T(p)$, there is some small $\delta$ for which $(1-\delta)\ell_{\mcap_p} + \delta \ell_\rho\sdl \ell_{\beta_{T}}$; that is, the linear projection of the $p$-cap is almost stochastically dominated by the linear projection of Brownian motion run for the proper amount of time.
In order to do this, we first establish that almost all of the probability mass of $\beta_{T}$ is in a cap of radius close to $p$.
In \pref{sec:bm}, we'll prove the following lemma:
\begin{lemma}
\torestate{
\label{lem:BM-time-s}
Let $(\bV_t)_{t \ge 0}$ be a Brownian motion on $\bbS^{d-1}$ starting at $V_0$.
Then for any time $t\ge 0$,
\[
\Pr\left[ 
\left|\iprod{V_0,\bV_t}-\exp\left(-(d-1)t\right)\right| \ge x \right] \le 2 \exp\left(-\tfrac{d-1}{2}\frac{x^2}{1-e^{-2(d-1)t}}\right).
\]
}
\end{lemma}
From this lemma, we can show that almost all of the mass of the cap decomposition of $\ell_{\beta_T}$ is contained inside a $(\ge \tau)$-cap:
\begin{claim} \label{claim:small-tail-sphere}
    Let $\ptau > 0$,
    $T \coloneqq \frac{1}{d-1} \parens*{\log \frac{1}{\ptau}-2\eps}$, and $\eps \in \left[0,\frac{1}{2}\log \frac{1}{\ptau}\right]$.
    Then the total mass of $\ell_{\beta_T}$ outside of $\scap_{\ge (1+\eps)\ptau}(V_0)$ for $V_0$ the starting point of the Brownian motion is bounded:
    \[
        \int_{-1}^{(1+\eps)\tau} d\ell_{\beta_T}(x)  \leq \delta(\eps) := 2\exp\left( - \frac{(d-1) \eps^2 \ptau^2}{2(1 - \ptau^2)} \right).
    \]
\end{claim}
\begin{proof}
    We let $(\bV_t)_{t\ge 0}$ be a Brownian motion on the sphere, $\bA_t = \iprod{\bV_t,V_0}$, and $\bA_t = \exp\left(-(d-1)t\right) + \bR_t$. 
    At time $T$, we have 
	\begin{align*}
		\bA_{T} = \exp\left( -(d - 1) \cdot T \right) + \bR_{T} 
		&= \ptau \cdot \exp\left( 2\eps \right) + \bR_{T} \\
		&\geq \ptau \cdot \left(1 + 2\eps \right) + \bR_{T}
		\geq \ptau \cdot \left(1 + 2\eps \right) + \bR_{T}.
	\end{align*}
	The event that $\bA_{T} \leq \ptau \cdot (1 + \eps)$ implies $\bR_{T} < -\eps \ptau $, so it suffices to upper bound the probability that $|\bR_{T}| > \eps \ptau$.
    Applying \pref{lem:BM-time-s},
    \[
        \Pr[|\bR_{T}| \geq \eps\ptau] 
        \leq 2\exp\left(- \tfrac{d-1}{2}\frac{\eps^2\ptau^2}{1 - e^{-2(d - 1) \cdot T}} \right)
        = 2\exp\left(- \frac{\eps^2 \ptau^2(d-1)}{2(1 - \ptau^2)} \right)
		\qedhere
    \]
\end{proof}

Now, we are ready to establish the stochastic domination of the combination.
\begin{claim}\label{claim:combo-dom}
    Let $p \in \left(0,\frac{1}{2}\right)$ and $\ptau = \tau(p) + \frac{4}{\sqrt{d}}$.
    For $T = \frac{1}{d-1}(\log \frac{1}{\ptau} - 2 \eps)$ with $\eps \in \bracks*{\frac{5}{(d-1)\ptau^2},\frac{1}{2}\log \frac{1}{\ptau}}$,
    \[
        \ell_{\beta_T}\sd (1 - 2\delta(\eps)) \ell_{\mcap_p} + 2\delta(\eps)\ell_\rho,
    \]
    for $\delta(\eps)$ as defined in the statement of \pref{claim:small-tail-sphere}.
\end{claim}

\begin{proof}
    Using \pref{claim:pancake-basis}, we write
		$\beta_T = \int_{-1}^{1} c_\theta \cdot \ell_{\mcap_{\ge \theta}} d\theta $, with $\int c_\theta d\theta = 1$.
    Let $\tau' \in [-1,1]$ be such that
	\begin{equation}
        \int_{-1}^{\tau'} c_\theta d\theta = 2\delta(\eps),\quad
	    \text{ and } \quad
        \int_{\tau'}^{1} c_\theta d\theta = 1 - 2\delta(\eps) .\label{eq:sums}
	\end{equation}
    The proof strategy is to show that the conclusion follows if $\tau'\ge\tau(p)$, and then establish that inequality.

    First observe that if $\alpha$ and $\{\gamma_x\}_{x \in X}$ are measures satisfying $\gamma_x \sd \alpha$ for all $x \in X$, then a convex combination $\int c_x \gamma_x dx \sd \alpha$ as well, from which the conclusion follows.
    Now, writing
	\[
        \beta_T = \int_{-1}^{\tau'} c_\theta \cdot \ell_{\mcap_{\ge \theta}} d\theta + \int_{\tau'}^1 c_\theta \cdot \ell_{\mcap_{\ge \theta}} d\theta,
    \]
    we see that the first term on the right-hand-side stochastically dominates 
    $2\delta(\eps)\cdot\ell_{\mcap_{\ge -1}} = 2\delta(\eps)\cdot \ell_\rho$ since for every $\theta\in[-1,\tau']$, $\theta\ge-1$ and therefore $\ell_{\mcap_{\ge \theta}}\sd \ell_{\mcap_{\ge -1}} = \ell_\rho$.
    By identical reasoning, the second term stochastically dominates $\ell_{\mcap_{\ge \tau(p)}}$ since for every $\theta \in [\tau',1]$, $\theta \ge \tau(p)$ and therefore $\ell_{\mcap_{\ge \theta}} \sd \ell_{\mcap_{\ge \tau}} = \ell_{\mcap_p}$. 

    We now show that the $\tau'$ satisfying \pref{eq:sums} is at least $\tau$, for which it is sufficient to show $\tau'\ge\ptau$.
    Let $\kappa = \int_{-1}^\ptau c_\theta \, d\theta$; $\tau' \ge \ptau$ is equivalent to showing that $\kappa \le 2\delta(\eps)$.
    Using \pref{claim:small-tail-sphere}, we know that $\Pr_{\bv \sim \beta_T}[\bv \in\scap_{\ge (1+\eps)\ptau}(V_0)]\ge 1-\delta(\eps)$.
    \begin{align*}
        1-\delta(\eps)
        &\le \Pr_{\bv \sim \beta_T}\left[\bv \in \scap_{\ge (1+\eps)\ptau}(V_0)\right]\\
        &= \int_{-1}^{(1+\eps)\ptau} c_\theta \cdot \Pr_{\bx \sim \ell_{\mcap_{\ge \theta}}}[\bx \ge (1+\eps)\ptau]\, d\theta + \int_{(1+\eps)\ptau}^1 c_\theta\, d\theta\\
        &\le \int_{-1}^{\ptau} c_\theta \cdot \Pr_{\bx \sim \ell_{\mcap_{\ge \theta}}}[\bx \ge (1+\eps)\ptau]\, d\theta + \int_{\ptau}^1 c_\theta\, d\theta\\
        &\le \left(\max_{\theta \in [-1,\ptau]}\Pr_{\bx \sim \ell_{\mcap_{\ge \theta}}}[\bx \ge (1+\eps)\ptau]\right) \cdot \kappa + \int_{\ptau}^1 c_\theta \, d\theta = \Pr_{x\sim\ell_{\mcap_{\ge \ptau}}}\bracks*{\bx\ge(1+\eps)\ptau} \cdot \kappa + \int_{\ptau}^1 c_{\theta}\, d\theta.
    \end{align*}
    Using \pref{lem:approx-tails-dpd} and $\ptau\ge 4/\sqrt{d}$,
	\begin{align*}
        \Pr_{\bx \sim \ell_{\mcap_{\ge \ptau}}}\left[\bx \ge (1+\eps)\ptau\right]
        &= \frac{\Unif(\scap_{\ge (1+\eps)\ptau})}{\Unif(\scap_{\ge \ptau})}\\
        &\leq \frac{3 \ptau \left( 1 - ((1+\eps)\ptau)^2\right)^{(d - 1)/2}}{2 \ptau(1+\eps) \left(1 - \ptau^2 \right)^{(d - 1)/2}} \\
        &\leq \frac{3}{2} \cdot \left( \frac{1 - (1+\eps)^2\ptau^2}{1 - \ptau^2} \right)^{(d - 1)/2} \\
        &= \frac{3}{2} \cdot \left( 1 - \frac{2 \eps \ptau^2 + \eps^2 \ptau^2}{1 - \ptau^2}  \right)^{(d - 1)/2} 
        \leq \frac{3}{2} \cdot \left( 1 - 2\eps \ptau^2  \right)^{(d - 1)/2} 
        \leq \frac{3}{1 + (d - 1) \eps \ptau^2}.
	\end{align*}
	The final quantity is smaller than $\frac{1}{2}$ given our lower bound on $\eps$, and $\int_{\tau}^1 c_{\theta}\, d\theta = 1-\kappa$.
    Plugging into the above, we have that 
    \[
        1-\delta(\eps) \le \frac{1}{2}\kappa + 1 - \kappa \implies \kappa \le 2\delta(\eps),
    \]
    which completes the proof.
\end{proof}

Finally, we will need the following claim to transfer the statement about the stochastic domination of a linear combination of $\ell_{P_p\alpha}$ and $\ell_\rho$ to just $\ell_{P_p\alpha}$:
\begin{lemma}   \label{lem:TV-relation}
    Suppose $\mu$ and $\nu$ are spherically monotone distributions, then for any $\eta\in[0,1)$,
    \[
        \dtv{\ell_\mu}{\ell_\rho} \le \frac{1}{1-\eta}\dtv{(1-\eta) \ell_\mu + \eta \ell_\nu}{\ell_\rho}.
    \]
\end{lemma}
\begin{proof}
    Let $s\in[-1,1]$ be such that:
    \[
        \dtv{\ell_\mu}{\ell_\rho} = \int_{s}^1 (\ell_\mu(x)-1) d\rhoflat = \int_{-1}^s (1 - \ell_\mu(x))d\rhoflat,
    \]
where $\rhoflat$ is the density of the $1$-dimensional projection of $\mcap_1(y)$.
    The choice of $s$ satisfying the above is the one satisfying $\ell_\mu(s)= 1$. 
 If $\ell_\nu(s)\ge 1$, by spherical monotonicity $\ell_\nu(x)\ge 1$ on $[s,1]$ and:
    \begin{align*}
        \dtv{(1-\eta)\ell_\mu+\eta\ell_\nu}{\ell_\rho} 
&\ge \int_{s}^1 \left((1-\eta)(\ell_\mu(x)-1) + \eta(\ell_\nu(x)-1)\right) d\rhoflat \\
&\ge (1-\eta) \int_s^1 (\ell_\mu(x)-1) d\rhoflat = (1-\eta)\cdot\dtv{\ell_\mu}{\ell_\rho}.
    \end{align*}
    On the other hand, if $\ell_\nu(s)\le 1$, by an identical argument we know:
    \[
        \dtv{(1-\eta)\ell_\mu+\eta \ell_\nu}{\ell_\rho} \ge (1-\eta)\int_{-1}^s (1-\ell_\mu(x)) d\rhoflat = (1-\eta)\cdot\dtv{\ell_\mu}{\ell_\rho}.
    \]
    The desired statement follows from rearranging the above inequality.
\end{proof}

We are now ready to prove \pref{thm:decay}, following the reasoning above, in combination with induction on $k$, the number of applications of $P_p$.
We state and prove a more refined version of \pref{thm:decay} below.
\begin{theorem}
    If a probability distribution $\alpha$ over $\bbS^{d-1}$ is symmetric and spherically monotone, then for any integer $k \ge 0$
    and for $\ptau = \tau(p,d) + \frac{4}{\sqrt{d}}$,
    \[
        \dtv{P_p^k \alpha}{\rho} \le \ptau^k \parens*{ \frac{\exp\parens*{\frac{4}{(d-1)^{1/4}\sqrt{\ptau}}}}{\sqrt{1-2\exp(-\ptau\sqrt{d-1})}} }.
    \]
\end{theorem}
Note that when $\tau^2 d \to \infty$, the parenthesized term is $1+o(1)$ and $\nu = \tau \cdot (1+o(1))$.
\begin{proof}[Proof]    \hypertarget{proof:decay}{}
    Suppose $\tau(p) \ge 1-1/(d-1)^{1/4}$, then the statement is vacuously true.
    Thus, we assume from now on $\tau(p) < 1 - 1/(d-1)^{1/4}$.

    Let $\ptau = \tau(p)+4/\sqrt{d}$, let $t = \frac{1}{d-1}\left(\log\frac{1}{\ptau} - 2\eps\right)$, and $\delta = 2\exp\left(-\frac{(d-1)\eps^2 \ptau^2}{2(1-\ptau^2)}\right)$ for $\eps = \frac{\sqrt{2-2\ptau^2}}{(d-1)^{1/4}\sqrt{\ptau}}$; note that for $d$ sufficiently large, $\eps \in \left[\frac{5}{(d-1)\ptau^2}, \frac{1}{2}\log \frac{1}{\ptau}\right]$.
    For convenience's sake, define $P_{p,\delta} = (1-2\delta)P_p + 2\delta P_1$.
    We will prove that 
    \begin{align}
        \ell_{P_{p,\delta}^k \alpha} \sdl \ell_{U_t^k \alpha}, \label{eq:dom}
        \qquad \text{ and }\qquad U_t^k \alpha,\,\, P_{p,\delta}^k \alpha\quad \text{are spherically monotone}.
    \end{align}
    Given this, the proof of the theorem will follow: by the linearity of the projection onto the line defined by $y$, and by the commutativity of convolution,
    \[
        \ell_{P_{p,\delta}^k \alpha} = \sum_{j=0}^k (1-2\delta)^{k-j} (2\delta)^j \binom{k}{j} \ell_{P_p^{k-j}P_1^j\alpha },
    \]
    So from \pref{claim:marginal}, \pref{lem:TV-relation}, \pref{eq:dom}, and \pref{lem:stoc-dom-tv},
    \begin{align}
        \dtv{P_p^k \alpha}{\rho}=
        \dtv{\ell_{P_p^k\alpha}}{\ell_\rho} \le \frac{1}{(1-2\delta)^k}\dtv{\ell_{P_{p,\delta}^k\alpha}}{\ell_\rho} \le \frac{1}{(1-2\delta)^k}\dtv{\ell_{U_t^k\alpha}}{\ell_\rho}.\label{eq:combo}
    \end{align}
    Then we can apply \pref{claim:marginal} to get that
    \begin{align}
        \dtv{\ell_{U_t^k \alpha}}{\ell_\rho} = \dtv{U_t^k \alpha}{\rho},\label{eq:line}
    \end{align}
    and finally using that $U_t^k = U_{k \cdot t}$ in conjunction with \pref{lem:log-Sobo-to-KL}, we have that
    \begin{align}
        \dtv{U_t^k \alpha}{\rho} = \dtv{U_{k \cdot t}\alpha}{\rho} \le \sqrt{\frac{1}{2}\exp(-2(d-1)tk) \cdot \dkl(\alpha \|\rho)},\label{eq:kl}
    \end{align}
    So combining \pref{eq:combo}, \pref{eq:line}, and \pref{eq:kl}, we have that 
    \[
        \dtv{P_p^k\alpha}{\rho} \le \sqrt{\frac{1}{2(1-2\delta)^k}\exp(-2(d-1)tk) \cdot \dkl(\alpha \|\rho)}.
    \]
    In our case, $\delta = \exp(-\ptau\sqrt{d-1})$, $t = \tfrac{1}{d-1}\left(\log \frac{1}{\ptau} - \frac{\sqrt{2-2\ptau^2}}{(d-1)^{1/4}\sqrt{\ptau}} \right)$, so combining these estimates,
    \[
        \dtv{P_p^k \alpha}{\rho} \le \ptau^k \cdot \left(\frac{\exp\parens*{ \frac{4}{(d-1)^{1/4}\sqrt{\ptau}} }}{\sqrt{1-2\exp(-\ptau\sqrt{d-1})}}\right)^{k} \cdot \sqrt{\tfrac{1}{2}\dkl(\alpha\|\rho)},
    \]
    as desired.

    Now we prove \pref{eq:dom}.
    The proof is by induction on $k$; when $k=0$, there is nothing to prove.
    Suppose now that the statement holds true for $k$; we shall prove it for $k+1$.
    By \pref{claim:bm-monotone}, the density of a time-$t$ spherical Brownian motion $\beta_t$ is spherically monotone about its starting point, and clearly, any convex combination of caps is spherically monotone.
    Hence we can apply \pref{lem:op-act}, \pref{part:spherical-mon} in conjunction with the induction hypothesis to conclude that both $P_{p,\delta}^{k+1} \alpha = P_{p,\delta}(P_{p,\delta}^k \alpha)$ and $U_{t}^{k+1}\alpha = U_t(U_t^k \alpha)$ are spherically monotone, giving the second part of the induction hypothesis.

    By our induction hypothesis $U_t^k \alpha$ and $P_{p,\delta}^k \alpha$ are spherically monotone with $\ell_{U_t^k \alpha} \sd \ell_{P_{p,\delta}^k \alpha}$, and so we can apply \pref{lem:op-act}, \pref{part:stoc-dom-dists} in conjunction with \pref{claim:bm-monotone} to conclude that
    \[
        \ell_{ U_t^{k+1} \alpha} = \ell_{P_{\beta_t} (U_t^k \alpha)} \sd \ell_{P_{\beta_t}(P_{p,\delta}^k \alpha)},
    \]
    and then apply \pref{lem:op-act}, \pref{part:stoc-dom-ops} in conjunction with \pref{claim:combo-dom} to conclude that
    \[
        \ell_{P_{\beta_t} (P_{p,\delta}^k \alpha)} \sd \ell_{P_{p,\delta}(P_{p,\delta}^k \alpha)} = \ell_{P_{p,\delta}^k \alpha},
    \]
    completing the proof.
\end{proof}

Now, we fill in the proofs of the lemmas from above.
\restateclaim{claim:pancake-basis}
\begin{proof}[Proof of \pref{claim:pancake-basis}]  \hypertarget{proof:pancake-basis}{}
We first prove the ``only if'' direction.
Since $\alpha$ is spherically symmetric about $y$, $\alpha(v) = \ell_{\alpha}(\iprod{v,y})$.
Let $d\ell_{\alpha}$ be the distributional derivative of $\ell_\alpha$, and set $dr(\theta) = \rho(\scap_{\ge\theta}(y))\, d\ell_\alpha(\theta)$.
    \begin{align*}
        \int (\mcap_{\ge \theta}(y))(v)\, dr(\theta) &= \int \frac{\Ind[\iprod{v,y}\ge\theta]}{\rho(\scap_{\ge\theta}(y))} \cdot \rho(\scap_{\ge\theta}(y)) \, d\ell_\alpha(\theta) 
        = \int \Ind[\iprod{v,y}\ge\theta]\, d\ell_\alpha(\theta) 
        = \alpha(v).
    \end{align*}
    To see that the measure $dr$ indeed gives a probability distribution, first observe that $dr(\theta)\ge 0$ for every $\theta$ due to the monotonicity of $\ell_\alpha$, and next observe that
    \begin{align*}
        1 = \int_{v \in \bbS^{d-1}}  \alpha(v) \, d\rho(v)
 	&= \int_{v \in \bbS^{d-1}} \int_{-1}^1 (\mcap_{\ge\theta}(y))(v) \, dr(\theta)\, d\rho(v)\\
        &= \int_{-1}^1 \int_{v \in \bbS^{d-1}} (\mcap_{\ge \theta}(y))(v)\, d\rho(v)\, dr(\theta) 
        = \int_{-1}^1 dr(\theta).
    \end{align*}
    In summary, since $r$ is a positive measure which integrates to $1$, it is a probability distribution.
The claim regarding $\ell_\alpha$ follows because the line projection onto $y$ is a linear operation.

Now we prove the converse.
Suppose $\alpha = \int_{-1}^1 \mcap_{\ge \theta}(y)\,dr(\theta)$. 
By linearity of projection onto the line defined by $y$, $\ell_{\alpha} = \int_{-1}^1 \ell_{\mcap_{\ge \theta}}\, dr(\theta)$.
Since $\ell_{\mcap_{\ge \theta}}$ is monotone for every $\theta$, and a non-negative combination of monotone functions is monotone, $\ell_\alpha$ is also monotone, concluding the proof.
\end{proof}

\restatelemma{lem:stoc-dom-tv}
\begin{proof}[Proof of \pref{lem:stoc-dom-tv}]  \hypertarget{proof:stoc-dom-tv}{}
    First, observe that $\ell_\nu \sd \ell_\rho$ and $\ell_\mu \sd \ell_\rho$ by the assumption that $\mu,\nu$ are spherically monotone. 
Thus, $\ell_\mu\sd\ell_\nu\sd\ell_\rho$.
Further, if measures $a,b$ on $[-1,1]$ satisfy $a \sd b$, then their CDFs $G_a$ and $G_b$ satisfy $G_a(s) \le G_b(s)$ for every $s$. 
Hence,
    \begin{align*}
        G_{\ell_\mu}(s) \le G_{\ell_\nu}(s) \le G_{\ell_\rho}(s) & \quad \forall s\in[-1,1].
    \end{align*}
    By definition of the total variation distance, for any non-decreasing density $\gamma:[-1,1]\to\R$,
    \[
        \dtv{\gamma}{\ell_\rho} = \max_{s\in[-1,1]} G_{\ell_\rho}(s)- G_{\gamma}(s).
    \]
    Thus,
    \[
        \dtv{\ell_\nu}{\ell_\rho} = G_{\ell_\rho}(s^*)-G_{\ell_\nu}(s^*) \le G_{\ell_\rho}(s^*) - G_{\ell_\mu}(s^*) \le \dtv{\ell_\mu}{\ell_\rho},
    \]
    which completes the proof.
\end{proof}

We'll now prove \pref{lem:op-act}.
\restatelemma{lem:op-act}
\begin{proof}[Proof of \pref{lem:op-act}]   \hypertarget{proof:op-act}{}
We first prove \pref{part:spherical-mon}.
We can write $\alpha$ and $\mu$ in terms of their {\em cap decompositions} as shown in \pref{claim:pancake-basis}, $\alpha = \int_{-1}^1 \mcap_{\ge\theta}(y) \, dr(\theta)$ and $\mu = \int_{-1}^1  \mcap_{\ge \psi}(z)\, ds(\psi)$ for some $z \in \bbS^{d-1}$.
$P_\mu$ is a linear operator, so $P_\mu \alpha = \int P_\mu \mcap_{\ge \theta}(y) \, dr(\theta)$.
Further, by the commutativity of convolution, $P_\mu \mcap_{\ge \theta}(y) = P_{\ge \theta}\mu_y$, where $\mu_y$ denotes the version of $\mu$ centered at $y$.
Hence,
\[
P_\mu \alpha = \int P_\mu \mcap_{\ge\theta}(y) dr(\theta) = \int P_{\ge \theta} \mu_y \, dr(\theta) = \int \int P_{\ge \theta} \mcap_{\ge \psi}(y) \, ds(\psi)\, dr(\theta).
\]
Each $P_{\ge \theta} \mcap_{\ge \psi}(y)$ is clearly spherically symmetric about $y$.
Since the projection onto the line defined by $y$ is a linear operation, $\ell_{P_\mu \alpha} = \int \int \ell_{P_{\ge \theta} \mcap_{\ge \psi}} ds(\psi) dr(\theta)$, and because a non-negative combination of monotone functions is monotone, it suffices to prove that for any $\theta,\psi \in [-1,1]$, $\ell_{P_{\ge \theta} \mcap_{\ge \psi}}$ is monotone.
By definition,
\begin{align*}
\ell_{P_{\ge \theta}\mcap_{\ge \psi}(y)}(t) 
&= \E_{\bv \sim \rho}\left[ \left(P_{\ge \theta} \mcap_{\ge \psi}(y)\right)(\bv) \mid \iprod{\bv,y} = t\right]\\
&= \E_{\bv \sim \rho}\left[ \E_{\bw \sim \mcap_{\ge \theta}(\bv)} \left[\left(\mcap_{\ge \psi}(y)\right)(\bw) \right]\mid \iprod{\bv,y} = t\right]\\
&= \E_{\bv \sim \rho}\left[ \E_{\bw \sim \mcap_{\ge \theta}(\bv)} \left[\tfrac{\Ind[\iprod{\bw,y} \ge \psi]}{\rho(\scap_{\ge \psi})} \right]\mid \iprod{\bv,y} = t\right]
= \frac{\Pr_{\bv,\bw \sim \rho}\left[\iprod{\bw,y} \ge \psi \mid \iprod{\bw,\bv} \ge \theta, \iprod{\bv, y} = t \right]}{\Pr_{\bw \sim \rho}\left[\iprod{\bw,y} \ge \psi\right]}.
\end{align*}
This ratio is monotone increasing in $t$, completing the proof of $(1)$.

Now we show \pref{part:stoc-dom-dists}. 
\pref{claim:pancake-basis} shows that by the spherical monotonicity of $\alpha$, we can express $\alpha$ in its cap decomposition,
\[
\alpha = \int_{0}^1 \mcap_{q} \, dr(q),
\]
and now by the linearity of convolution, $P_\alpha = \int_{0}^1 P_{\mcap_{q}} \, dr(q)$, and $P_\alpha \mu = \int P_q \mu\, dr(q)$, $P_\alpha \nu = \int P_q \nu dr(q)$.
So, to show that $\ell_{P_\alpha\mu}\sd \ell_{P_\alpha\nu}$, it suffices to argue ``slice-by-slice'' that for every $q \in [0,1]$, $\ell_{P_q \mu} \sd \ell_{P_q \nu}$.

This follows from the following coupling argument:
we sample $(\bx,\by)$ from $(\ell_{P_q \mu},\ell_{P_q \nu})$ in a coupled manner as follows: first, sample $(\ba_\mu,\ba_\nu) \sim (\ell_\mu, \ell_\nu)$ in a coupled manner so that $\ba_\mu \ge \ba_\nu$; such a coupling is guaranteed because $\ell_\mu \sd \ell_\nu$. 
Next, choose $(\bv_\mu,\bv_\nu)$ at random in $\bbS^{d-1}$ conditioned on $\iprod{\bv_\mu,y} = \ba_\mu$ and $\iprod{\bv_\nu,y} = \ba_\nu$.
Now, let $\btheta_\mu$ be the random variable $\iprod{y,\bu_{\mu}}$ for $\bu_\mu \sim \mcap_q(\bv_\mu)$, and $\btheta_\nu = \iprod{y,\bu_{\nu}}$ for $\bu_\nu \sim \mcap_q(\bv_\nu)$.
Note that the marginal over $\btheta_\mu$ is $\ell_{P_q \mu}$ and the marginal over $\btheta_\nu$ is $\ell_{P_q \nu}$.	
The probability $\Pr[\theta_\mu > t]$ is proportional to the measure of the intersection of $\scap_{\ge t}(y)$ and $\scap_{q}(\bu_{\mu})$, and similarly the probability $\Pr[\theta_\nu > t]$ is proportional to the measure of the intersection of $\scap_{\ge t}(y)$ and $\scap_q(\bu_\nu)$.
By our choice of coupling, the angle between $\bu_\mu$ and $y$ is smaller than the angle between $\bu_\nu$ and $y$, so for every $t \in [-1,1]$,
\[
\Pr[\theta_\mu > t] \ge \Pr[\theta_\nu > t],
\]
and hence we may couple $\btheta_\mu$ and $\btheta_\nu$ so that $\btheta_\mu \ge \btheta_\nu$ always.
Taking $\bx = \btheta_\mu$ and $\by = \btheta_\nu$ in this coupling gives our conclusion.

Finally, observe that by the commutativity of convolution, $P_\mu \alpha = P_\alpha \mu$ and $P_\nu \alpha = P_\alpha \nu$, and so \pref{part:stoc-dom-ops} follows from \pref{part:stoc-dom-dists}.
\end{proof}

\subsection{Concentration of spherical Brownian Motion within a cap}
\label{sec:bm}
In this section, we study the concentration of Brownian Motion on $\bbS^{d-1}$ in the spherical cap around its starting point. 

\restatelemma{lem:BM-time-s}

\begin{proof}[Proof of \pref{lem:BM-time-s}]
Letting $\bA_t = \iprod{V_0, \bV_t}$ be the correlation of the motion at step $t$ with the starting point, $\parens*{\bB_t}_{t\ge 0}$ be standard Brownian motion on $\R^d$, $\parens*{\bB'_t}_{t\ge 0}$ be standard Brownian motion on $\R$, and $\theta = d-1$,
\begin{align*}
d\bA_t = \iprod{V_0, d\bV_t} 
&= -\theta \cdot \bA_t\, dt + \sqrt{2} \angles*{V_0, \parens*{\Id - \bV_t \bV_t^{\top}}\,d\bB_t } \\
&= -\theta\cdot \bA_t\, dt + \sqrt{2}\iprod{\parens*{\Id - \bV_t \bV_t^{\top}} V_0, d\bB_t} \\
&= -\theta\cdot \bA_t\, dt + \sqrt{2}\sqrt{1-\bA_t^2}\, d\bB'_t
\end{align*}

The solution to the deterministic differential equation $d x_t = -\theta x_t$ with initial condition $x_0 = 1$ is $x_t = \exp(-\theta t)$.
To this end, it's convenient to split $\bA_t$ up into a deterministic and a random part:
\[
\bA_t = \exp\left(-\theta t\right) + \bR_t,
\]
with the initial condition $R_0=0$.
Then via calculation,
\begin{align}
d \bR_t 
&= - \theta \bR_t dt + \sqrt{2}\sqrt{1-\bA_t^2}\, d\bB'_t.
\end{align}
We now relate $\bR_t$ to a stochastic process without drift, as is done, for example, in the analysis of the Ornstein-Uhlenbeck process.
Consider $\bR_t\exp(\theta t)$.
Note that
\begin{align*}
    d(\bR_t\exp(\theta t)) &= \exp(\theta t)\, d\bR_t + \bR_t\, \theta \exp(\theta t)\,dt \\
    &= -\bR_t\, \theta \exp(\theta t)\,dt + \sqrt{2}\exp(\theta t) \sqrt{1-\bA_t^2}\, d\bB'_t + \bR_t\, \theta \exp(\theta t)\,dt \\
    &= \sqrt{2}\exp(\theta t) \sqrt{1-\bA_t^2}\, d\bB'_t,
\end{align*}
a process without drift.

The following version of the Azuma--Hoeffding inequality will allow us to argue that this driftless process concentrates.
\begin{lemma}\torestate{\label{lem:cont-AH}
Let $(\bX_t)_{t\ge 0}\subset \R$ be a stochastic process adapted to the filtration $\calF_t$ with
$
\E[e^{r \,d \bX_t} \mid \calF_t] < \exp\left(r^2\sigma_t^2dt\right),
$
for all $t,r$.
Then for all $s,x > 0$,
\[
	\Pr[|\bX_s-\bX_0| \ge x] \leq 2\exp\left(\frac{-x^2}{4\int_0^s \sigma_t^2 dt}\right).
\]}
\end{lemma}
Versions of this lemma are known (c.f. \cite{Dembo96} and references therein), we include a proof in \pref{app:cont-ah} for completeness.

\noindent
We apply \pref{lem:cont-AH} to prove that
 \[\Pr\bracks*{|\bR_s|\ge x} \le 2\exp\left(-C'\cdot\frac{x^2 d\theta}{1-\exp(-\theta s)}\right).
\]
Indeed, $\bR_t \exp(\theta t)$ is a stochastic process without drift and satisfies that 
\begin{align*}
\E[\exp(r\, d(\bR_t \exp(\theta t)))] 
&= \E\left[\exp\left(\sqrt{2}r\exp(\theta t)\sqrt{1-\bA_t^2} \, d\bB'_t\right)\right]  \\
&\leq \exp\left(r^2\exp(2\theta t) \cdot \parens*{1 - \bA_t^2} \, dt\right)\\
&\le\exp\left(r^2\exp(2\theta t) \, dt\right),
\end{align*}
Since $\bA_t$ is real-valued.
So we can apply \pref{lem:cont-AH} to the process and derive that
\begin{align*}
\Pr[|\bR_s| \ge x] 
&= \Pr[|\bR_s \exp(\theta s)| \ge x\exp(\theta s)] \\
&\le  2 \exp\left(-\frac{x^2\exp(2\theta s)}{4\int_{0}^s\exp(2\theta t)dt}\right) \\
&= 2 \exp\left(-\frac{ \theta x^2}{2\cdot \parens*{1 - \exp(-2\theta s)} }\right),
\end{align*}
and plugging in $\theta = d-1$ concludes the proof.
\end{proof}

\section{The second eigenvalue of links}    \label{sec:link-eigs}
In this section we analyze the links of the random geometric complex.
Each link is a random geometric graph in a cap centered around some $w\in\bbS^{d-1}$ on $\boldm$ vertices where $\boldm\sim\Binom(n,p)$.
We are interested in obtaining a high probability bound on the second eigenvalue of $\wh{A}_{\bG}\coloneqq D_{\bG}^{-1/2}A_{\bG}D_{\bG}^{-1/2}$, the normalized adjacency matrix of link graph $\bG$, where $A_{\bG}$ and $D_{\bG}$ denote its adjacency matrix and diagonal degree matrix.
Since the number of vertices $\boldm$ concentrates well in our setting, throughout this section we treat the number of vertices $m$ as fixed and handle the variation in $\boldm$ in \pref{sec:wrapup}.
We also specialize the parameters to the regime relevant in proving \pref{thm:hdx} in \pref{sec:wrapup} --- in particular, the relationship between $n$, $p$ and $d$ is such that $\lim_{n\to\infty}\tau(p,d)$ is a constant in $(0,1)$, $np$ is a polynomially large function of $n$, and $d = \Omega(\log n)$.

\begin{theorem} \torestate{ \label{thm:main-thm-links}
    Let $0 < \tau < 1$ be a constant.
    Let $\bv_1,\dots,\bv_m\sim\scap_{\ge\tau}(w)$ and $\bG\coloneqq\GG_{\tau}(\bv_1,\dots,\bv_{m})$.
    Then for $q\coloneqq \ol{\Phi}_{\Beta{d}}\parens*{\frac{\tau}{1+\tau}}$, suppose
    $qm \gg \log^8 m \cdot \log^{3/2} \frac{1}{q} \cdot \parens*{\frac{1+\tau}{\tau}}^3$
    and $d\ge C\cdot\log m$ for any constant $C > 0$, then for any constant $\gamma > 0$,
    \[
        \displaystyle\Pr\bracks*{|\lambda|_2\parens*{\wh{A}_{\bG}} > \frac{\tau}{1+\tau} + o_{d,m}(1)} \le O\parens*{m^{-\gamma}}.
    \]
    }
\end{theorem}

To prove \pref{thm:main-thm-links}, by \pref{fact:rank-1-sub} it suffices to bound $\norm*{\wh{A}_{\bG} - R}$ for any rank-$1$ PSD matrix $R$. 
For a given $\bG$, the minimizing $R$ for $\norm*{\wh{A}_{\bG}-R}$ is $R = R_{\bG} = \frac{D_{\bG}^{1/2}JD_{\bG}^{1/2}}{\tr\parens*{D_{\bG}}}$ where $J$ is the all-ones matrix.

One challenge in directly performing the trace method on $\wh{A}_{\bG} - R_{\bG}$ is that the degree of any vertex $i$ is a random variable that depends on the locations of all the vectors, and hence introduces extra correlations.
In \pref{sec:trace-method}, this issue was resolved because the degrees concentrated very well, and hence $D_{\bG}^{-1/2}$ and $D_{\bG}^{1/2}$ were close to scalar multiples of identity.
However, in the links the degrees of vertices in $\bG$ no longer concentrate around a single value, and even the behavior of the {\em expected} degree of vertex $i$ depends on which ``shell'' $\bv_i$ is contained in around $w$, $\iprod{\bv_i,w}$.
To better control the degrees, we will study the spectral norm of $\wh{A}_{\bG}-R_{\bG}|\bkappa$ conditioned on the shells $\bkappa\coloneqq\braces*{\bkappa_i\coloneqq \angles*{w, \bv_i}}_{i=1}^{m}$.

Let $\eD\in \R^{m\times m}$ be the conditional expected diagonal degree matrix with $\eD[i,i] = \E[\deg_{\bG}(i)\mid \bkappa]$. 
Then we define the new normalized matrix $\nA=\eD^{-1/2}A_{\bG}\eD^{-1/2}$ and the new conditional rank-$1$ PSD matrix $R_{\bkappa} = \frac{\eD^{1/2}J\eD^{1/2}}{\tr\parens*{\eD}}$. Then by optimality of $R_{\bG}$:
\begin{align*}
    \norm*{\wh{A}_{\bG}-R_{\bG}} &\le \norm*{\wh{A}_{\bG}- \parens*{D_{\bG}^{-1/2}\eD^{1/2}} R_{\bkappa} \parens*{\eD^{1/2}D_{\bG}^{-1/2}}}\\
    &=\norm*{ \parens*{D_{\bG}^{-1/2}\eD^{1/2}} \parens*{\nA - R_{\bkappa}} \parens*{\eD^{1/2}D_{\bG}^{-1/2}} } \\
    &\le \norm*{\nA - R_{\bkappa}} \cdot \norm*{D_{\bG}^{-1/2}\eD^{1/2}}^2.
\end{align*}
Since $\norm*{D_{\bG}^{-1/2}\eD^{1/2}}^2 = \norm*{D_{\bG}^{-1}\eD}$, this is equivalent to bounding
\begin{align*}
    \norm*{\nA - R_{\bkappa}} \cdot \norm*{D_{\bG}^{-1}\eD} &\le \norm*{\nA - R_{\bkappa}} \cdot \max_{i\in[m]} \frac{\eD[i,i]}{D_{\bG}[i,i]} 
\intertext{
Now, in the trace method it is convenient to work with $A_{\bG} - \E[A_{\bG}] \mid \bkappa$, which is not a rank-1 matrix. 
So, applying the triangle inequality,
}
    &\le \parens*{ \norm*{\nA - \E \bracks*{\nA\mid\bkappa} } + \norm*{\E \bracks*{\nA\mid\bkappa} - R_{\bkappa} }} \cdot \max_{i\in[m]} \frac{\eD[i,i]}{D_{\bG}[i,i]}  \numberthis \label{eq:bound-norm-adj},
\end{align*}
It then suffices to bound $\norm*{\E \bracks*{\nA\mid\bkappa} - R_{\bkappa}}$, $\max_{i\in[m]} \frac{\eD[i,i]}{D_{\bG}[i,i]}$, and $\norm*{\nA - \E \bracks*{\nA\mid\bkappa}}$ to complete the proof of \pref{thm:main-thm-links}.

In \pref{sec:orbit-norm} we'll show that $\E[\nA \mid \bkappa]$ is close to $R_{\bG}$ in spectral norm:
\begin{lemma}   \label{lem:orbit-bound}
    If $d\ge C\cdot\log m$ for some constant $C>0$ and the constant
    $\tau \in (0,1)$ satisfies $qm \gg \log^8 m$,
     then
    \[
    \displaystyle{\norm*{\E \bracks*{\nA\mid\bkappa} - R_{\bkappa}}} \le O\parens*{\sqrt{\frac{\log^2d}{d}}}
    \]
    with probability at least $1-o\parens*{m^{-\gamma}}$ for any constant $\gamma > 0$.
\end{lemma}

And the remainder of this section will be devoted to bounding the other two quantities, as follows:
\begin{lemma}   \label{lem:diag-entries}
    For any $0 < \alpha < 1$, 
    \[
        \max_{i\in[m]} \frac{\eDdet[i,i]}{D_{\bG}[i,i]} \le \frac{1}{1-\alpha},
    \]
    with probability at least $1-m\cdot \exp\parens*{-\frac{\alpha^2 q (m-1)}{4}}$.
\end{lemma}

\begin{lemma}   \label{lem:link-trace-method}
    For any $\kappa, q$ and $m$ and $qm\gg\log^8 m \cdot \log^{3/2} \frac{1}{q} \cdot \parens*{\frac{1+\tau}{\tau}}^3$
    \[
        \norm*{\nA - \E \bracks*{\nA\mid\kappa}} \le (1+o_m(1))\cdot\frac{\tau}{1+\tau}.
    \]
\end{lemma}

In service of proving \pref{lem:orbit-bound}, \pref{lem:diag-entries} and \pref{lem:link-trace-method}, we need the following fact that arises in studying random geometric graphs with shifted edge-connectivity thresholds.

\begin{definition}
We define the bivariate function $T(x,y) \coloneqq \frac{\tau - xy}{\sqrt{(1-x^2)(1-y^2)}}$ as the {\em shifted threshold} function, defined so that 
\[
\Pr_{\bx,\by \sim \bbS^{d-2}}\left[\iprod{\bx,\by} \ge T(x,y)\right] = \Pr_{\bu,\bv \sim \bbS^{d-1}}\left[\iprod{\bu,\bv} \ge \tau \mid \iprod{\bu,w} = x, \iprod{\bv,w} = y \right].
\]
\end{definition}
\begin{claim}
\label{clm:tau-tau-worst-case}
    The shifted threshold function $T(x,y) \coloneqq \frac{\tau - xy}{\sqrt{(1-x^2)(1-y^2)}}$ on the domain $x,y\in[\tau,1]$ is maximized when $x = y = \tau$, and achieves value $\frac{\tau}{1+\tau}$.
    Additionally $\partial_x T(x,y)$ and $\partial_y T(x,y)$ are both negative.
\end{claim}

\begin{proof}
The derivatives $\partial_y T(x,y) = \frac{\tau}{\sqrt{1-x^2}} \cdot g(y) - \frac{x}{\sqrt{1-x^2}} \cdot h(y)$ and $\partial_x T(x,y) = \frac{\tau}{\sqrt{1-y^2}} \cdot g(x) - \frac{y}{\sqrt{1-y^2}} \cdot h(x)$, where $g(z) \coloneqq \frac{z}{(1-z^2)^{3/2}} $ and $h(z) \coloneqq \frac{1}{\sqrt{1-z^2}} + \frac{z^2}{(1-z^2)^{3/2}}$.
Since $g(z) < h(z)$ for $z\in(0,1]$, then for $x,y\geq \tau$ we deduce that $\partial_y T, \partial_x T < 0$.
Therefore, $T$ achieves the maximum value $\frac{\tau - \tau^2}{1-\tau^2} = \frac{\tau}{1+\tau}$ when $x = y = \tau$.
\end{proof}

\noindent Now we prove \pref{lem:diag-entries}.
\begin{proof}[Proof of \pref{lem:diag-entries}]
    For any $\alpha\in(0,1)$, consider the event that $\max_{i\in[m]}\frac{\eDdet[i,i]}{D_{\bG}[i,i]} > \frac{1}{1-\alpha}$. We can bound the probability that this event happens by union bound and Bernstein's inequality:
    \begin{align*}
        \Pr\left[\exists i\in[m],~D_{\bG}[i,i] \le (1-\alpha)\eDdet[i,i]\right]
        &\le \sum_{i=1}^m \Pr[D_{\bG}[i,i] \le (1-\alpha)\eDdet[i,i]] \\
        &\le m \cdot \max_i \exp\left( -\frac{1}{2}\cdot\frac{ \alpha^2 \eDdet[i,i]^2 }{(\alpha+1)\eDdet[i,i] }\right) \\
        &\le m\cdot \max_i \exp\parens*{ -\frac{\alpha^2 \eDdet[i,i]}{4} }
    \end{align*}
    Observe that $\eDdet[i,i] = \sum_{j\ne i} \ol{\Phi}_{\Beta{d-1}}(T(\kappa_i,\kappa_j))$.
    By \pref{clm:tau-tau-worst-case}, $T(\kappa_i,\kappa_j)\le \frac{\tau}{1+\tau}$, so $\ol{\Phi}_{\Beta{d - 1}}\parens*{T(\kappa_i,\kappa_j)} \ge \ol{\Phi}_{\Beta{d - 1}}\parens*{\frac{\tau}{1+\tau}} = q$.  Consequently, $\eDdet[i,i]\ge q(m-1)$ from which the desired statement follows.
\end{proof}

\subsection{Spectral norm bound for centered links}

\noindent In the rest of the section, we prove \pref{lem:link-trace-method} by bounding the expected trace $\E\left[\tr\left( \left(\nA - \E[\nA\mid \kappa]\right)^{\ell} \right) \right]$, for $\kappa \in [\tau,1]^m$ a fixed configuration of shells.
The proof will be almost identical to the one in \pref{sec:trace-method}, but here we have to deal with the fact that the graph is not vertex-transitive.
\begin{proof}[Proof of \pref{lem:link-trace-method}]
First observe:
\begin{align*}
    \E\left[\tr\left( \left(\nA - \E[\nA\mid \kappa]\right)^{\ell} \right) \right] &= \E\bracks*{ \tr \parens*{ \parens*{ \eDdet^{-1/2} A_{\bG} \eDdet^{-1/2} - \eDdet^{-1/2} \E\bracks*{A_{\bG}|\kappa} \eDdet^{-1/2} }^{\ell} } } \\
    &= \E \bracks*{ \tr\parens*{ \parens*{ \eDdet^{-1} A_{\bG} - \eDdet^{-1} \E\bracks*{A_{\bG}|\kappa} }^{\ell} } }.
\end{align*}
We rewrite the expression in terms of $\eDdet^{-1} A_{\bG}$ which approximates the transition matrix of the random walk on $\bG$.\footnote{If $\eDdet$ were not the \emph{expected} degree matrix but rather the exact degree matrix of $\bG$, we would have a true transition matrix here.}
Next, we expand the expression in terms of walks in $\calK_m$.

Following the convention of \pref{sec:trace-method}, we use $\calW_{\ell}$ to denote the collection of length-$\ell$ walks in $\calK_m$. 
For every $W\in \calW_{\ell}$, use $\gw = (\vw,\ew)$ to denote the multigraph obtained by the vertices and edges used in $W$.
Use $\mX(e)$ to denote the number of times that an edge $e$ appears in the walk $W$. 
\begin{definition}
We also introduce the following notation. 
Let $\dW(\kappa) := \prod_{(i_t,i_{t+1})\in W} \eDdet^{-1}[{i_t, i_t}]$ denote the normalization constant along the path $W$ conditioned on the shells $\kappa$.
Also define $p_{e} =  \E[\Ind[e\in\bG]\mid \kappa]$ to be the probability that an edge $e$ exists conditioned on $\kappa$.
\end{definition}

Then:
\begin{align}
\E\left[\tr\left( \left(\nA - \E[\nA \mid \kappa]\right)^{\ell} \right)\mid \kappa \right]
&=  \sum_{W \in \calW_\ell} \dW(\kappa) \cdot \E\left[\prod_{e \in \ew} \left(\Ind[e\in\bG] - p_{e}\right)^{\mX(e)}\mid \kappa \right]
\label{eq:pause}
\end{align}

Next we apply the decomposition in \pref{sec:trace-method} to $\gw$ and obtain the 2-core graph $\gtw$ and the forest graph $\gow$. Since conditioned on the vectors $\bv_i\in \vtw$ the events $e\in\bG$ are independent for all $e \in \eow$, the expectation in \pref{eq:pause} can be decomposed into two parts:
\begin{align*}
    &\E\bracks*{\prod_{e\in E(W)} \parens*{\Ind[e\in\bG] - p_e}^{m(e)} \mid \kappa } \\
    =& \prod_{e\in E_1(W)} \E \bracks*{\parens*{\Ind[e\in \bG] - p_e }^{m(e)} \mid \kappa} \E_{\substack{\bv_i \\ i\in V_2(W)}} \bracks*{ \prod_{e\in E_2(W)} \parens*{\Ind[e\in \bG] - p_e}^{m(e)} \mid \kappa }  \numberthis \label{eq:sep-2-core-link}
\end{align*} 

We bound the contribution from the edges in $\etw$ by further spliting $\gtw$ into paths consisting of degree-$2$ vertices and the junction graph $\gjw = (\vjw,\ejw)$ as defined in \pref{def:junc-ver}. 
As in \pref{sec:trace-method} the key observation here is that conditioned on vertices in $\vjw$ the contributions from the paths of degree-$2$ vertices are all independent from each other:
\begin{align*}
    \abs*{\E_{\substack{\bv_i\\i\in V_2(W)}} \left[\prod_{e\in E_2(W)} \parens*{\Ind[e\in\bG] - p_e}^{m(e)}\mid \kappa\right] }
    \le \E_{\substack{\bv_i\\i\in J(W)}} \prod_{f\in E_{J(W)}} \abs*{\E_{\substack{\bv_i\\i\in \gamma(f)\setminus J(W)}} \left[\prod_{e\in\gamma(f)} \parens*{\Ind[e\in\bG] - p_e}^{m(e)} \mid \kappa \right] }
    \numberthis \label{eq:indep-paths-link}
\end{align*}
Now, let $\wo_{\kappa,\kappa'}$ be the transition operator for the random step that walks from vector $v$ in $\Shell_{=\kappa}(w)$ to a uniformly random vector $\bv'$ in $\Shell_{=\kappa'}(w) \cap \scap_{\ge \tau}(v)$ .
Like in \pref{sec:trace-method}, we use $\gamma(f) = (f_0,f_1,\dots,f_{\ell(f)})$ to identify the walk in $\gtw$ corresponding to the edge $f\in E_{J(W)}$.
We denote the edge $(f_i, f_{i+1})$ with $\gamma_i(f)$.
We simplify the contribution from each path $\gamma(f)$ where $f\in\ejw$ as follows:
\begin{align*}
    &\abs*{\E_{\bv_i:i\in \gamma(f)\setminus J(W)} \left[\prod_{e\in\gamma(f)} \parens*{\Ind[e\in\bG] - p_e}^{m(e)} \mid \kappa \right] } \\
    =&\abs*{\E_{\bv_i:i\in \gamma(f)\setminus J(W)} \left[\prod_{e\in\gamma(f)} \parens*{\Ind[e\in \bG] \cdot \parens*{(1-p_e)^{\mX(e)} - (-p_e)^{\mX(e)}} + (-p_e)^{\mX(e)}} \mid \kappa \right] }\\
    =& \abs*{\sum_{T\subseteq \gamma(f)} \E_{\bv_i:i\in\gamma(f)\setminus J(W)} \left[\prod_{e\in T} \Ind[e\in \bG] \cdot \parens*{(1-p_e)^{\mX(e)} - (-p_e)^{\mX(e)}} \prod_{e\in \gamma(f) \setminus T} (-p_e)^{\mX(e)}\mid \kappa \right]} \\
    \le& \abs*{ \prod_{e\in\gamma(f)} \parens*{p_e(1-p_e)^{\mX(e)} + (1-p_e)(-p_e)^{\mX(e)} } } \\
    &+ \abs*{\prod_{e\in\gamma(f)} \parens*{(1-p_e)^{\mX(e)} - (-p_e)^{\mX(e)}} \cdot \prod_{i=0}^{\ell(f)-2}p_{\gamma_i(f)} \cdot \parens*{ \angles*{\prod_{i=0}^{\ell(f)-2}\wo_{\kappa_{f_i},\kappa_{f_{i+1}}} \delta_{\bv_{f_0}}, \scap_{p_{\gamma_{\ell(f)-1}(f)}}(\bv_{f_{\ell(f)}})}  - p_{\gamma_{\ell(f)-1}(f)}}}  \numberthis \label{eq:each-path-link}
\end{align*}

Let $\sw \subseteq \ejw$ be the set of edges $f$ such $\mX(e) = 1$ for some $e\in \gamma(f)$, and $\dw = \ejw \setminus \sw$.
For any $f\in\sw$ the first term in \pref{eq:each-path-link} vanishes, while for any $f \in \dw$
\[\abs*{ \prod_{e\in\gamma(f)} \parens*{p_e(1-p_e)^{\mX(e)} + (1-p_e)(-p_e)^{\mX(e)} } } \le \prod_{e\in\gamma(f)} \parens*{p_e(1-p_e)^{2} + (1-p_e)p_e^{2} } \le \prod_{e\in\gamma(f)} p_e .\]
\noindent Therefore we can derive the following bound on the contribution from the $2$-core graph.
\begin{align*}
    \pref{eq:indep-paths-link} \le
    & \E_{\substack{\bv_i\\i\in J(W)}} \left[ \prod_{f\in \dw} \prod_{i=0}^{\ell(f)-2}p_{\gamma_i(f)} \cdot \parens*{ \abs*{ \angles*{\prod_{i=0}^{\ell(f)-2}\wo_{\kappa_{f_i},\kappa_{f_{i+1}}} \delta_{\bv_{f_0}}, \scap_{p_{\gamma_{\ell(f)-1}(f)}}(\bv_{f_{\ell(f)}})}  - p_{\gamma_{\ell(f)-1}(f)} } + p_{\gamma_{\ell(f)-1}(f)}} \cdot \right. \\\
    & \left.\prod_{f\in \sw} \prod_{i=0}^{\ell(f)-2}p_{\gamma_i(f)} \cdot \abs*{ \angles*{\prod_{i=0}^{\ell(f)-2}\wo_{\kappa_{f_i},\kappa_{f_{i+1}}} \delta_{\bv_{f_0}}, \scap_{p_{\gamma_{\ell(f)-1}(f)}}(\bv_{f_{\ell(f)}})}  - p_{\gamma_{\ell(f)-1}(f) } } \mid \kappa \right]
\end{align*}

To bound the absolute value terms, we take an arbitrary spanning tree $\tjw$ of $\gjw$, and bound the absolute value differently depending on whether $f \in \tjw$ or not.

To bound this expectation, let $\tjw$ be a spanning tree of $\gjw$.
For every edge not in $\tjw$, we apply a worst-case bound.
To state this bound, we define $C \coloneqq \sqrt{\frac{1}{2}\log\frac{1}{q}}\cdot\parens*{\frac{1+\tau}{\tau}}$ and $\lambda\coloneqq\frac{\tau}{1+\tau}$.
\begin{claim}
\label{clm:non-tree-worst-case}
For every shell configuration $\kappa\in [\tau,1]^m$ and non-tree edge $f\in \ejw\setminus\tjw$ , we have that 
\[
    \abs*{ \angles*{\prod_{i=0}^{\ell(f)-2}\wo_{\kappa_{f_i},\kappa_{f_{i+1}}} \delta_{\bv_{f_0}}, \scap_{p_{\gamma_{\ell(f)-1}(f)}}(\bv_{f_{\ell(f)}})}  - p_{\gamma_{\ell(f)-1}(f) }} \le C \cdot \lambda^{\ell(f)-1} \]
\end{claim}

\begin{proof}
To prove the claim, we first need to understand the random variable
\[
    \angles*{\prod_{i=0}^{\ell(f)-2}\wo_{\kappa_{f_i},\kappa_{f_{i+1}}} \delta_{\bv_{f_0}}, \scap_{p_{\gamma_{\ell(f)-1}(f) }}(\bv_{f_{\ell(f)}})}  .
\]
Recall that at time step $i$ the operator $\wo_{\kappa_{f_i},\kappa_{f_{i+1}}}$ denotes the random step that takes a vector $\bv_{f_i} = \kappa_{f_{i}}\cdot w + \sqrt{1-\kappa_{f_{i}}^2} \cdot \bz_{i}$ and outputs $\bv_{f_{i+1}} \coloneqq \kappa_{f_{i+1}}\cdot w + \sqrt{1-\kappa_{f_{i+1}}^2} \cdot \bz_{i+1}$ where $\bz_{i+1}$ is a uniformly random unit vector orthogonal to $w$ such that
\[
    \angles*{ \bv_{f_i}, \bv_{f_{i+1}}} \ge \tau.
\]
This is equivalent to 
\[
    \kappa_{f_i}\kappa_{f_{i+1}} + \sqrt{(1-\kappa_{f_i}^2)(1-\kappa_{f_{i+1}}^2)} \cdot \iprod{\bz_i,\bz_{i+1}} \ge \tau,
\]
which can then be rearranged as
\[
    \iprod{\bz_i,\bz_{i+1}} \ge T\parens*{\kappa_{f_i}\kappa_{f_{i+1}}} \coloneqq \frac{\tau-\kappa_{f_i}\kappa_{f_{i+1}}}{\sqrt{(1-\kappa_{f_i}^2)(1-\kappa_{f_{i+1}}^2)}}.
\]
In particular, we are choosing $\bz_{i+1}$ in the $p_{\gamma_i(f)}$-cap of $\bz_i$ within the $d-2$ dimensional unit sphere orthogonal to $w$.
So the operator $\prod_{i=0}^{\ell(f)-2}\wo_{\kappa_{f_i},\kappa_{f_{i+1}}}$ can be decomposed into its action in the span of $w$ and that in the space orthogonal to $w$. 
The action in the span of $w$ conditioned on $\kappa$ is deterministic.
Orthogonal to $w$, it is the operator $\prod_{i=0}^{\ell(f)-2} P_{p_{\gamma_i(f)}}$ on $\bbS^{d-2}$.
Thus the quantity we are interested in understanding is the same as
\[
    \angles*{ \prod_{i=0}^{\ell(f)-2} P_{p_{\gamma_i(f)}} \delta_{\bz_0}, \scap_{p_{\gamma_{\ell(f)-1}}} (\bz_{\ell(f)}) }.
\]
Now, observe that:
\[
    \abs*{\angles*{ \prod_{i=0}^{\ell(f)-2} P_{p_{\gamma_i(f)}} \delta_{\bz_0}, \scap_{p_{\gamma_{\ell(f)-1}}} (\bz_{\ell(f)}) } - p_{\gamma_{\ell(f)-1}}} \le \dtv{\prod_{i=0}^{\ell(f)-2} P_{p_{\gamma_i(f)}} \delta_{\bz_0} }{\Unif}.
\]
Recall that $p_{\gamma_i(f)} = \Phi_{\Beta{d-1}}\parens*{\tau_{\kappa_i,\kappa_{i+1}}}$, which by \pref{clm:tau-tau-worst-case} is minimized when $\tau_{\kappa_i,\kappa_{i+1}} = \frac{\tau}{1+\tau}$, which means $p_{\gamma_i(f)} \ge q$.
Thus, by \pref{claim:marginal}, \pref{lem:op-act}, and \pref{lem:stoc-dom-tv}, which make concrete the intuition that applying $P_q$ should only mix slower than applying $P_{q'}$ for $q'\ge q$, we have:
\[
    \abs*{\angles*{ \prod_{i=0}^{\ell(f)-2} P_{p_{\gamma_i(f)}} \delta_{\bz_0}, \scap_{p_{\gamma_{\ell(f)-1}}} (\bz_{\ell(f)}) } - p_{\gamma_{\ell(f)-1}}} \le \dtv{P_{q}^{\ell(f)-1} \delta_{\bz_0} }{\Unif} = \dtv{\frac{1}{q}P_q^{\ell(f)-2} \scap_q{\bz_0} }{\Unif}.
\]
\noindent Then by \pref{thm:decay}, the above is
\begin{align*}
\le \sqrt{\frac{1}{2}\log\frac{1}{q}} \cdot \parens*{\frac{\tau}{1+\tau}}^{\ell(f)-2} = \sqrt{\frac{1}{2}\log\frac{1}{q}}\cdot\frac{1+\tau}{\tau}\cdot\parens*{\frac{\tau}{1+\tau}}^{\ell(f)-1} = C \cdot \lambda^{|\gamma(f)|-1},
\end{align*}
which completes the proof.
\end{proof}

Next we bound the contribution of a tree edge $f\in\tjw$ using the following claim whose proof is identical to that of \pref{claim:tree-bound}.

\begin{claim}
\label{clm:tree-bound-link}
For every shell vector $\kappa$ and tree edge $f\in \tjw$ , we have that 
\begin{align*}
    \E_{\substack{\bv_i\\i\in J(W)}} \prod_{f\in T_J(W)} \prod_{i=0}^{\ell(f)-2}p_{\gamma_i(f)}  \cdot \parens*{ \abs*{\angles*{\prod_{i=0}^{\ell(f)|-1}\wo_{\kappa_{f_i},\kappa_{f_{i+1}}} \delta_{\bv_{f_0}}, \scap_{p_{\gamma_{\ell(f)-1}(f)}}(\bv_{f_{\ell(f)}})}  - p_{\gamma_{\ell(f)-1}(f)} } + p_{\gamma_{\ell(f)-1}(f)} \cdot\Ind[f\in \dw] } \\
    \le \prod_{f\in T_J(W)} \prod_{i=0}^{\ell(f)-1}p_{\gamma_i(f)} \cdot \parens*{ 2C\lambda^{\ell(f)} + \Ind[f\in \dw]}.
\end{align*}
\end{claim}
\noindent Combining the two bounds for different edges in $\gjw$ to obtain the simplified bound for \pref{eq:indep-paths-link}: 
\begin{align*}
\pref{eq:indep-paths-link} \le
    & \prod_{f\in \tjw} \prod_{i=0}^{\ell(f)-1}p_{\gamma_i(f)}  \parens*{ 2C\lambda^{\ell(f)} + \Ind[f\in \dw]} \cdot \\
    & \cdot \prod_{f\in \ejw\setminus\tjw} \prod_{i=0}^{\ell(f)-2}p_{\gamma_i(f)} \cdot \parens*{ C\lambda^{\ell(f)-1} + p_{\gamma_{\ell(f)-1}(f)-1}\Ind[f\in \dw]} 
\end{align*}

We now recall some notation from \pref{sec:trace-method}.
We use $e(W)$ to denote $\abs*{\ew}$, $\sing(W)$ for the number of singleton edges in $\gtw$, and $\exc(G)$ for $\abs*{E(G)} - (\abs*{V(G)} - 1)$, the number of edges $G$ has more than a tree.
The relations between these variable are already shown in \pref{obs:minor-exc-ineq} and \pref{claim:walk-count}.
So here we directly apply these results to get that 
\begin{align*}
\pref{eq:indep-paths-link} \le &\prod_{e\in\etw} p_e \cdot \prod_{f\in\ejw\setminus \tjw} p_{ \gamma_{\ell(f)-1}(f) }^{-1} \cdot \lambda^{\sing(W)-\exc(W)}\cdot(3C)^{\abs*{\ejw}} \\
\le &\prod_{e\in\etw} p_e \cdot \prod_{f\in\ejw\setminus \tjw} p_{ \gamma_{\ell(f)-1}(f) }^{-1} \cdot \lambda^{\sing(W)-\exc(W)}\cdot(3C)^{3\exc(W)} \quad\text{by \pref{obs:minor-exc-ineq}} 
\end{align*}
\noindent Therefore
\begin{align*}
\pref{eq:sep-2-core-link} \le & \prod_{e\in E_1(W)} \E \bracks*{\parens*{\Ind[e\in \bG] - p_e }^{m(e)} \mid \kappa} \cdot \prod_{e\in\etw} p_e \cdot \prod_{f\in\ejw\setminus \tjw} p_{ \gamma_{\ell(f)-1}(f) }^{-1} \cdot \lambda^{\sing(W)-\exc(W)}\cdot(3C)^{3\exc(W)} \\
\le& \prod_{e\in E_1(W)}p_e  \cdot \prod_{e\in\etw} p_e \cdot \prod_{f\in\ejw\setminus \tjw} p_{ \gamma_{\ell(f)-1}(f) }^{-1} \cdot \lambda^{\sing(W)-\exc(W)}\cdot(3C)^{3\exc(W)} \\
\le & \prod_{e\in\ew}p_e \cdot \prod_{f\in\ejw\setminus \tjw} p_{ \gamma_{\ell(f)-1}(f) }^{-1}\cdot\lambda^{\sing(W)} \parens*{\frac{27C^3}{\lambda}}^{\exc(W)} \numberthis \label{eq:term-bound-link}
\end{align*} 

Before finally bounding the trace power, we make the following observations.
\begin{observation} \label{obs:exp-deg-bd}
As a consequence of \pref{clm:tau-tau-worst-case} for all $i \in [m]$, the expected degree of vertex $i$ satisfies
\[
    \eDdet[{i,i}] = \E[\deg_{\bG}(i)\mid \kappa] \ge (m-1)\cdot q
\]
\end{observation}

We define $\WS$ to be the set of distinct unlabelled walks of length $\ell$. Then as a corollary of \pref{claim:walk-count}, we have
\begin{corollary}\label{claim:unlabelled-walk-count}
The number of unlabelled walks $\uw \in \WS$ such that $e(\uw) = a, \suw = b,$ and $\exc(\uw) = c$ is at most:
\[\ell^{3(\ell-b)}\cdot\ell^{2c}.\]
\end{corollary}

The result follows by observing that since $\uw$ is unlabelled, we can remove the $m^{a-c+1}$ term that counts the number of distinct labelings in \pref{claim:walk-count}.

For an unlabeled walk $\uw$ and labeled walk $W$, we say $W\sim\uw$ if $W$ is a labeling of $\uw$ in $[m]$.
\begin{claim}   \label{claim:bound-equivalent-walks}
For any unlabelled walk $\uw$ we have that 
\[
    \sum_{W\sim\uw} \dW(\kappa)\prod_{e\in\ew}p_e \cdot \prod_{f\in\ejw\setminus \tjw} p_{\gamma_{\ell(f)-1}(f)}^{-1} \le \parens*{m\cdot q}^{-\exc(U)-\frac{\ell-s(\uw)}{2}}
\]
\end{claim}

\begin{proof}
For each $W\sim \uw$ we use $i_1,\dots,i_{a} \in [m]$ to denote the label of each vertex in $W$ in the order of visit. 
Then we construct the canonical spanning tree $\tw$ by adding each directed edge in the order of $W$ as long as the edge goes to an unvisited vertex. Use $\Par(i_j)$ to denote the parent of vertex $i_j$. Then the $j$-th edge of $\tw$ is $(\Par(i_{j+1}),i_{j+1})$, and use $\tw^{(j)}$ to denote the tree consisting of the first $j$ edges of $\tw$. Then $i_{j+1}$ is always a leaf in $\tw^{(j)}$.

$\tw$ gives rise to a canonical spanning tree $\tjw$ in the contracted graph $\gjw$: an edge $f$ is in $\tjw$ if and only if every edge in the path $\gamma(f)$ is in $\tw$.
From this fact we can deduce that
\[
    \tw = \ew\setminus\left\{ \gamma_{\ell(f)-1}(f) : f\in\ejw\setminus\tjw\right\}.
\]
Therefore, using \pref{obs:exp-deg-bd} we can take a loose upper bound on the contribution of edges outside of $T(W)$ and write
\begin{align*}
&\sum_{W\sim\uw} \dW(\kappa)\prod_{e\in\ew}p_e \cdot \prod_{f\in\ejw\setminus \tjw} p_{\gamma_{\ell(f)-1}(f)}^{-1}\\
    &\le \sum_{W\sim\uw} \parens*{(m-1)\cdot q}^{-\ell+(\abs{\vuw}-1)}\prod_{(i,j)\in\tw}\frac{p_{i,j}}{\eDdet[i,i]}\\
    &\le \parens*{(m-1)\cdot q}^{-\ell+(\abs{\vuw}-1)}\sum_{i_1,\dots,i_{a\in[m]}}\prod_{j=2}^{a}\frac{p_{\Par(i_j),i_j}}{\eDdet[{\Par(i_j),\Par(i_j)}]} \numberthis \label{eq:uw}
\end{align*}
where $a = |V(U)|$.
Next we show by induction on $a$ that 
\[
    \sum_{i_1,\dots,i_{a}}
    \prod_{j=2}^{a}
    \frac{p_{(\Par(i_j),i_j)}}{\eDdet[\Par(i_j),\Par(i_j)]} = 1
\]
The base case $a = 1$ is true by definition.
Suppose this is true for $a-1$. Then:
\begin{align*}
\sum_{i_1,\dots,i_{a}}\prod_{j=2}^{a}\frac{p_{(\Par(i_j),i_j)}}{\eDdet[{\Par(i_j),\Par(i_j)}]}
&= \sum_{i_1,\dots,i_{a-1}} \prod_{j=2}^{a-1}\frac{p_{(\Par(i_j),i_j)}}{\eDdet[{\Par(i_j),\Par(i_j)}]} \cdot \sum_{i_a=1}^{m} \frac{p_{(\Par(i_a),i_a)}}{\eDdet{[\Par(i_a),\Par(i_a)]}} 
&\intertext{By definition $\sum_{i_a} \frac{p_{(\Par(i_a),i_a)}}{\eDdet{[\Par(i_a),\Par(i_a)]}} = 1$, so we have: } 
&= \sum_{i_1,\dots,i_{a-1}} \prod_{j=2}^{a-1}\frac{p_{(\Par(i_j),i_j)}}{\eDdet{[\Par(i_a),\Par(i_a)]}} \cdot 1 = 1
\end{align*}
\noindent
Finally, observe that $\ell - (|V(U)| - 1) \ge \exc(U) + \frac{\ell-\suw}{2}$ is at least the number of steps that use a previously walked-on edge.
The way to see this is to observe that the quantity $\ell - (|V(U)|-1)$ counts the number of steps to a previously visited vertex.
Such a step can either (1) use an excess edge for the first time, of which there are $\exc(U)$ steps, or (2) use a previously walked-on edge, which must be at least half the steps that do not use a singleton edge, i.e. at least $\frac{\ell-\suw}{2}$ steps.
Thus we conclude that $\pref{eq:uw} \le \parens*{(m-1)\cdot q}^{-\ell+(\abs{\vuw}-1)} \le \parens*{(m-1)\cdot q}^{-\exc(U)-\frac{\ell-\suw}{2}}$.
\end{proof}
\noindent Now we are finally already to bound the expected trace power.
Plugging \pref{eq:term-bound-link} into \pref{eq:pause} gives:
\begin{align*}
\pref{eq:pause} &= \sum_{\uw\in\WS}\sum_{W\sim\uw} \dW(\kappa)\prod_{e\in\ew}p_e \cdot \prod_{f\in\ejw\setminus \tjw} p_{\gamma_{\ell(f)-1}(f)}^{-1}\cdot\lambda^{\sing(W)} \parens*{\frac{27C^3}{\lambda}}^{\exc(W)} \\
& \le \sum_{\uw\in\WS} \parens*{(m-1)\cdot q}^{-\exc(U)-\frac{\ell-s(\uw)}{2}} \cdot \lambda^{s(\uw)} \parens*{\frac{27C^3}{\lambda}}^{\exc(\uw)} &\text{by \pref{claim:bound-equivalent-walks}} \\
& = \sum_{a=1}^{\ell} \sum_{b=1}^{\ell} \sum_{c=1}^{\ell} \sum_{\substack{\uw\in\WS \\ e(\uw)=a,\, \suw=b,\, \exc(\uw)=c}} \parens*{(m-1)\cdot q}^{-\frac{\ell-b}{2}} \cdot \lambda^{b} \parens*{\frac{27C^3}{\lambda q (m-1)}}^{c} \\
& = \sum_{a=1}^{\ell} \sum_{b=1}^{\ell} \sum_{c=1}^{\ell} \parens*{(m-1)\cdot q}^{-\frac{\ell-b}{2}} \cdot \lambda^{b} \parens*{\frac{27C^3}{\lambda q (m-1)}}^{c} \cdot \ell^{2(\ell-b)}  \cdot \ell^{2c} &\text{by \pref{claim:unlabelled-walk-count}} \\
&  = \ell \sum_{b=1}^{\ell} \sum_{c=1}^{\ell} \parens*{\frac{\ell^2}{\sqrt{(m-1)\cdot q}}}^{\ell-b} \cdot \lambda^b \parens*{\frac{27C^3\ell^2}{\lambda q (m-1) }}^{c} \\
& = \ell^3 \max\parens*{1,\parens*{\frac{27C^3\ell^2}{\lambda q (m-1) }}^{\ell}}\cdot \max\parens*{\lambda^{\ell},\parens*{\frac{\ell^2}{\sqrt{(m-1)\cdot q}}}^{\ell} }  
\end{align*}
By choosing $\ell = \log^2 m$, we can conclude that with probability at least $1-m^{-\gamma}$,
\[
    \norm*{\nA - \E [\nA\mid\kappa]} \le (1+o(1)) \cdot \parens*{1+\frac{27C^3\log^4 m}{\lambda q m}} \cdot \max\braces*{ \lambda, \frac{\log^4 m}{\sqrt{qm}}}
\]
Since $qm \gg \log^8 m \cdot \log^{3/2}\frac{1}{q} \cdot \parens*{\frac{1+\tau}{\tau}}^3$, we have:
    \[
        \norm*{\nA - \E [\nA\mid\kappa]} \le (1+o(1)) \cdot \lambda.    \qedhere
    \]
\end{proof}

\section{The second eigenvalue of the shell walk}\label{sec:shells}

The goal of this section is to prove \pref{lem:orbit-bound}, and in particular bound $\norm*{\E[\nA \mid \bkappa] - R_{\bkappa} }$ where $\bkappa \sim (\Beta{d}_{\geq \tau})^{\otimes m}$ is a configuration of $m$ shells, and we have conditioned on $\kappa_i = \iprod{w,\bv_i}$ for all $i \in [m]$.

To make the matrix more amenable to analysis via the coupling-based techniques we use here, we first observe that the spectral norm we are interested in bounding is equal to the largest eigenvalue of $\E[\eD^{-1}A_{\bG}\mid\bkappa] - \vec{1}\bpi^{\top}$ where $\bpi \coloneqq \frac{\eD}{\tr(\eD)}\vec{1}$ is the stationary distribution of the Markov chain described by the transition matrix $\E[\eD^{-1}A_{\bG}\mid\bkappa]$.
Indeed:
\[
    \norm*{\E[\nA \mid \bkappa] - R_{\bkappa} } = |\lambda|_{\max}\parens*{\E[\nA \mid \bkappa] - R_{\bkappa}} = |\lambda|_{\max}\parens*{\E[\eD^{-1}A_{\bG}\mid\bkappa] - 1\bpi^{\top}}
\]
where the first equality uses symmetry of the matrix and the second equality uses the fact that the spectra of $M$ and $\eD^{-1/2}M\eD^{1/2}$ are identical.
For convenience, let $Q = \E[A_{\bG} \mid \bkappa]$ and let $\nQ = \eD^{-1}Q$.
The following main result of this section implies \pref{lem:orbit-bound}.
\begin{lemma}    \label{lem:main-orbit-bound}
    There exists a constant $C > 0$ such that for any $d \ge C\log m$, any threshold $\tau \in (0,1)$ such that $q\coloneqq \ol{\Phi}_{\Beta{d}}\parens*{\frac{\tau}{1+\tau}} \gg \log^8 m/m$, and any constant $\gamma > 0$, with probability at least $1-o(m^{-\gamma})$ over the shells $\bkappa\sim (\Beta{d}_{\geq \tau})^{\otimes m}$,
    \[
        \abs*{\lambda}_{\max}\parens*{\nQ-\vec{1}\bpi^{\top}} \le O\parens*{ \sqrt{ \frac{\log^2 d}{d} } }.
    \]
\end{lemma}

In service of proving \pref{lem:main-orbit-bound}, we show:
\begin{lemma}    \label{lem:row-sum}
There exists a constant $C > 0$ such that for any $d \ge C\log m$, any threshold
    $\tau \in (0,1)$ such that $qm \gg \log^8 m$, and 
    any constant $\gamma > 0$, with probability at least $1-o(m^{-\gamma})$ over the shells $\bkappa\sim (\Beta{d}_{\geq \tau})^{\otimes m}$,
    \[
        \max_{i,j\in[n]} \norm*{ \parens*{\nQ^2}_{i,*} - \parens*{\nQ^2}_{j,*} }_1 \le O\parens*{\frac{\log^2 d}{d}},
    \]
    where $\parens*{\nQ^2}_{i,*}$ denotes the $i$-th row of the matrix $\nQ^2$.
\end{lemma}
We show how to prove \pref{lem:main-orbit-bound} using \pref{lem:row-sum} and then dedicate the rest of the section to proving \pref{lem:row-sum}.
\begin{proof}[Proof of \pref{lem:main-orbit-bound}]
    First, observe that $\abs*{\lambda}_{\max}\parens*{\nQ-\vec{1}\bpi^{\top}} = \ssqrt{\abs*{\lambda}_{\max}\parens*{\nQ^2 - \vec{1}{\bpi^{\top}}} }$.
    Via the row sum bound for the largest magnitude eigenvalue of a matrix (\pref{claim:row-sum-bound}), \pref{lem:row-sum} and the fact that $\bpi$ is the stationary distribution of $\nQ$:
    \begin{align*}
        \abs*{\lambda}_{\max}\parens*{\nQ^2 - \vec{1}\bpi^{\top}}
        &\le \max_{i\in[n]} \norm*{\parens*{\nQ^2}_{i,*}-\bpi^{\top}}_1 \\
        &= \max_{i\in[n]} \norm*{\parens*{\nQ^2}_{i,*} - \sum_{j\in[n]} \bpi_j \parens*{\nQ^2}_{j,*}}_1 \\
        &= \max_{i\in[n]} \norm*{\sum_{j\in[n]} \bpi_j \parens*{\parens*{\nQ^2}_{i,*} - \parens*{\nQ^2}_{j,*}} }_1 \\
        &\le \max_{i,j\in[n]} \norm*{ \parens*{\nQ^2}_{i,*} - \parens*{\nQ^2}_{j,*} }_1.
        \qedhere
    \end{align*}
\end{proof}

\label{sec:orbit-norm}

\subsection{Coupling for the shell walk}

In this section we give the proof of \pref{lem:row-sum} assuming a few key lemmas. The proofs for the key lemmas are deferred to the next section. 

\subsubsection{A high-probability condition for $\bkappa$}

To simplify the upcoming computations for \pref{lem:row-sum}, we will condition on the following high-probability event $\mathcal{E}_\gamma$ over the sample space of the shells $\bkappa$: 
\begin{definition} \label{def:orbit-event}
    Let $\mathcal{E}_\gamma$ be the event that for all $m$ shells $\bkappa_i \in \bkappa$ in the link, $\bkappa_i \leq \xbound$, where
    \[
    \xbound = \tau\parens*{m^{-2\gamma-1}\cdot \ol{\Phi}_{\Beta{d}}(\tau) ,d}.
    \]
    Note that the outermost $\tau(\cdot)$ refers to the threshold function, rather than the value of the threshold such that $\ol{\Phi}_{\Beta{d}}(\tau) = p$.
\end{definition}

\begin{claim}
    The event $\calE_\gamma$ occurs with probability at least $1-m^{-2\gamma}$.
\end{claim}
\begin{proof}
    By definition, for any shell $\bkappa_i$:
    $
    \Pr_{\bkappa_i \sim \Beta{d}\vert_{\ge\tau}}[\bkappa_i \ge \xbound] \le m^{-2\gamma-1}.
    $
    Our conclusion follows from taking a union bound over all $m$ shells.
\end{proof}
The conditioning on $\mathcal{E}_\gamma$ can be folded into high-probability guarantee over $\bkappa$ in \pref{lem:row-sum}.
Thus, for the remainder of the section, we can assume that $\bkappa$ obeys event $\mathcal{E}_\gamma$. 
This will be especially relevant in the analysis of the outlier shells (\pref{sec:outlier-cols}).

\begin{claim}
	If $d \geq C \log m$ for some constant $C > 0$, then $\xbound \le 1 - \eps_\gamma$, where $\eps_\gamma > 0$ is a constant depending only on $\gamma$.
\end{claim}
\begin{proof}
    Since $\tau$ is a constant bounded away from $1$, and $d = \Omega(\log m)$, by the lower bound in \pref{lem:approx-tails-dpd}, the quantity $m^{-2\gamma-1}\cdot\ol{\Phi}_{\Beta{d}}(\tau)$ is at least $\exp(-C_{\gamma}d)$ for some constant $C_{\gamma}$ depending on $\gamma$.
    By the upper bound in \pref{lem:approx-tails-dpd}, there is a constant $\eps_{\gamma} > 0$ such that $\ol{\Phi}_{\Beta{d}}(1-\eps_{\gamma}) \le \exp(-C_{\gamma}d)$.
    Since $\ol{\Phi}_{\Beta{d}}$ is a decreasing function, $\eta \le 1 - \eps_{\gamma}$.
\end{proof}

\subsubsection{``Typical'' and ``outlier'' shells}
In the proof of \pref{lem:row-sum}, we analyze the contributions of ``typical'' and ``outlier'' shells separately.

\begin{definition}
We say that a shell $\kappa_i$ is ``typical'' if $\kappa_i \in [\tau, \tau(1 + \alpha)]$, for $\alpha = \frac{36 \log d}{\tau^2(d - 3) (1 - \xbound)}$.
\end{definition}

\begin{remark}\label{fact:typical-alpha}
	$\alpha$ is chosen so that $\nQ$, when restricted to typical rows and columns, will resemble a rank-1 matrix. 
	For our eventual choices of $d$ and $m$, the event that every shell is typical \emph{does not} occur with high probability; we will inevitably need to deal with outlier shells.
\end{remark}

\subsubsection{Total variation bound from similarity of typical rows and scarcity of outlier columns}

To obtain the desired row-sum bound in  \pref{lem:row-sum}, we will prove the following two lemmas about the matrix $\nQ$. 
The first shows that outlier columns do not contribute much to the total row sum of any row:
\begin{lemma} \torestate{\label{lem:outlier-trans-prob} 
    For any $d \ge C\log m$ for some constant $C > 0$ and any 
    $ \tau \in [0,1]$ such that $qm \gg \log^8 m$, 
    if $\kappa_i \leq \xbound$,
    \[
    \sum_{k = 1}^m \nQ_{i, k} \cdot \Ind{[k \text{ outlier}]} \leq O\left(\frac{1}{d}\right)
    \]
    with probability $1 - o(m^{- \log m})$.}
\end{lemma}
\noindent The second shows that typical rows are similar at indices corresponding to typical columns:

    \begin{lemma} \torestate{\label{lem:typical-trans-prob}
        For any dimension $d$ and any threshold $\tau \in (0,1)$,
        if $\kappa_i$, $\kappa_j$ correspond to typical shells, then for all $\ell$ such that $\kappa_\ell$ is typical,
        \[
        \nQ_{i, \ell} \in \left( 1 \pm O\left( \frac{\log^2 d}{d} \right) \right) \nQ_{j, \ell}
        \]
        }
    \end{lemma}

These lemmas are both proven by direct calculation, and we leave their proofs to \pref{sec:outlier-cols} and \pref{sec:typical-entries-typical-rows} respectively.

To illustrate these statements, we provide a schematic of the matrix $Q$ below, organized into its typical and outlier rows and columns.
\pref{lem:typical-trans-prob} states that the sub-rows in area (I) of the matrix are all nearly equal to each other. 
\pref{lem:outlier-trans-prob} says that the sum of its entries in area (II) or area (IV) is a $O\left(\frac{1}{d}\right)$ fraction of the total row sum.

\begin{center}
    \begin{tikzpicture}
    \matrix [matrix of math nodes,left delimiter=(,right delimiter=),row sep=0.15cm,column sep=0.15cm] (m) {
        \, \, &   \, \,   &   \, \,  &  \, \, \, \, \,  \\
        \, \, &  \text{(I)} &   \, \,  & \text{(II)} \\
        \, \,  &  \, \,  & \, \,  &  \, \, \, \, \, \\
        \, \, & \text{(III)}   & \,  \,  & \text{(IV)}  \\ };
    
    \draw[dashed] ($0.5*(m-1-3.north east)+0.4*(m-1-4.north west)$) -- ($0.5*(m-4-3.south east)+0.6*(m-4-4.south west)$);
    
    \draw[dashed] ($0.6*(m-3-1.south west)+0.4*(m-4-1.north west)$) -- ($0.5*(m-3-4.south east)+0.7*(m-4-4.north east)$);
    
    \node[above=1pt of m-1-1] (top-1) {};
    \node[above=1pt of m-1-3] (top-3) {};
    \node[above=1pt of m-1-4] (top-4) {};
    
    \node[right=12pt of m-1-4] (right-1) {};
    \node[right=12pt of m-3-4] (right-3) {};
    \node[right=6pt of m-4-4] (right-4) {};
    
    \node[rectangle,above delimiter=\{] (del-top-1) at ($0.5*(top-1.south) +0.5*(top-3.south)$) {\tikz{\path (top-1.south west) rectangle (top-3.north east);}};
    \node[above=10pt] at (del-top-1.north) {typical};
    \node[rectangle,above delimiter=\{] (del-top-2) at ($0.5*(top-4.south) +0.5*(top-4.south)$) {\tikz{\path (top-4.south west) rectangle (top-4.north east);}};
    \node[above=10pt] at (del-top-2.north) {outlier};
    
    \node[rectangle,right delimiter=\}] (del-right-1) at ($0.5*(right-1.west) +0.5*(right-3.west)$) {\tikz{\path (right-1.north east) rectangle (right-3.south west);}};
    \node[right=22pt] at (del-right-1.west) {typical};
    \node[rectangle,right delimiter=\}] (del-right-2) at ($0.5*(right-4.west) +0.5*(right-4.west)$) {\tikz{\path (right-4.south east) rectangle (right-4.south west);}};
    \node[right=22pt] at (del-right-2.west) {outlier};
    \end{tikzpicture}
\end{center}

One straightforward corollary of \pref{lem:typical-trans-prob} and \pref{lem:outlier-trans-prob} is that the $\ell_1$ norms of the differences between any two \emph{typical} rows of $\nQ$ is at most $O \left(\frac{\log^2 d}{d}\right)$.
More formally: 
\begin{corollary} \label{cor:typical-1-norm}
    For any $d \ge C\log m$ for some constant $C > 0$ and any threshold
    $0< \tau \le 1$ such that $qm \gg \log^8 m$, let $i, j$ be rows of $\nQ$ corresponding to typical shells $\kappa_i$, $\kappa_j$. Then:
    \[
    \left \|\left(\nQ\right)_{i, *} - \left(\nQ\right)_{j, *} \right\|_1 \leq O \left(\frac{\log^2 d}{d}\right)
    \] with probability $1 - o(m^{- \log m})$.
\end{corollary}
\begin{proof}
    We split $\left \|\left(\nQ\right)_{i, *} - \left(\nQ\right)_{j, *} \right\|_1$ based on its contributions from typical columns and outlier columns. 
    \[
        \left \|\left(\nQ\right)_{i, *} - \left(\nQ\right)_{j, *} \right\|_1 = \sum_{\ell \text{ typical}} \left|\nQ_{i, \ell} - \nQ_{j, \ell} \right| + \sum_{\ell \text{ outlier}} \left|\nQ_{i, \ell} - \nQ_{j, \ell}\right| 
    \]
    \pref{lem:outlier-trans-prob} and the triangle inequality tell us that with probability $1 - o(m^{-\log m})$:
    \[
         \sum_{\ell \text{ outlier}} \left|\nQ_{i, \ell} - \nQ_{j, \ell}\right| \leq \sum_{\ell \text{ outlier}} \left|\nQ_{i, \ell} \right| + \sum_{\ell \text{ outlier}} \left| \nQ_{j, \ell}\right| \leq O\left(\frac{1}{d}\right)
    \]
    \pref{lem:typical-trans-prob} tells us that for some constant $C > 0$: 
    \begin{align*}
        \sum_{\ell \text{ typical}} \left|\nQ_{i, \ell} - \nQ_{j, \ell} \right| &\leq \sum_{\ell \text{ typical}} \left[1 + \left(\frac{C\log^2 d}{d} - 1\right) \right] \nQ_{j, \ell} \\
        &= \frac{C\log^2 d}{d} \sum_{\ell \text{ typical}} \nQ_{j, \ell} \leq \frac{C\log^2 d}{d} 
    \end{align*}
    Combining the bounds on $\sum_{\ell \text{ typical}} \left|\nQ_{i, \ell} - \nQ_{j, \ell} \right|$ and $\sum_{\ell \text{ outlier}} \left|\nQ_{i, \ell} - \nQ_{j, \ell} \right|$ gives the desired result. 
\end{proof}

\noindent We can furthermore use \pref{lem:outlier-trans-prob}, \pref{cor:typical-1-norm}, and a coupling argument, to prove \pref{lem:row-sum}.
\begin{proof}[Proof of \pref{lem:row-sum}]
    We may assume event $\mathcal{E}_\gamma$ (that all shells $\kappa_i \leq \xbound$), and the outcomes of \pref{lem:outlier-trans-prob} and \pref{cor:typical-1-norm}. 
    The union of these three events occur with probability $1 - o(m^{-\gamma} + m^{-\log m})$.
    
    Let $(\bX_a^{(t)})_{t\ge 0}$ be the trajectory of Markov chain $\nQ$ starting at vertex $a$.
    For any pair of vertices $i,j\in[n]$, we couple $\bX_i^{(2)}$ and $\bX_j^{(2)}$ such that they are equal with probability $1 - O\left(\frac{\log^2 d}{d} \right)$, and so by \pref{fact:coupling-tv}:
    \[
        \left\|\text{pmf}\left(\bX_i^{(2)} \right) - \text{pmf} \left(\bX_j^{(2)} \right) \right\|_{\text{TV}} = \frac{1}{2}\left\| \left(\nQ^2 \right)_{i, *} - \left(\nQ^2 \right)_{j, *} \right\|_1 \leq O\left(\frac{\log^2 d}{d} \right)
    \]
    from which the desired result follows.
    
    We now exhibit such a coupling between $\bX_i^{(2)}$ and $\bX_j^{(2)}$.
    Observe that $\bX_i^{(1)}$ and $\bX_j^{(1)}$ are distributed according to $\parens*{\nQ}_{i,*}$ and $\parens*{\nQ}_{j,*}$.
    When $\kappa_i$ and $\kappa_j$ are both typical shells, we can couple $\bX_i^{(1)}$ and $\bX_j^{(1)}$ such that they are equal with probability $1 - O\left(\frac{\log^2 d}{d} \right)$ using \pref{cor:typical-1-norm} and \pref{fact:coupling-tv}. As a result $\bX_i^{(2)}$ and $\bX_j^{(2)}$ can be coupled so that they agree with probability $1 - O\left(\frac{\log^2 d}{d} \right)$.
    When either $\kappa_i$ or $\kappa_j$ is an outlier shell, though $\bX_i^{(1)}$ and $\bX_j^{(1)}$ may have TV distance greater than $O\left(\frac{\log^2 d}{d} \right)$, by \pref{lem:outlier-trans-prob} for both random variables $1 - O\parens*{\frac{1}{d}}$-fraction of the probability mass is over the typical shells. Due to that, we can couple $\bX_i^{(2)}$ and $\bX_j^{(2)}$ with probability $\parens*{1 - O\parens*{\frac{1}{d}}} \cdot \parens*{1 - O\left(\frac{\log^2 d}{d}\right)} $ by \pref{lem:outlier-trans-prob}. 
    Thereby we complete the proof.
\end{proof}

\subsection{Relating typical rows and bounding outlier columns}

Our next step is to prove \pref{lem:outlier-trans-prob} and \pref{lem:typical-trans-prob}. Throughout this section, instead of working with $\nQ$, we will work with  $Q = \E[\bA_{\bG} \mid \kappa]$; it will be simpler to operate on the entries of $Q$ and later relate them to $\nQ$.
We first characterize the entries of $Q$ using the $\Beta{d - 1}$ distribution. 

\subsubsection{The conditional expected adjacency matrix}

For each pair of vertices $i,j \in [m]$, we have 
\[
Q_{i,j} 
= \porb{\kappa_i}{\kappa_j} 
\defeq \Pr_{\bv_i ,\bv_j \sim \rho_w}\left[\iprod{\bv_i, \bv_j} \geq \tau \mid \iprod{\bv_i, w} = \kappa_i, \Iprod{\bv_j, w} = \kappa_j\right].
\]
Though $\porb{x}{y}$ is symmetric in its inputs $x, y$, we choose this notation because we will often work with the function $\porb{x}{\cdot}$, where the input is any value in $[\tau, 1]$.

\begin{claim}
    The quantity $\porb{x}{y}$ is exactly a tail probability of the $\Beta{d - 1}$ distribution:
    \[
        \porb{x}{y} = \ol{\Phi}_{\Beta{d-1}} \parens*{T(x,y)}
    \]
\end{claim}
\begin{proof}
    Conditional on $\iprod{\bv_i,w} = x$ and $\iprod{\bv_j,w} = y$, $\bv_i$ and $\bv_j$ are distributed as $\bv_i = x \cdot w + \sqrt{1-x^2}\cdot \bu_i$ and $\bv_j = y \cdot w + \sqrt{1-y^2}\cdot \bu_j$, for $\bu_i,\bu_j$ uniformly random unit vectors orthogonal to $w$.
    Now, observe that the condition $\angles*{\bv_i,\bv_j} \ge \tau$ is equivalent to $\angles*{\bu_i,\bu_j} \ge \frac{\tau-xy}{\ssqrt{(1-x^2)(1-y^2)}}$, and thus the desired statement follows since $\bu_i$ and $\bu_j$ are sampled from a space isometric to $\bbS^{d-2}$.
\end{proof}

\subsubsection{The contribution of outlier columns} \label{sec:outlier-cols}

The goal of this section is to prove \pref{lem:outlier-trans-prob}.
\restatelemma{lem:outlier-trans-prob}

\noindent The lemma statement is equivalent to the following about $Q$: for all $i \in [m]$, 
\[
\frac{\sum_{k = 1}^m Q_{i, k} \cdot \Ind{[k \text{ outlier}]}}{\sum_{k = 1}^m Q_{i, k}} \leq O\left(\frac{1}{d}\right)
\]
with probability $1 - o(m^{- \log m})$.
Recalling that we use $Z$ to denote the normalizing constant from \pref{sec:prelim-sphere}, we can define:
\begin{align*}
    N(x) &\coloneqq \int_{\tau(1 + \alpha)}^1 \porb{x}{y} \cdot Z^{-1} (1 - y^2)^{(d - 3)/2} dy = Z^{-1} \int_{\tau(1 + \alpha)}^1 (1 - y^2)^{(d - 3)/2} \cdot \ol{\Phi}_{\Beta{d - 1}}\left(T(x, y) \right) dy \\
    D(x) &\coloneqq Z^{-1} \int_{\tau}^1 \porb{x}{y} \cdot (1 - y^2)^{(d - 3)/2} dy = Z^{-1} \int_{\tau}^{1} (1 - y^2)^{(d - 3)/2} \cdot \ol{\Phi}_{\Beta{d - 1}}\left(T(x, y) \right) dy 
\end{align*}
By our definitions of $N(x)$ and $D(x)$, and recalling that we condition on $\mathcal{E}_\gamma$ (\pref{def:orbit-event}) throughout this section,
\[
N(\kappa_i) = \E_{\bkappa_\ell}[Q_{i, \ell} \cdot \Ind{[\ell \text{ outlier}]}], \; \; D(\kappa_i) = \E_{\bkappa_\ell}[Q_{i, \ell}] 
\]
The $Z^{-1} (1 - y^2)^{(d - 3)/2}$ expression in each integrand comes from the probability density over shells. 

First, when $\kappa_i \leq \eta$, we establish that the ratio of the \emph{expected} sum of outlier $Q_{i, k}$ and typical $Q_{i, k}$ is of the desired magnitude of $O\left(\frac{1}{d}\right)$. 
\begin{lemma} \torestate{\label{lem:outlier-mean} For any $d \ge C\log m$ for some constant $C > 0$, any constant $ \tau \in [0,1]$, and any $x \leq \eta$, 
	\[
    \frac{N(x)}{D(x)} \leq O \left(\frac{1}{d}\right).
    \]}
\end{lemma}

\noindent The proof is by computation and is provided in \pref{app:n-d-ratio}. 
We are now ready to prove \pref{lem:outlier-trans-prob}.

\begin{proof}[Proof of \pref{lem:outlier-trans-prob}]
    For convenience, we use $N$ and $D$ as shorthand for $N(\kappa_i)$ and $D(\kappa_i)$.
    We compute a high probability lower bound for the numerator $\sum_{k = 1}^m Q_{i, k} \cdot \Ind{[k \text{ outlier}]}$ and a high probability upper bound for the denominator $\sum_{k = 1}^m Q_{i, k}$. 
    
    \medskip
    \noindent \textbf{Concentration of the numerator:} We first show that $\sum_{k = 1}^m Q_{i, k} \cdot \Ind{[k \text{ outlier}]}$ concentrates around $Nm$.
    First, each $Q_{i, k} \leq 1$. 
    Then, $\Var(Q_{i, k}) \leq \E[Q_{i, k}^2 \cdot \Ind{[k \text{ outlier}]}] \leq \E[Q_{i, k} \cdot \Ind{[k \text{ outlier}]}] = N$. 
    Applying Bernstein's Inequality, we obtain:
        \[
        \Pr\left[\sum_{k = 1}^m Q_{i, k} \cdot \Ind{[k \text{ outlier}]} \geq Nm + \left(\sqrt{Nm} + 1\right) \log^2 m\right] \leq \exp \left( -\frac{\left(\sqrt{Nm} + 1\right)^2 \log^4 m}{\frac{1}{2} \cdot Nm + \frac{1}{3} \cdot \left(\sqrt{Nm} + 1\right) \log^2 m} \right) \leq m^{-\log m}
        \] 
     \noindent \textbf{Concentration of the denominator:} We next show that $\sum_{k = 1}^m Q_{i, k}$ concentrates around $Dm$.
     Using a similar bound on variance as above, and applying Bernstein's inequality again:
        \[
        \Pr\left[\sum_{k = 1}^m Q_{i, k} \leq Dm - \left(\sqrt{Dm} + 1 \right) \log m \right] \leq \exp \left( -\frac{\left( \sqrt{Dm} + 1 \right)^2 \log^2 m}{\frac{1}{2} \cdot Dm + \frac{1}{3} \cdot \left(\sqrt{Dm} + 1\right) \log m} \right) \leq m^{-\log m}
        \] 
    Thus, with probability greater that $1 - 2 m^{-\log m}$, the ratio $\frac{\sum_{k = 1}^n Q_{i, k} \cdot \Ind{[k \text{ outlier}]}}{\sum_{k = 1}^m Q_{i, k}}$ is at most
    \[
    \frac{Nm + \left(\sqrt{Nm} + 1\right) \log^2 m}{Dm - \left(\sqrt{Dm} + 1 \right) \log m}
    \]
We can upper bound this by $\leq O \left(\frac{1}{d} \right)$, as \pref{lem:outlier-mean} tells us $\frac{N}{D} \leq O\left(\frac{1}{d}\right)$, and since $Dm \geq qm \gg  \log^8 m$ the first terms in the ratio dominate.
\end{proof}

\subsubsection{Relating typical rows} \label{sec:typical-entries-typical-rows}

Our goal for this section is to prove:
\restatelemma{lem:typical-trans-prob}

We will translate \pref{lem:typical-trans-prob} into a statement about $Q$ first. 
Let $Q_{i,*}$ and $Q_{j, *}$ be rows of $Q$ corresponding to            typical shells $\kappa_i,\kappa_j$. 
We will prove that $Q_{i, *}$ and $Q_{j, *}$, when restricted to typical columns, are nearly constant scalings of each other. 
Formally, we will prove:
\begin{lemma} \label{lem:typical-ratio} For any dimension $d$ and any threshold $\tau \in (0,1)$, let $\kappa_i, \kappa_j, \kappa_\ell$ be typical shells. Then,
    \[
    \frac{\porb{\kappa_i}{\kappa_\ell}}{\porb{\kappa_j}{\kappa_\ell}} \cdot \left(\frac{\porb{\kappa_i}{\tau}}{\porb{\kappa_j}{\tau}} \right)^{-1} \in 1 \pm O \left(\frac{\log^2 d}{d}\right)
    \]
\end{lemma}
\noindent In other words, this establishes that for any typical shells $\kappa_i, \kappa_j, \kappa_\ell$, 
\[
\frac{\porb{\kappa_i}{\kappa_\ell}}{\porb{\kappa_j}{\kappa_\ell}} \approx \frac{\porb{\kappa_i}{\tau}}{\porb{\kappa_j}{\tau}},
\] 
where the quantity on the right is a constant $c_{ij}$ depending only on $\kappa_i$ and $\kappa_j$ (not $\kappa_\ell$).

\begin{proof}[Proof of \pref{lem:typical-trans-prob} using \pref{lem:typical-ratio}.]
By the definition of $\nQ$, 
\[
    \nQ_{i, \ell} = \frac{Q_{i, \ell}}{\sum_{k = 1}^m Q_{i, k}}, \; \nQ_{j, \ell} = \frac{Q_{j, \ell}}{\sum_{k = 1}^m Q_{j, k}} 
\]
It suffices to prove that $\frac{\nQ_{i \ell}}{\nQ_{j, \ell}}$ is close to 1. 
Expanding $\frac{\nQ_{i \ell}}{\nQ_{j, \ell}}$, we can upper bound:
\begin{align*}
    \frac{\nQ_{i, \ell}}{\nQ_{j, \ell}} &= \frac{Q_{i, \ell}}{\sum_{k = 1}^m Q_{i, k}} \cdot \frac{\sum_{k = 1}^m Q_{j, k}}{Q_{j, \ell}} \\
    &= \frac{Q_{i, \ell}}{Q_{j, \ell}} \cdot \frac{\sum_{k = 1}^m Q_{j, k}}{\sum_{k = 1}^m Q_{i, k}} \\
    &\leq \left(1 + \frac{C \log^2 d}{d} \right) \left(\frac{\porb{\kappa_i}{\tau}}{\porb{\kappa_j}{\tau}} \right) \cdot \frac{\sum_{k \text{ typical}} Q_{j, k} + \sum_{k \text{ outlier}} Q_{j, k}}{\sum_{k \text{ typical}} Q_{i, k} + \sum_{k \text{ outlier}} Q_{i, k}} \\
    &\leq \left(1 + \frac{C \log^2 d}{d} \right) \left(\frac{\porb{\kappa_i}{\tau}}{\porb{\kappa_j}{\tau}} \right) \cdot \frac{\left(1 + \frac{C'}{d}\right)\sum_{k \text{ typical}} Q_{j, k}}{\sum_{k \text{ typical}} Q_{i, k}} \\
    &\leq \left(1 + \frac{C \log^2 d}{d} \right) \left(\frac{\porb{\kappa_i}{\tau}}{\porb{\kappa_j}{\tau}} \right) \cdot \frac{\left(1 + \frac{C'}{d}\right)\left(1 + \frac{C \log^2 d}{d} \right) \sum_{k \text{ typical}} \left( \frac{\porb{\kappa_j}{\tau}}{\porb{\kappa_i}{\tau}} \right) \cdot Q_{i, k}}{\sum_{k \text{ typical}} Q_{i, k}} \\
    &\leq 1 + \frac{C'' \log^2 d}{d} 
\end{align*}
The first inequality uses \pref{lem:typical-ratio} to bound $\frac{Q_{i, \ell}}{Q_{j, \ell}}$. 
The second inequality uses the fact that the outlier entries of $Q_{i, *}$ and $Q_{j, *}$ only occupy an $O\left(\frac{1}{d}\right)$ fraction of the $\ell_1$ norm of each row (\pref{lem:outlier-trans-prob}). 
The third inequality again comes from an application of \pref{lem:typical-ratio} to relate $Q_{i, k}$ to $Q_{j, k}$ when $k$ is typical. 
The lower bound follows analogously. 
\end{proof}

\begin{proof}[Proof of \pref{lem:typical-ratio}]
    By definition:
    \begin{align*}
        \frac{\porb{\kappa_i}{\kappa_\ell}}{\porb{\kappa_j}{\kappa_\ell}} \cdot \left(\frac{\porb{\kappa_i}{\tau}}{\porb{\kappa_j}{\tau}} \right)^{-1}
        &= \frac{\ol{\Phi}_{\Beta{d - 1}} \left(T(\kappa_i, \kappa_\ell) \right)}{\ol{\Phi}_{\Beta{d - 1}} \left(T(\kappa_j, \kappa_\ell) \right)} \cdot \frac{\ol{\Phi}_{\Beta{d - 1}} \left(T(\kappa_j, \tau) \right)}{\ol{\Phi}_{\Beta{d - 1}} \left(T(\kappa_i, \tau) \right)} \numberthis \label{eq:ratio-ratios}
    \end{align*}
    Since $\kappa_i, \kappa_j, \kappa_{\ell}\in[\tau,\tau(1+\alpha)]$, by \pref{clm:tau-tau-worst-case}, all terms of the form $T(x,y)$ in the above are lower bounded by $T(\tau(1+\alpha),\tau(1+\alpha))$, which is lower bounded by a constant for large enough $d$.
    Thus, by \pref{lem:approx-tails-dpd}:
    \[
        \pref{eq:ratio-ratios} = \left(1 \pm O\left(\frac{\log d}{d}\right) \right) \cdot \frac{T(\kappa_j, \kappa_\ell) \cdot T(\kappa_i, \tau)}{T(\kappa_i, \kappa_\ell) \cdot T(\kappa_j, \tau)} \cdot \left( \frac{A}{B} \right)^{(d - 1)/2}
    \]
    where
    $
        A \coloneqq \left(1 - T(\kappa_i, \kappa_\ell)^2\right) \left(1 - T(\kappa_j, \tau)^2\right)
    $
    and 
    $
        B \coloneqq \left(1 - T(\kappa_j, \kappa_\ell)^2\right) \left(1 - T(\kappa_i, \tau)^2\right).
    $
    We show:
    \begin{align}
        \left| \frac{T(\kappa_j, \kappa_\ell) \cdot T(\kappa_i, \tau)}{T(\kappa_i, \kappa_\ell) \cdot T(\kappa_j, \tau)} - 1 \right| &\leq O(\alpha^2) \label{eq:minor-ratio} \\
        \left| \frac{A}{B} - 1 \right| &\leq O(\alpha^2) \label{eq:major-ratio}
    \end{align}
    where \pref{eq:minor-ratio} is proved in \pref{claim:H-ratio} and \pref{eq:major-ratio} is proved in \pref{claim:a-b-ratio}.
    Consequently,
    \[
        \frac{\porb{a}{x}}{\porb{b}{x}} \cdot \frac{\porb{b}{\tau}}{\porb{a}{\tau}} = \left(1 \pm O\left(\frac{\log d}{d}\right) \right) \cdot \left(1 \pm O(\alpha^2) \right) \cdot \left(1 \pm O(d\alpha^2) \right)
    \]
    The term of order $d\alpha^2$ dominates, and because $\alpha = O(\frac{\log d}{d})$ we conclude the desired result.
\end{proof}

\section{$2$-dimensional expansion of the random geometric complex}
\label{sec:wrapup}
In this section we prove \pref{thm:hdx}.
\begin{theorem} \label{thm:hdx-full}
    For every $0 < \eps < 1$, $0 < \eta < 2\eps$ and $d = \eta \log_{4/3} n$, if $\bH\sim\grg^{(2)}_d\parens*{n, n^{-1+\eps}}$, then every link of $\bH$ is a $\parens*{\frac{1}{2}-\delta}$-expander, and its $1$-skeleton is a $\parens*{1-\frac{4\delta}{1+2\delta}}$-expander with high probability where $\delta = \frac{1}{2}\cdot\frac{1-\sqrt{1-\exp\left(-2\log\frac{4}{3}\cdot(1-\eps)/\eta\right)}}{1+\sqrt{1-\exp\left(-2\log\frac{4}{3}\cdot(1-\eps)/\eta\right)}} - o_n(1)$.
\end{theorem}

One of the ingredients in the proof of \pref{thm:hdx-full} is the spectral expansion of random geometric graphs, which is a corollary of \pref{thm:random-restriction} and \pref{cor:decay-useful}:
\hypertarget{thm2:sphere}
{\restatetheorem{thm:sphere}}

\noindent The second ingredient is a bound on the second eigenvalue of the links, proved in \pref{sec:link-eigs}:
\restatetheorem{thm:main-thm-links}

\begin{proof}[Proof of \pref{thm:hdx-full}]
    To show that the links expand, we apply \pref{thm:main-thm-links} in combination with a union bound over all links.
    The second eigenvalue bound for the $1$-skeleton is then proved using \pref{thm:opp-trickling}, the trickling-down theorem.  
    Let $p = n^{-1+\eps}$, $d = \eta \log_{4/3} n$ and $\tau = \tau(p,d)$.

    Let $\bG\coloneqq\GG_{\tau}(\bv_1,\dots,\bv_n)$ to denote the geometric graph of the collection of vectors used to generate $\bH$.
    The number of vertices that fall in the neighborhood of a vertex $v$ within $\bG$ is $\boldm_v\sim\Binom(n, p)$, and hence $\boldm_v \ge m\coloneqq pn-2\sqrt{pn\log n}$ except with probability $o(1/n)$.
    For the rest of the proof, we condition on the event that $\boldm_v \ge m$ for all $v$, which happens with probability $1-o(1)$ by the union bound.

    The link $\bH_v$ of a vertex $v$ is obtained by taking $\bG_v$, the subgraph of $\bG$ induced by the neighborhood of $v$, and then  removing the isolate vertices. 
    Note that the isolated vertices need to be removed since when sampling a random complex, we remove all edges that are not in any triangles.
    Our goal is to control the second eigenvalue of all the links in $\bH$, and we do so by showing bounds on the second eigenvalue of $\bG_v$ for all $v$.
    The second eigenvalue bounds show that with high probability, for all $v$, the graph $\bG_v$ is connected, and hence has no isolated vertices.
    Consequently, $\bH_v$ is in fact equal to $\bG_v$ and the second eigenvalue bounds port over.

    As a first step, we show that $\bG_v$ satisfies the hypothesis of \pref{thm:main-thm-links}.
    In particular, we show for $q\coloneqq \Phi_{\Beta{d-1}}\parens*{\frac{\tau}{1+\tau}}$
    \[
        q \cdot m_v \gg \log^4 m_v \cdot \log^2 \frac{1}{q} \cdot \parens*{\frac{1+\tau}{\tau}}^4.  \numberthis \label{eq:link-condition-met}
    \]
    Using \pref{lem:approx-tails-dpd} and the fact that the tail function of a probability distribution is monotone decreasing, we can lower bound $q$:
    \[
        q \ge \ol{\Phi}_{\Beta{d-1}}\parens*{\frac{1}{2}} \ge \frac{2Z_{d-1}}{d-2}\cdot\parens*{\frac{3}{4}}^{(d-2)/2}\cdot\parens*{1-\frac{16\log d}{d-1}} \ge \Omega\parens*{\frac{1}{\sqrt{d}}}\cdot n^{-\eta/2},
    \]
    where the first inequality holds since $\tau \in (0,1]$, and the last inequality holds by definitions of $d$ and $Z_d$. 
    We now lower bound $\tau$ by a constant.
    By \pref{lem:approx-tails-dpd}:
    \[
        \ol{\Phi}_{\Beta{d}}\parens*{\sqrt{1-\exp\parens*{-\frac{(1-\eps)\log\frac{4}{3}}{\eta}}}} \ge \Omega\parens*{\frac{1}{\sqrt{d}}}\cdot\sqrt{p}\cdot\parens*{1 - O\parens*{\frac{\log d}{d}}} \ge p = \ol{\Phi}_{\Beta{d}}(\tau),
    \]
    where the first inequality holds by definition of $Z_d$.
    Since $\ol{\Phi}_{\Beta{d}}$ is a decreasing function,
    \[
        \tau\ge \sqrt{1-\exp\parens*{-\frac{(1-\eps)\log\frac{4}{3}}{\eta}}}.  \label{eq:tau-lower-bound}
    \]
    Consequently:
    \[
        \log^4 m_v \cdot \log^2\frac{1}{q} \cdot \parens*{\frac{1+\tau}{\tau}}^4 \le \log^6 n.    
    \]
    On the other hand, $qm_v\ge \Omega\parens*{\frac{1}{\sqrt{d}}}\cdot n^{\eps-\eta/2} \gg \log^6 n$, which establishes \pref{eq:link-condition-met}.
    By \pref{eq:link-condition-met} and \pref{thm:main-thm-links} with $\gamma = 2/\eps$, with probability at least $1-O(1/n^{2})$:
    \[
        \abs*{\lambda}_2\parens*{\wh{A}_{\bG_{v}}} \le \frac{\tau}{1+\tau} + o_{n}(1).
    \]
    By the union bound over all vertices, with probability $1-O(1/n)$:
    \[
        \abs*{\lambda}_2\parens*{\wh{A}_{\bG_{v}}} \le \frac{\tau}{1+\tau} + o_{n}(1) \qquad \forall v\in[n].
    \]
    Henceforth, we condition on the above.
    Since $\frac{\tau}{1+\tau}<1$, for all $v\in[n]$, each $\bG_v$ is connected and has no isolated vertices and hence $\bH_v=\bG_v$.
    Consequently
    \[
        \abs*{\lambda}_2\parens*{\wh{A}_{\bH_{v}}} \le \frac{\tau}{1+\tau} + o_{n}(1) \qquad \forall v\in[n].
    \]
    Assuming the $1$-skeleton $\bH^{(1)}$ is connected, by the trickling-down theorem (\pref{thm:opp-trickling}) it satisfies:
    \[
        \abs*{\lambda}_2\parens*{\wh{A}_{\bH^{(1)}}} \le \frac{\frac{\tau}{1+\tau} + o_n(1)}{1-\frac{\tau}{1+\tau} - o_n(1)} = \tau + o_n(1).
    \]
    It remains to bound $\tau$, $\tau/(1+\tau)$ and show $\bH^{(1)}$ is connected.
    By \pref{lem:approx-tails-dpd}, the following inequality must be satisfied:
    \[
        p \le \frac{Z_d}{\tau(d-1)}\cdot\parens*{1-\tau^2}^{(d-1)/2}.
    \]
    Since the right hand side of the above is a decreasing function of $\tau$ and plugging in $\sqrt{1-\exp\parens*{-\frac{2(1-\eps)\log\frac{4}{3}}{\eta}}}$ yields a value smaller than $p$, we know:
    \[
        \tau\le\sqrt{1-\exp\parens*{-\frac{2(1-\eps)\log\frac{4}{3}}{\eta}}} = 1-\frac{4\delta}{1+2\delta}.  \numberthis \label{eq:tau-upper-bound}
    \]
    The function $\tau/(1+\tau)$ is an increasing function and hence:
    \[
        \frac{\tau}{1+\tau} \le \frac{\sqrt{1-\exp\parens*{-\frac{2(1-\eps)\log\frac{4}{3}}{\eta}}}}{1+\sqrt{1-\exp\parens*{-\frac{2(1-\eps)\log\frac{4}{3}}{\eta}}}} + o_n(1) = \frac{1}{2} - \delta.
    \]
    Finally, to show $\bH^{(1)}$ is connected, it suffices to illustrate $\wt{\bH}^{(1)}$, a reweighted version of $\bH^{(1)}$, whose normalized adjacency matrix has a spectral gap.
    We use $\bG$ as our reweighting of $\bH^{(1)}$, which is valid since all edges in $\bG$ occur in $\bH^{(1)}$ with probability $1 - o_n(1)$.
    Indeed, for every vertex $v$ and neighbor $w$ the vertex $w$ has some neighbor $w'$ within $\bG_v$, which means $\{v,w,w'\}$ is a triangle in $\bH$ causing $\{v,w\}$ to appear in $\bH^{(1)}$.
    By our choice of parameters, the lower and upper bounds on $\tau$ shown in \pref{eq:tau-lower-bound} and \pref{eq:tau-upper-bound} respectively, and \hyperlink{thm2:sphere}{Theorem \ref*{thm:sphere}}, we know $|\lambda|_2\parens*{A_{\bG}} < \tau+o_n(1) < 1$, which implies $\bH^{(1)}$ is connected, which completes our proof.
\end{proof}

\section{Tightness of the tricking-down theorem}\label{sec:tight}
In this section we will show that the trickle-down theorem is tight:
\restateprop{prop:trickle-tight}

We prove the proposition by showing that a random geometric graph's adjacency matrix (when weighted in a regular way) has second eigenvalue at least $\tau$, and then prove that the random geometric complex indeed satisfies that regularity condition.
\begin{lemma}   \label{lem:grg-specgap-lower}
    Let $\bG\sim\grg_d(n,p)$ generated by vectors $\bv_1,\dots,\bv_n$, and let $W$ be the transition matrix of any time-reversible Markov chain on $\bG$ with stationary distribution $\pi$.
    Then with high probability $\lambda_2(W)\ge \tau - o_n(1) - O(\dtv{\pi}{U_n}^2)$ where $U_n$ is the uniform distribution on $[n]$.
\end{lemma}
\begin{proof}
    When $\dtv{\pi}{U_n}\ge 0.1$, the statement is vacuously true.
    Thus, we assume $\dtv{\pi}{U_n} < 0.1$ for the rest of this proof.
    We see that:
    \begin{align*}
        1 - \lambda_2(W)
        = \min_{\substack{f:V(\bG)\to\R^d\\ f \text{ non-constant}}} \frac{\E_{x\sim_W y}\norm*{f(x)-f(y)}^2}{\E_{x,y\sim\pi}\norm*{f(x)-f(y)}^2}
        \le \frac{\E_{x\sim_W y}\norm*{\bv_x-\bv_y}^2}{\E_{x,y\sim\pi}\norm*{\bv_x-\bv_y}^2}
        \le \frac{2(1-\tau(p,d))}{\E_{x,y\sim\pi}\norm*{\bv_x-\bv_y}^2}
        \numberthis \label{eq:spec-gap-bound}
    \end{align*}
    where the last inequality uses that for adjacent $x,y$, $\angles*{\bv_x,\bv_y}\ge\tau(p,d)$.
    To lower bound the denominator, observe:
    \begin{align*}
        \E_{x,y\sim\pi} \norm*{\bv_x-\bv_y}^2
        &= \sum_{x,y\in [n]} \pi(x)\pi(y)\parens*{2 - 2\angles*{\bv_x, \bv_y}}
        = 2\parens*{1 - \sum_{x,y\in[n]} \angles*{ \pi(x) \bv_x, \pi(y) \bv_y } } \\
        &= 2 \parens*{1 - \norm*{\sum_{x\in[n]} \pi(x)\bv_x}^2 }
        = 2\parens*{1 - \norm*{ \sum_{x\in[n]} \frac{1}{n} \bv_x + \sum_{x\in[n]} \left(\pi(x)-\frac{1}{n}\right)\bv_x }^2 } \\
        &\ge 2\parens*{1 - \norm*{ \frac{1}{n} \sum_{x\in[n]} \bv_x }^2
        - 4 \norm*{ \frac{1}{n} \sum_{x\in[n]} \bv_x }\cdot\dtv{\pi}{U_n}
        - 4 \dtv{\pi}{U_n}^2 }.
    \end{align*}
    By standard concentration arguments, $\norm*{\frac{1}{n}\sum_{x\in[n]} \bv_x}$ is $o_n(1)$ with high probability.
    Plugging in the lower bound into \pref{eq:spec-gap-bound} tells us:
    \begin{align*}
        1 - \lambda_2(W) \le 1 - \tau(p,d) + o_n(1) + O\parens*{ \dtv{\pi}{U_n}^2 },
    \end{align*}
    which can be rearranged into the desired inequality.
\end{proof}

Armed with this lemma we can prove \pref{prop:trickle-tight}.
\begin{proof}[Proof of \pref{prop:trickle-tight}]
    Let $\tau = \frac{\lambda}{1-\lambda}$, which is in $(0,1)$ for $\lambda\in\left(0,\frac{1}{2}\right)$.
    Using the bounds from \pref{lem:approx-tails-dpd} we can choose $n$ and $d = \Theta(\log n)$ such that for $p = \ol{\Phi}_{\Beta{d}}(\tau)$, we have $\frac{np^2}{2} \gg \poly\log n$.

    Let $\bH\sim\grg^{(2)}_d(n,p)$.
    Since each link contains $\Binom(n-1,p)$ vertices, and $(n-1)p\gg\poly\log n$, every link has $(n-1)p(1\pm o_n(1))\ge m\coloneqq np/2$ vertices with probability $1-O(n^{-1})$.
    Also, $\ol{\Phi}_{\Beta{d}}\parens*{ \frac{\tau}{1+\tau} } \ge \ol{\Phi}_{\Beta{d}}(\tau)$, so $m\cdot \ol{\Phi}_{\Beta{d}}\parens*{ \frac{\tau}{1+\tau} } \ge \frac{np^2}{2} \gg \poly\log m$, and so the conditions of \pref{thm:main-thm-links} are met so that by a union bound we can conclude that all links have second eigenvalue at most $\frac{\tau}{1+\tau}+o(1) = \lambda + o(1)$.

    Simultaneously, for any pair of vertices the number of triangles they participate in are within a multiplicative factor of $1\pm\frac{\log n}{\sqrt{p^2 n}} = 1 \pm o_n(1)$ of each other, as we argue in the next paragraph.
    Since the stationary distribution $\pi$ of the random walk on $\bG$ weighted according to $\bH^{(1)}$, the $1$-skeleton of $\bH$, puts mass on vertex $v$ proportional to the number of triangles $v$ participates in, it must be the case that $\pi(v) = (1\pm o_n(1))\cdot\frac{1}{n}$.
    Consequently, $\dtv{\pi}{U_n} = o_n(1)$, and by \pref{lem:grg-specgap-lower}, $\lambda_2(\bH^{(1)})\ge\tau - o_n(1) = \frac{\lambda}{1-\lambda} - o_n(1)$.

    We now show concentration for the number of triangles that contain a vertex.
    Indeed, the number of triangles that a vertex $v$ participates in is equal to the number of edges in its link.
    Using, $\boldm_v$ to denote the number of vertices in the link of $v$, $\deg(u)$ to denote the degree of a vertex $u$ within the link of $v$, and $\bkappa$ to denote the collection of shells that vertices in the link of $v$ lie in, we have:
    \[
        |E(\mathrm{Link}(v))| = \frac{1}{2} \sum_{i=1}^{\boldm_v} \deg_v(u).
    \]
    Henceforth we condition on $\boldm_v$ achieving some value in $(1\pm o_n(1))p(n-1)$.
    The average degree of a vertex $u$ within the link of $v$ is at least $\frac{np^2}{2}$, and hence by Bernstein's inequality each $\deg_v(u) = (1\pm o_n(1))\E \bracks*{\deg_v(u)|\bkappa_u}$ except with probability $O(n^{-4})$ since $\deg_v(u)|\bkappa_u$ is a sum of independent indicator random variables.
    The random variables $\E\bracks*{\deg_v(u)|\bkappa_u}$ are independent and distributed as $p(\bkappa_u)\boldm_v$ where $p(\bkappa_u)$ is the probability that a uniformly random vector in $\scap_p(v)$ falls in $\scap_p(u)$ where $\angles*{u,v} = \bkappa_u$.
    We can show with Bernstein's inequality that:
    \[
        \sum_{u=1}^{\boldm_v} \E\bracks*{\deg_v(u)|\bkappa_u} = (1\pm o(1)) \E[p(\bkappa_u)] \boldm_v^2  
    \]
    except with probability $O(n^{-4})$.
    By the union bound, with probability $O(n^{-1})$ for all $v\in[n]$,
    \[
        |E(\mathrm{Link}(v))| = \frac{1\pm o_n(1)}{2} \E[p(\bkappa_u)] \boldm_v^2,
    \]
    which completes the proof.
\end{proof}

\ifdim\displayauthors=1pt
\section*{Acknowledgments}
We thank Vishesh Jain for a game-changing pep talk on the power of the trace method.
Thanks also to Ryan O'Donnell, Prasad Raghavendra, and Nikhil Srivastava for inspiring conversations about random graphs and high-dimensional expanders.
We would like to thank Omar Alrabiah for feedback on an earlier version of the paper.
S.M.\ and T.S.\ were visiting the ``Computational Complexity of Statistical Inference'' program at Simons Institute for the Theory of Computing when some of the work was conducted, and would like to thank the institute for their hospitality and support.
T.S. thanks Samory Kpotufe for suggesting to a version of \pref{question:block-model}.
\else
\fi

\bibliographystyle{alpha}
\bibliography{main}

\newcommand{\etalchar}[1]{$^{#1}$}
\begin{thebibliography}{AdlVKK03}

\bibitem[ABS15]{ABS15}
Sanjeev Arora, Boaz Barak, and David Steurer.
\newblock Subexponential algorithms for unique games and related problems.
\newblock {\em Journal of the ACM (JACM)}, 62(5):1--25, 2015.

\bibitem[AdlVKK03]{AFKK03}
Noga Alon, W~Fernandez de~la Vega, Ravi Kannan, and Marek Karpinski.
\newblock {Random sampling and approximation of MAX-CSPs}.
\newblock {\em Journal of computer and system sciences}, 67(2):212--243, 2003.

\bibitem[AL20]{AL20}
Vedat~Levi Alev and Lap~Chi Lau.
\newblock Improved analysis of higher order random walks and applications.
\newblock In {\em Proceedings of the 52nd Annual ACM SIGACT Symposium on Theory
  of Computing}, pages 1198--1211, 2020.

\bibitem[Alo86]{Alon86}
Noga Alon.
\newblock Eigenvalues and expanders.
\newblock {\em Combinatorica}, 6(2):83--96, 1986.

\bibitem[ALOV19]{ALOV19}
Nima Anari, Kuikui Liu, Shayan {Oveis Gharan}, and Cynthia Vinzant.
\newblock {Log-concave polynomials II: high-dimensional walks and an FPRAS for
  counting bases of a matroid}.
\newblock In {\em Proceedings of the 51st Annual ACM SIGACT Symposium on Theory
  of Computing}, pages 1--12, 2019.

\bibitem[BDER16]{BDER16}
S{\'e}bastien Bubeck, Jian Ding, Ronen Eldan, and Mikl{\'o}s~Z R{\'a}cz.
\newblock Testing for high-dimensional geometry in random graphs.
\newblock {\em Random Structures \& Algorithms}, 49(3):503--532, 2016.

\bibitem[BGL{\etalchar{+}}14]{BGL14}
Dominique Bakry, Ivan Gentil, Michel Ledoux, et~al.
\newblock {\em Analysis and geometry of Markov diffusion operators}, volume
  103.
\newblock Springer, 2014.

\bibitem[BHHS11]{BHHS11}
Boaz Barak, Moritz Hardt, Thomas Holenstein, and David Steurer.
\newblock Subsampling mathematical relaxations and average-case complexity.
\newblock In {\em Proceedings of the twenty-second annual {ACM-SIAM} Symposium
  on Discrete Algorithms}, pages 512--531. SIAM, 2011.

\bibitem[BHK11]{BHK11}
Eric Babson, Christopher Hoffman, and Matthew Kahle.
\newblock The fundamental group of random 2-complexes.
\newblock {\em Journal of the American Mathematical Society}, 24(1):1--28,
  2011.

\bibitem[BKW14]{BKW14}
Itai Benjamini, Gady Kozma, and Nicholas Wormald.
\newblock The mixing time of the giant component of a random graph.
\newblock {\em Random Structures \& Algorithms}, 45(3):383--407, 2014.

\bibitem[BL06]{BL06}
Yonatan Bilu and Nathan Linial.
\newblock Lifts, discrepancy and nearly optimal spectral gap.
\newblock {\em Combinatorica}, 26(5):495--519, 2006.

\bibitem[Bor13]{Bordenave13}
Charles Bordenave.
\newblock On {E}uclidean random matrices in high dimension.
\newblock {\em Electronic Communications in Probability}, 18:1--8, 2013.

\bibitem[CLP20]{CLP20}
Michael Chapman, Nati Linial, and Yuval Peled.
\newblock Expander graphs -- both local and global.
\newblock {\em Combinatorica}, 40(4):473--509, 2020.

\bibitem[Con19]{Conlon19}
David Conlon.
\newblock Hypergraph expanders from {C}ayley graphs.
\newblock {\em Israel Journal of Mathematics}, 233(1):49--65, 2019.

\bibitem[CS13]{CS13}
Xiuyuan Cheng and Amit Singer.
\newblock The spectrum of random inner-product kernel matrices.
\newblock {\em Random Matrices: Theory and Applications}, 2(04):1350010, 2013.

\bibitem[CS{\.Z}03]{CSZ03}
Donald~I Cartwright, Patrick Sol{\'e}, and Andrzej {\.Z}uk.
\newblock Ramanujan geometries of type $\tilde{A}_n$.
\newblock {\em Discrete mathematics}, 269(1-3):35--43, 2003.

\bibitem[CTZ20]{CTZ20}
David Conlon, Jonathan Tidor, and Yufei Zhao.
\newblock Hypergraph expanders of all uniformities from {C}ayley graphs.
\newblock {\em Proceedings of the London Mathematical Society},
  121(5):1311--1336, 2020.

\bibitem[DEKL14]{DEKL14}
Jean Dolbeault, Maria~J Esteban, Michal Kowalczyk, and Michael Loss.
\newblock Sharp interpolation inequalities on the sphere: new methods and
  consequences.
\newblock In {\em Partial Differential Equations: Theory, Control and
  Approximation}, pages 225--242. Springer, 2014.

\bibitem[DEL{\etalchar{+}}22]{DELLM22}
Irit Dinur, Shai Evra, Ron Livne, Alexander Lubotzky, and Shahar Mozes.
\newblock Locally testable codes with constant rate, distance, and locality.
\newblock In {\em Proceedings of the 54th Annual ACM SIGACT Symposium on Theory
  of Computing}, pages 357--374, 2022.

\bibitem[Dem96]{Dembo96}
Amir Dembo.
\newblock Moderate deviations for martingales with bounded jumps.
\newblock {\em Electronic Communications in Probability}, 1:11--17, 1996.

\bibitem[DGLU11]{DGLU11}
Luc Devroye, Andr{\'a}s Gy{\"o}rgy, G{\'a}bor Lugosi, and Frederic Udina.
\newblock High-dimensional random geometric graphs and their clique number.
\newblock {\em Electronic Journal of Probability}, 16:2481--2508, 2011.

\bibitem[DHLV22]{DHLV22}
Irit Dinur, Min-Hsiu Hsieh, Ting-Chun Lin, and Thomas Vidick.
\newblock {Good Quantum LDPC Codes with Linear Time Decoders}.
\newblock {\em arXiv preprint arXiv:2206.07750}, 2022.

\bibitem[Din07]{Dinur07}
Irit Dinur.
\newblock {The PCP theorem by gap amplification}.
\newblock {\em Journal of the ACM (JACM)}, 54(3):12--es, 2007.

\bibitem[DK17]{DK17}
Irit Dinur and Tali Kaufman.
\newblock High dimensional expanders imply agreement expanders.
\newblock In {\em 2017 IEEE 58th Annual Symposium on Foundations of Computer
  Science (FOCS)}, pages 974--985. IEEE, 2017.

\bibitem[DV13]{DV13}
Yen Do and Van Vu.
\newblock The spectrum of random kernel matrices: universality results for
  rough and varying kernels.
\newblock {\em Random Matrices: Theory and Applications}, 2(03):1350005, 2013.

\bibitem[EK10]{Karoui10}
Noureddine El~Karoui.
\newblock The spectrum of kernel random matrices.
\newblock {\em The Annals of Statistics}, 38(1):1--50, 2010.

\bibitem[EK16]{EK16}
Shai Evra and Tali Kaufman.
\newblock Bounded degree cosystolic expanders of every dimension.
\newblock In {\em Proceedings of the forty-eighth annual ACM symposium on
  Theory of Computing}, pages 36--48, 2016.

\bibitem[FGL{\etalchar{+}}12]{FGLNP12}
Jacob Fox, Mikhail Gromov, Vincent Lafforgue, Assaf Naor, and J{\'a}nos Pach.
\newblock Overlap properties of geometric expanders.
\newblock {\em Journal f{\"u}r die reine und angewandte Mathematik (Crelles
  Journal)}, 2012(671):49--83, 2012.

\bibitem[FI20]{FI20}
Ehud Friedgut and Yonatan Iluz.
\newblock Hyper-regular graphs and high dimensional expanders.
\newblock {\em arXiv preprint arXiv:2010.03829}, 2020.

\bibitem[FM19]{FM19}
Zhou Fan and Andrea Montanari.
\newblock The spectral norm of random inner-product kernel matrices.
\newblock {\em Probability Theory and Related Fields}, 173(1):27--85, 2019.

\bibitem[FR08]{FR08}
Nikolaos Fountoulakis and Bruce~A Reed.
\newblock The evolution of the mixing rate of a simple random walk on the giant
  component of a random graph.
\newblock {\em Random Structures \& Algorithms}, 33(1):68--86, 2008.

\bibitem[Fri93]{Friedman93}
Joel Friedman.
\newblock Some geometric aspects of graphs and their eigenfunctions.
\newblock {\em Duke Mathematical Journal}, 69(3):487--525, 1993.

\bibitem[Fri08]{Friedman08}
Joel Friedman.
\newblock {\em A proof of {A}lon's second eigenvalue conjecture and related
  problems}.
\newblock American Mathematical Soc., 2008.

\bibitem[GG81]{GG81}
Ofer Gabber and Zvi Galil.
\newblock Explicit constructions of linear-sized superconcentrators.
\newblock {\em Journal of Computer and System Sciences}, 22(3):407--420, 1981.

\bibitem[GGR98]{GGR98}
Oded Goldreich, Shafi Goldwasser, and Dana Ron.
\newblock Property testing and its connection to learning and approximation.
\newblock {\em Journal of the {ACM} ({JACM})}, 45(4):653--750, 1998.

\bibitem[Gol21]{Gol21}
Louis Golowich.
\newblock Improved product-based high-dimensional expanders.
\newblock In {\em Approximation, Randomization, and Combinatorial Optimization.
  Algorithms and Techniques ({APPROX}/{RANDOM} 2021)}. Schloss
  Dagstuhl-Leibniz-Zentrum f{\"u}r Informatik, 2021.

\bibitem[Gro10]{G10}
Mikhail Gromov.
\newblock Singularities, expanders and topology of maps. part 2: From
  combinatorics to topology via algebraic isoperimetry.
\newblock {\em Geometric and Functional Analysis}, 20(2):416--526, 2010.

\bibitem[HKP{\etalchar{+}}17]{HKPRSS17}
Samuel~B Hopkins, Pravesh~K Kothari, Aaron Potechin, Prasad Raghavendra, Tselil
  Schramm, and David Steurer.
\newblock The power of sum-of-squares for detecting hidden structures.
\newblock In {\em 2017 {IEEE} 58th Annual Symposium on Foundations of Computer
  Science ({FOCS})}, pages 720--731. {IEEE}, 2017.

\bibitem[HLW06]{HLW06}
Shlomo Hoory, Nathan Linial, and Avi Wigderson.
\newblock Expander graphs and their applications.
\newblock {\em Bulletin of the American Mathematical Society}, 43(4):439--561,
  2006.

\bibitem[INW94]{INW94}
Russell Impagliazzo, Noam Nisan, and Avi Wigderson.
\newblock Pseudorandomness for network algorithms.
\newblock In {\em Proceedings of the twenty-sixth annual {ACM} symposium on
  Theory of computing}, pages 356--364, 1994.

\bibitem[KB93]{KM93}
Andrey Kolmogorov and Ya~M. Barzdin.
\newblock On the realization of networks in three-dimensional space.
\newblock In {\em Selected Works of AN Kolmogorov}, pages 194--202. Springer,
  1993.

\bibitem[KG00]{KG00}
Vladimir Koltchinskii and Evarist Gin{\'e}.
\newblock Random matrix approximation of spectra of integral operators.
\newblock {\em Bernoulli}, pages 113--167, 2000.

\bibitem[KKL14]{KKL14}
Tali Kaufman, David Kazhdan, and Alexander Lubotzky.
\newblock Ramanujan complexes and bounded degree topological expanders.
\newblock In {\em 2014 IEEE 55th Annual Symposium on Foundations of Computer
  Science}, pages 484--493. IEEE, 2014.

\bibitem[KO18]{KO18}
Tali Kaufman and Izhar Oppenheim.
\newblock Construction of new local spectral high dimensional expanders.
\newblock In {\em Proceedings of the 50th Annual ACM SIGACT Symposium on Theory
  of Computing}, pages 773--786, 2018.

\bibitem[LH22]{LH22}
Ting-Chun Lin and Min-Hsiu Hsieh.
\newblock {$ c^{3}$-Local Testable Codes from Lossless Expanders}.
\newblock {\em arXiv preprint arXiv:2201.11369}, 2022.

\bibitem[Li04]{Li04}
Wen-Ching~Winnie Li.
\newblock Ramanujan hypergraphs.
\newblock {\em Geometric \& Functional Analysis GAFA}, 14(2):380--399, 2004.

\bibitem[Lin]{LinialSimons}
Nati Linial.
\newblock Some geometric perspectives on combinatorics: High-dimensional, local
  and local-to-global {II}.
\newblock Simons Institute big data bootcamp workshop,
  \url{https://www.youtube.com/watch?v=Kt9SWmIPsH8}.

\bibitem[LLR95]{LLR95}
Nathan Linial, Eran London, and Yuri Rabinovich.
\newblock The geometry of graphs and some of its algorithmic applications.
\newblock {\em Combinatorica}, 15(2):215--245, 1995.

\bibitem[LM06]{LM06}
Nathan Linial and Roy Meshulam.
\newblock Homological connectivity of random 2-complexes.
\newblock {\em Combinatorica}, 26(4):475--487, 2006.

\bibitem[LMY20]{LMY20}
Siqi Liu, Sidhanth Mohanty, and Elizabeth Yang.
\newblock High-dimensional expanders from expanders.
\newblock In {\em 11th Innovations in Theoretical Computer Science Conference
  ({ITCS} 2020)}. Schloss {D}agstuhl-{L}eibniz-Zentrum f{\"u}r Informatik,
  2020.

\bibitem[LP16]{LP16}
Nathan Linial and Yuval Peled.
\newblock On the phase transition in random simplicial complexes.
\newblock {\em Annals of mathematics}, pages 745--773, 2016.

\bibitem[LPS88]{LPS88}
Alexander Lubotzky, Ralph Phillips, and Peter Sarnak.
\newblock Ramanujan graphs.
\newblock {\em Combinatorica}, 8(3):261--277, 1988.

\bibitem[LRS15]{LRS15}
James~R Lee, Prasad Raghavendra, and David Steurer.
\newblock Lower bounds on the size of semidefinite programming relaxations.
\newblock In {\em Proceedings of the forty-seventh annual {ACM} symposium on
  Theory of computing}, pages 567--576, 2015.

\bibitem[LSV05a]{LSV05a}
Alexander Lubotzky, Beth Samuels, and Uzi Vishne.
\newblock Explicit constructions of ramanujan complexes of type ad.
\newblock {\em European Journal of Combinatorics}, 26(6):965--993, 2005.

\bibitem[LSV05b]{LSV05b}
Alexander Lubotzky, Beth Samuels, and Uzi Vishne.
\newblock Ramanujan complexes of type{\~a} d.
\newblock {\em Israel journal of Mathematics}, 149(1):267--299, 2005.

\bibitem[Lub]{LubotzkySTOC}
Alex Lubotzky.
\newblock High-dimensional expanders and property testing.
\newblock {STOC} workshop on advances in coding theory,
  \url{https://youtu.be/P5vs4ARRQjU?t=2967}.

\bibitem[Lub12]{Lubotzky12}
Alexander Lubotzky.
\newblock Expander graphs in pure and applied mathematics.
\newblock {\em Bulletin of the American Mathematical Society}, 49(1):113--162,
  2012.

\bibitem[Lub18]{Lubotzky18}
Alexander Lubotzky.
\newblock High dimensional expanders.
\newblock In {\em Proceedings of the International Congress of Mathematicians:
  Rio de Janeiro 2018}, pages 705--730. World Scientific, 2018.

\bibitem[LY22]{LY22}
Yue~M Lu and Horng-Tzer Yau.
\newblock An equivalence principle for the spectrum of random inner-product
  kernel matrices.
\newblock {\em arXiv preprint arXiv:2205.06308}, 2022.

\bibitem[LZ22]{LZ22}
Anthony Leverrier and Gilles Z{\'e}mor.
\newblock Quantum tanner codes.
\newblock {\em arXiv preprint arXiv:2202.13641}, 2022.

\bibitem[Mar73]{Margulis73}
Grigorii~Aleksandrovich Margulis.
\newblock Explicit constructions of concentrators.
\newblock {\em Problemy Peredachi Informatsii}, 9(4):71--80, 1973.

\bibitem[Mar88]{Margulis88}
Grigory Margulis.
\newblock Explicit group-theoretic constructions of combinatorial schemes and
  their applications in the construction of expanders and concentrators.
\newblock {\em Problemy Peredachi Informatsii}, 24(1):51--60, 1988.

\bibitem[MOP20]{MOP20}
Sidhanth Mohanty, Ryan O'Donnell, and Pedro Paredes.
\newblock {Explicit near-{R}amanujan graphs of every degree}.
\newblock In {\em Proceedings of the 52nd Annual ACM SIGACT Symposium on Theory
  of Computing}, pages 510--523, 2020.

\bibitem[MW09]{MW09}
Roy Meshulam and Nathan Wallach.
\newblock Homological connectivity of random $k$-dimensional complexes.
\newblock {\em Random Structures \& Algorithms}, 34(3):408--417, 2009.

\bibitem[Nil91]{Nilli91}
Alon Nilli.
\newblock On the second eigenvalue of a graph.
\newblock {\em Discrete Mathematics}, 91(2):207--210, 1991.

\bibitem[OP22]{OP22}
Ryan O'Donnell and Kevin Pratt.
\newblock {High-Dimensional Expanders from Chevalley Groups}.
\newblock {\em arXiv preprint arXiv:2203.03705}, 2022.

\bibitem[Opp18]{Opp18}
Izhar Oppenheim.
\newblock {Local spectral expansion approach to high dimensional expanders part
  I: Descent of spectral gaps}.
\newblock {\em Discrete \& Computational Geometry}, 59(2):293--330, 2018.

\bibitem[Pen03]{Pen03}
Mathew Penrose.
\newblock {\em Random geometric graphs}, volume~5.
\newblock OUP Oxford, 2003.

\bibitem[Pin73]{Pin73}
Mark~S Pinsker.
\newblock On the complexity of a concentrator.
\newblock In {\em 7th International Telegraffic Conference}, volume~4, pages
  1--318. Citeseer, 1973.

\bibitem[PK22]{PK22}
Pavel Panteleev and Gleb Kalachev.
\newblock Asymptotically good quantum and locally testable classical ldpc
  codes.
\newblock In {\em Proceedings of the 54th Annual ACM SIGACT Symposium on Theory
  of Computing}, pages 375--388, 2022.

\bibitem[SS96]{SS96}
Michael Sipser and Daniel~A Spielman.
\newblock Expander codes.
\newblock {\em IEEE transactions on Information Theory}, 42(6):1710--1722,
  1996.

\end{thebibliography}

\appendix
\section{Azuma-Hoeffding for continuous processes}
\label{app:cont-ah}
\restatelemma{lem:cont-AH}
\begin{proof}[Proof of \pref{lem:cont-AH}]
The proof is the same as that of the standard Azuma-Hoeffding inequality.
Without loss of generality assume $\bX_0 = 0$.
From Markov's inequality,
\begin{align*}
\Pr[|\bX_s| \ge x]\le e^{-\theta x}\left(\E[\exp(\theta \bX_s)] + \E[\exp(-\theta\bX_s)]\right)
\end{align*}
We have that $\bX_{t + dt} = \bX_{t} + d\bX_t$, and so for any $\theta$,
\begin{align*}
d\E[\exp(\theta \bX_t)] 
&=\E[\exp(\theta\bX_{t+dt}) - \exp(\theta\bX_t)] \\
&= \E[\exp(\theta \bX_t)] \cdot \E[\exp(\theta d\bX_t) - 1 \mid \calF_t] \\
&\le \E[\exp(\theta \bX_t)] \cdot \left(\exp\left(\tfrac{\theta^2 \sigma_t^2dt}{K}\right) - 1\right)
= \E[\exp(\theta \bX_t)] \cdot \tfrac{\theta^2 \sigma_t^2 dt}{K},
\end{align*}
since the higher-order terms in the Taylor expansion of $\exp$ go to zero.
Hence, we conclude that $d \log \E[\exp(\theta \bX_t)] \le \frac{\theta^2\sigma_t^2 dt}{K}$, and thus $\log \E[\exp(\theta\bX_s)] = \int_0^s d\log\E \exp(\theta\bX_t) \le \int_0^s \frac{\theta^2\sigma_t^2 dt}{K}$ (where we used that $\bX_0 = 0$).
In conclusion,
\[
\E[\exp(\theta \bX_s)] \le \exp\left(\frac{\theta^2\int_{0}^s\sigma_t^2 dt}{K}\right),
\]
and choosing $\theta = \frac{Kx}{2\int_0^s \sigma_t^2 dt}$, and then repeating the argument with $-\bX_s$, completes the proof.
\end{proof}

\section{Bounding the expected contribution of outlier shells}
\label{app:n-d-ratio}

\restatelemma{lem:outlier-mean}
\begin{proof}
	For the function $T(x,y)\coloneqq\frac{\tau-xy}{\sqrt{1-x^2}\sqrt{1-y^2}}$, we are interested in bounding
	\[
		\frac{\displaystyle\int_{\tau(1+\alpha)}^1 (1-y^2)^{(d-3)/2}\cdot\ol{\Phi}_{\Beta{d-1}}(T(x,y))\, dy }{\displaystyle\int_{\tau}^1 (1-y^2)^{(d-3)/2}\cdot\ol{\Phi}_{\Beta{d-1}}(T(x,y))\, dy} \le \frac{\displaystyle\int_{\tau(1+\alpha)}^1 (1-y^2)^{(d-3)/2}\cdot\ol{\Phi}_{\Beta{d-1}}(T(x,y))\, dy }{\displaystyle\int_{\tau}^{\tau(1+\alpha/2)} (1-y^2)^{(d-3)/2}\cdot\ol{\Phi}_{\Beta{d-1}}(T(x,y))\, dy}
	    \numberthis	\label{eq:integral-ratio}
	\]
	For any $y\in[\tau(1+\alpha),1]$ and $z\in[\tau,\tau(1+\alpha/2)]$, we show that the expression
	\[
		\frac{(1-y^2)^{(d-3)/2}\cdot\ol{\Phi}_{\Beta{d-1}}(T(x,y))}{(1-z^2)^{(d-3)/2}\cdot\ol{\Phi}_{\Beta{d-1}}(T(x,z))}
		\numberthis	\label{eq:ratio-individual}
	\]
	is bounded by $\alpha\tau/d$.
	Consequently, \pref{eq:integral-ratio} is bounded by $\frac{2}{d}$.

	We now prove \pref{eq:ratio-individual}.
	Recall that by \pref{clm:tau-tau-worst-case}, $T(x,y)$ is decreasing in $x$ and $y$ on $[\tau,1]\times[\tau,1]$.
	For fixed $x$, let $y_*$ be chosen such that $T(x,y_*) = \frac{4}{\sqrt{d}}$.
	We split into cases depending on where $y$ and $z$ fall with respect to $y_*$.

	\noindent {\bf Case 1: $y_* \le z$}.  In this case, the numerator of \pref{eq:ratio-individual} can be bounded by $(1-y^2)^{(d-3)/2}$ and the denominator can be lower bounded by $\beta\cdot(1-z^2)^{(d-3)/2}$ for some constant $\beta > 0$.
	Combined with the fact that the derivative of $1-y^2$ is $-2y$ which is at most $-2\tau$ and our choice of $\alpha$, we have:
	\[
		\pref{eq:ratio-individual} \le \frac{1}{\beta}\cdot\parens*{\frac{1-y^2}{1-z^2}}^{(d-3)/2}
		\le \frac{(1-\alpha\tau^2)^{(d-3)/2}}{\beta} \le \frac{\alpha\tau}{d}.
	\]

	\noindent {\bf Case 2: $y_* \ge y$}.
	In this case, via \pref{lem:approx-tails-dpd}, the numerator of \pref{eq:ratio-individual} can be bounded by
	\begin{align*}
		\parens*{1-y^2}^{(d-3)/2}\cdot\frac{Z_{d-1}}{T(x,y)(d-2)}\parens*{1-T(x,y)^2}^{(d-2)/2}
		= \parens*{\frac{1-y^2-\tau^2+2\tau xy}{1-x^2}}^{(d-3)/2} \cdot \frac{Z_{d-1}\sqrt{1-T(x,y)^2}}{T(x,y)(d-2)}
	\end{align*}
	and similarly the denominator can be lower bounded by
	\[
		\beta\cdot\parens*{\frac{1-z^2-\tau^2+2\tau xz}{1-x^2}}^{(d-3)/2} \cdot \frac{Z_{d-1}\sqrt{1-T(x,z)^2}}{T(x,z)(d-2)}
	\]
	for some constant $\beta > 0$.
	Consequently,
	\[
		\pref{eq:ratio-individual} \le \frac{1}{\beta}\cdot\parens*{\frac{1-y^2-\tau^2+2\tau xy}{1-z^2-\tau^2+2\tau xz}}^{(d-3)/2} \cdot \frac{T(x,z)\sqrt{1-T(x,y)^2}}{T(x,y)\sqrt{1-T(x,z)^2}} 
		\le O\parens*{\sqrt{d}} \cdot \parens*{\frac{1-y^2-\tau^2+2\tau xy}{1-z^2-\tau^2+2\tau xz}}^{(d-3)/2}
	\]
	The derivative of the expression $G(z)\coloneqq 1 - z^2 - \tau^2 + 2\tau xz$ is $-2z + 2\tau x$, which is bounded by $-2\tau(1-x) \le -2\tau(1-\xbound)$ when $ \tau \le z $.
	By the derivative bound and our choice of $\alpha$ the above expression is at most:
	\[
		O\parens*{\sqrt{d}}\cdot\parens*{\frac{3-2\tau(1-\xbound)(y-z)}{3}}^{(d-3)/2} \le O\parens*{\sqrt{d}}\cdot\parens*{1-\frac{(1-\xbound)\alpha\tau^2}{3}}^{(d-3)/2} \le \frac{\alpha\tau}{d}.
	\]

	\noindent {\bf Case 3: $y \le y_* \le z$}.
	In this case:
	\begin{align*}
		\pref{eq:ratio-individual} &= \frac{(1-y^2)^{(d-3)/2}\cdot\ol{\Phi}_{\Beta{d-1}}(T(x,y))}{(1-y_*^2)^{(d-3)/2}\cdot\ol{\Phi}_{\Beta{d-1}}(T(x,y_*))}\cdot \frac{(1-y_*^2)^{(d-3)/2}\cdot\ol{\Phi}_{\Beta{d-1}}(T(x,y_*))}{(1-z^2)^{(d-3)/2}\cdot\ol{\Phi}_{\Beta{d-1}}(T(x,z))}.
	\end{align*}
	By an identical calculation to the case where $y_*\ge y$, the first part of the above product is bounded by
	\[
		O\parens*{\sqrt{d}}\cdot\parens*{1-\frac{2\tau(1-\eta)(y_*-y)}{3}}^{(d-3)/2}
	\]
	and the second part is bounded by
	\[
		O(1)\cdot\parens*{1 - 2(z-y_*)\tau}^{(d-3)/2}.	
	\]
	Either $y_*-y > \alpha\tau/4$ or $z-y_* > \alpha\tau/4$, and so
	\[
		\pref{eq:ratio-individual}
		\le O\parens*{\sqrt{d}}\cdot \max\braces*{ 1-\frac{(1-\eta)\alpha\tau^2}{6}, 1-\frac{\alpha\tau^2}{4} }^{(d-3)/2}
		\le \parens*{1-\frac{(1-\eta)\alpha\tau^2}{6}}^{(d-3)/2} \le \frac{\alpha\tau}{d},
	\]
	which completes the proof.
\end{proof}
\section{Computations for evaluating ratios of typical entries} \label{sec:code}

\begin{claim} \label{claim:H-ratio}
	Let $\kappa_x = \tau(1 + \delta_x)$, where $\alpha$ is the typicality threshold, and $x \in \{i, j, \ell\}$. Then,
	\[
		\left| \frac{T(\kappa_j, \kappa_\ell) \cdot T(\kappa_i, \tau)}{T(\kappa_i, \kappa_\ell) \cdot T(\kappa_j, \tau)} - 1 \right| \leq O(\alpha^2)
	\]
\end{claim}
\begin{proof}
	We first simplify $\frac{T(\kappa_j, \kappa_\ell) \cdot T(\kappa_i, \tau)}{T(\kappa_i, \kappa_\ell) \cdot T(\kappa_j, \tau)}$ by expanding the expression for $T(\cdot, \cdot)$ and cancelling terms:
	\begin{align*}
		\frac{T(\kappa_j, \kappa_\ell) \cdot T(\kappa_i, \tau)}{T(\kappa_i, \kappa_\ell) \cdot T(\kappa_j, \tau)} &= \frac{(\tau - \kappa_j \kappa_\ell)(1 - \kappa_i)}{(\tau - \kappa_i \kappa_\ell)(1 - \kappa_j)} \\
		\frac{T(\kappa_j, \kappa_\ell) \cdot T(\kappa_i, \tau)}{T(\kappa_i, \kappa_\ell) \cdot T(\kappa_j, \tau)} - 1 &= \frac{(\tau - \kappa_j \kappa_\ell)(1 - \kappa_i) - (\tau - \kappa_i \kappa_\ell)(1 - \kappa_j)}{(\tau - \kappa_i \kappa_\ell)(1 - \kappa_j)} 
	\end{align*}
	The claim will follow from an upper bound on the numerator and a lower bound on the denominator.
	
	\noindent \textbf{Upper bound on numerator:}
	The following calculation establishes that the magnitude of the  numerator $|(\tau - \kappa_j \kappa_\ell)(1 - \kappa_i) - (\tau - \kappa_i \kappa_\ell)(1 - \kappa_j)|$ is $O(\alpha^2)$.
	\begin{align*}
		|(\tau - \kappa_j \kappa_\ell)(1 - \kappa_i) - (\tau - \kappa_i \kappa_\ell)(1 - \kappa_j)| &= |(\tau - \tau \kappa_i - \kappa_j \kappa_\ell + \kappa_i \kappa_j \kappa_\ell) - (\tau - \tau\kappa_j - \kappa_i \kappa_\ell + \kappa_i \kappa_j \kappa_\ell)| \\
		&= |\tau(\kappa_j - \kappa_i) - \kappa_\ell(\kappa_j - \kappa_i)| = |(\tau - \kappa_\ell)(\kappa_j - \kappa_i)| \\
		&= \tau^2|1 - (1 + \delta_\ell)|\cdot|1 + \delta_j - (1 + \delta_i)| = \tau^2(\delta_\ell)|\delta_j - \delta_i| \\
		&\le \tau^2 \alpha^2
	\end{align*}
	where the last inequality uses $\delta_\ell \leq \alpha$ and $|\delta_j - \delta_i| \leq \alpha$. 
		
	\noindent \textbf{Lower bound on denominator:}
	The denominator $(\tau - \kappa_i \kappa_\ell)(1 - \kappa_j)$ is lower bounded by a constant. 
	\begin{align*}
		(\tau - \kappa_i \kappa_\ell)(1 - \kappa_j) &\geq (\tau - \tau^2(1 + \alpha)^2)(1 - \tau(1 + \alpha)) 
	\end{align*}
	As $\alpha \leq O\left(\frac{\log d}{d}\right)$, both $(\tau - \tau^2(1 + \alpha)^2)$ and $(1 - \tau(1 + \alpha))$ are lower bounded by a constant.
\end{proof}

\begin{claim} \label{claim:a-b-ratio}
Let $\kappa_x = \tau(1 + \delta_x)$, where $\alpha$ is the typicality threshold, and $x \in \{i, j, \ell\}$, and define $A \coloneqq \left(1 - T(\kappa_i, \kappa_\ell)^2\right) \left(1 - T(\kappa_j, \tau)^2\right)$ and $B \coloneqq \left(1 - T(\kappa_j, \kappa_\ell)^2\right) \left(1 - T(\kappa_i, \tau)^2\right)$. Then,
\[
\left| \frac{A}{B} - 1 \right| \leq O(\alpha^2)
\]
\end{claim}
\begin{proof}
	It suffices to lower bound $B$ by a constant, and then prove that $|A - B| \leq O(\alpha^2)$. 
	Lower bounding $B$ by a constant is straightforward; $T(x, y)$ is maximized when $x = y = \tau$, achieving a value of $\frac{\tau}{1 + \tau}$. 
	Thus, $B \geq (1 - \frac{\tau^2}{(1 + \tau)^2})^2$, which is a constant.

	To get a handle on $|A - B|$, we first expand the expressions:
	\begin{align*}
	A - B &= [T(\kappa_i, \tau)^2 - T(\kappa_j, \tau)^2 - T(\kappa_i, \kappa_\ell)^2 + T(\kappa_j, \kappa_\ell)^2] + [T(\kappa_i, \kappa_\ell)^2 T(\kappa_j, \tau)^2 - T(\kappa_j, \kappa_\ell)^2 T(\kappa_i, \tau)^2] \\
	&= \frac{1}{(1 - \tau^2)(1 - \kappa_i^2)(1 - \kappa_j^2)(1 - \kappa_\ell^2)} \cdot [\tau^2 \cdot f_1(\delta_i, \delta_j, \delta_\ell) + \tau^4 \cdot f_2(\delta_i, \delta_j, \delta_\ell)]
	\end{align*}
	where
	\begin{align*}
		f_1(x, y, z) &\coloneqq (1 - \tau(1 + x))(1 - \tau^2(1 + y)^2)(1 - \tau^2(1 + z)^2) - (1 - \tau(1 + y))(1 - \tau^2(1 + x)^2)(1 - \tau^2(1 + z)^2) \\ 
		&\; \; \; \; \; \; \; \; \; \; \; \; - (1 - \tau(1 + x)(1 + z))(1 - \tau^2(1 + y)^2)(1 - \tau^2) + (1 - \tau(1 + y)(1 + z))(1 - \tau^2(1 + x)^2)(1 - \tau^2) \\
		f_2(x, y, z) &= (1 - \tau(1 + x)(1 + z))(1 - \tau(1 + y)) - (1 - \tau(1 + y)(1 + z))(1 - \tau(1 + x)) 
	\end{align*}
	The term $(1 - \tau^2)(1 - \kappa_i^2)(1 - \kappa_j^2)(1 - \kappa_\ell^2)$ is lower bounded by a constant for large enough $d$ and so it suffices to bound $|f_1(\delta_i,\delta_j,\delta_{\ell})|$ and $|f_2(\delta_i,\delta_j,\delta_{\ell})|$ by $O(\alpha^2)$.

	To do so, we use the fact that $f_1$ and $f_2$ is a polynomial of constant degree (in particular, of degree at most $5$) whose degree-$0$ and degree-$1$ terms are $0$, and whose remaining coefficients are bounded by a constant.
	The result then follows since $\delta_i,\delta_j,\delta_k\le\alpha$.

	It is easy to see that their coefficients are bounded by a constant.
	The fact that the degree-$0$ and degree-$1$ terms of $f_k$ vanish for $k=1,2$ follows from the fact that $f_k(0, 0, 0) = 0$ and that the univariate polynomials $f_k(x, 0, 0)$, $f_k(0, y, 0)$, and $f_k(0, 0, z)$ are identically $0$.
\end{proof}

\end{document}